\documentclass{amsart}
\usepackage{amsmath,amsfonts,amssymb,amsthm}
\usepackage{color}
\newcommand{\Deg}{N}

\newcommand{\mysim}{\text{\ ``$\sim$''\ }}
\newcommand{\Narea}{N_{\mathit{area}}}
\newcommand{\carea}{c_{\mathit{area}}}

\renewcommand{\bold}[1]{\medskip \noindent {\bf #1 }\nopagebreak}

\renewcommand{\epsilon}{\varepsilon}
\newcommand{\Vol}{\operatorname{Vol}}

\newcommand{\area}{\operatorname{area}}

\newcommand{\PSL}{\operatorname{PSL}(2,{\R})}
\newcommand{\SL}{\operatorname{SL}(2,{\R})}
\newcommand{\SLZ}{\operatorname{SL}(2,{\Z})}
\newcommand{\GL}{\operatorname{GL}(2,{\R})}

\newcommand\N{\mathbb N}

\newcommand\Z{\mathbb Z}

\newcommand{\T}{\mathbb T}

\newcommand\Q{\mathbb Q}

\newcommand{\R}{{\mathbb R}}
\newcommand{\reals}{{\mathbb R}}

\newcommand\C[1]{{\mathbb C}^{#1}}
\newcommand{\cx}{{\mathbb C}}

\newcommand{\CP}{{\mathbb C}\!\operatorname{P}^1}

\newcommand{\cB}{{\mathcal B}}
\newcommand{\cC}{{\mathcal C}}
\newcommand{\cH}{{\mathcal H}}

\newcommand{\cQ}{{\mathcal Q}}
\newcommand{\cP}{\mathcal P}
\newcommand{\cR}{\mathcal R}

\newcommand{\kk}{\mathbf{k}}
\newcommand{\sss}{\mathbf{s}}
\newcommand{\dd}{\mathbf{d}}
\newcommand{\one}{\mathbf{1}}
\newcommand{\ee}{\mathbf{e}}
\newcommand{\nn}{\mathbf{n}}
\newcommand{\bbb}{\mathbf{b}}

\newlength{\halfbls}

\def\@ifundefined#1#2#3%
  {\expandafter\ifx\csname#1\endcsname\relax#2\else#3\fi}

\@ifundefined{theoremstyle}{
}{
\theoremstyle{plain} %default
}
\newtheorem{theorem}{Theorem}[section]
\newtheorem*{NoNumberTheorem}{Theorem}
\newtheorem{prop}[theorem]{Proposition}
\newtheorem{proposition}[theorem]{Proposition}
\newtheorem*{NNProposition}{Proposition}
\newtheorem{lemma}[theorem]{Lemma}
\newtheorem{cor}[theorem]{Corollary}
\newtheorem{corollary}[theorem]{Corollary}

\@ifundefined{theoremstyle}{
}{
\theoremstyle{definition} %default
}
\newtheorem{definition}[theorem]{Definition}

\newtheorem{Convention}[theorem]{Convention}

\newtheorem{remark}[theorem]{Remark}

\numberwithin{equation}{section}

\newlength{\figboxwidth}
\newcommand{\appendixmode}{
        \setcounter{section}{0}
        \renewcommand{\thesection}{\Alph{section}}
}

\newcommand{\into}{\hookrightarrow}

\begin{document}

\title[Right-angled billiards and volumes of moduli spaces]
{Right-angled billiards and volumes of moduli spaces of quadratic
differentials on $\CP$}

\date{May 17, 2015}

\begin{abstract}

We  use the relation between the volumes of the strata of meromorphic
quadratic  differentials  with  at  most  simple  poles  on $\CP$ and
counting  functions  of  the number of (bands of)
simple closed geodesics in
associated  flat  metrics with singularities to prove a very explicit
formula   for   the  volume  of  each  such  stratum  conjectured  by
M.~Kontsevich a decade ago.

Applying  ergodic  techniques  to  the Teichm\"uller geodesic flow we
obtain
quadratic asymptotics for the number of (bands of) closed
trajectories  and  for  the number of generalized diagonals in almost
all right-angled billiards.
\end{abstract}

% First author
\author{Jayadev~S.~Athreya}
\address{
Department of Mathematics,
University of Illinois,
Urbana, IL~61801 USA}
\email{jathreya@illinois.edu}

% Second author
\author{Alex Eskin}
\address{
Department of Mathematics,
University of Chicago,
Chicago, Illinois 60637, USA
}
\email{eskin@math.uchicago.edu}

% Third author
\author{Anton Zorich}
\address{
Institut Universitaire de France;
Institut de Math\'matiques de Jussieu --
Paris Rive Gauche,
UMR7586,
B\^atiment Sophie Germain,
Case 7012,
75205 PARIS Cedex 13, France}
\email{anton.zorich@imj-prg.fr}

\makeatletter
\let\@wraptoccontribs\wraptoccontribs
\makeatother

% Contributors
\contrib[With an Appendix by]{Jon Chaika}
\address{Department of Mathematics,
University of Utah,
Salt Lake City, Utah~84112, USA.
}
\email{chaika@math.utah.edu}

\thanks{
J.S.A. is partially supported by NSF grants DMS 1069153, NSF grants DMS 1107452, 1107263, 1107367 RNMS: GEometric structures And Representation varieties (the GEAR Network), and NSF CAREER grant DMS 1351853.
A.E. is partially supported by NSF grants DMS 0244542 and DMS 0604251.
A.Z. is partially supported by IUF and by ANR ``GeoDyM''
}

\maketitle

\setcounter{tocdepth}{1} % {2}
\tableofcontents

%####################################################################
%####################################################################
%####################################################################

\section{Introduction}
\label{intro} \noindent Motivated by the study of computing
asymptotics for the number of generalized diagonals and for the
number of closed billiard trajectories  in right-angled polygons, we
were naturally led to questions on Masur--Veech volumes of strata of
moduli spaces of quadratic differentials on $\CP$. Our main result,
explicitly computing these volumes, resolves a conjecture of
M.~Kontsevich.
%--------------------------------------------------------------------
\subsection{Volumes of moduli spaces of quadratic differentials}
\label{sec:subsec:intro:volumes}

\begin{theorem}[Kontsevich Conjecture]
\label{theorem:volume}
The volume of any stratum $\cQ_1(d_1, \dots, d_k)$ of meromorphic quadratic
differentials with at most simple poles on $\CP$ (i.e. $d_i\in \{-1\,;\,0\}\cup\N$
for $i=1,\dots,k$, and $\sum_{i=1}^k d_i = -4$) is equal to
\begin{equation}
\label{eq:volume}
\Vol\cQ_1(d_1, \dots, d_k) =  2\pi^2\cdot \prod_{i=1}^k v(d_i)
\end{equation}
(where all the zeroes and poles are ``named''.)
\end{theorem}

\noindent Here, the function $v$ is defined on integers $n$ greater than or equal to $-1$ by
\begin{equation}
\label{eq:v}
v(n):=\cfrac{n!!}{(n+1)!!} \cdot \pi^n \cdot
\begin{cases}
\pi& \text{ when $n$ is odd}\\
2  & \text{ when $n$ is even}
\end{cases}
\end{equation}
for     $n=-1,0,1,2,3\dots$, and   the     double    factorial
$n!!=n\cdot(n-2)\cdot\dots$  is the product of all even (respectively
odd) positive integers smaller than or equal to $n$. By convention we
set
$$
(-1)!! = 0!!=1 \ ,
$$
which implies that
$$
v(-1)=1 \qquad\text{and}\qquad v(0)=2.
$$

This  formula  for  the  volume (up to some normalization factor) was
conjectured  by M.~Kontsevich about ten years ago. It is much simpler
than   the   formula  for  the  volumes  of  the  strata  of  Abelian
differentials          found          by         A.~Eskin         and
\mbox{A.~Okounkov}~\cite{Eskin:Okounkov}.

When  this paper was written, there was not a single stratum of quadratic
differentials for which the explicit volume was known,  though  an  algorithm  of  computation was presented
in~\cite{EO2}.  In addition to this work,  there  is  some  very  recent  progress in
evaluation  of  volumes of low-dimensional strata in genera different
from   $0$.   Rigorous  formal  methods  used  in~\cite{Goujard}  (in
particular, implementation of the algorithm~\cite{EO2}) are confirmed
by   independent  numerical  experiments~\cite{experimental:volumes}.
However,   any   known   approach  involves significant computer-assisted
computations,  and  is  limited  to volumes of strata of sufficiently
small dimension, while Theorem~\ref{theorem:volume} provides a simple
formula for \textit{all} strata in genus $0$.

Returning to our original motivation, we obtain as an important
application    of Theorem~\ref{theorem:volume}  asymptotics for the
number of closed  trajectories  and for the  number  of  generalized
diagonals in right-angled                       polygons
(see \S\ref{sec:subsec:introduction:beginning}  below).  This choice
is particularly natural in the context of this paper since we have to
solve an  analogous  problem  for quadratic differentials and to
compute the  corresponding  Siegel--Veech constants $c_\cC$ for the
strata of quadratic  differentials  in  genus  $0$ anyway: it makes
part of the proof  of Theorem~\ref{theorem:volume}. This theorem also
immediately provides    asymptotics    for    certain  Hurwitz
numbers,   see \S\ref{sec:subsec:counting:pillowcase:covers}.
Another example of  applications  is  discussed
in~\cite{Delecroix:Zorich} where the values of volumes and the
related Siegel--Veech constants are used to compute  Lyapunov
exponents  of  the Hodge bundle over hyperelliptic loci  in  the
strata  of  quadratic differentials and to compute the diffusion
rate  for  interesting  families  of generalized wind-tree
billiards~\cite{Delecroix:Hubert:Lelievre}.

\subsection*{Strategy of the proof.}
We   start   by   solving   the   counting   problems  for  quadratic
differentials.  The  Siegel--Veech  constant $\carea$ responsible for
the  exact  quadratic  asymptotics of the weighted number of bands of
regular closed geodesics on almost any flat sphere in a given stratum
$\cQ(d_1,\dots,d_n)$  of meromorphic quadratic differentials with
at      most      simple      poles   on $\CP$  was     recently     computed
in~\cite{Eskin:Kontsevich:Zorich},
\begin{equation}
\label{eq:carea:answer}
\carea(\cQ(d_1,\dots,d_n))=
-\cfrac{1}{8\pi^2}\,\sum_{j=1}^n \cfrac{d_j(d_j+4)}{d_j+2}\,.
\end{equation}
Developing techniques elaborated in~\cite{Eskin:Masur:Zorich} for the
strata  of  Abelian  differentials  and  using  the  further  results
from~\cite{Boissy:configurations}   and~\cite{Masur:Zorich}   on  the
\textit{principal  boundary} of the strata of quadratic differentials
we  express the Siegel--Veech constant $\carea$
in genus $0$
in  terms  of  the  ratio  of  the  volumes  of  appropriate  strata,
\begin{equation}
\label{eq:carea:in:terms:of:volumes}
\carea(\cQ(d_1,\dots,d_n))=
\frac{\text{Explicit polynomial in volumes of simpler strata}}
{\Vol(\cQ(d_1,\dots,d_n))}\,.
\end{equation}
In  this  way  we  obtain a series of identities on the volumes of the
strata  of  meromorphic  quadratic  differentials with at most simple
poles  in  genus zero. The resulting identities recursively determine
the  volumes of all strata.
   %  This allows us to find the volumes of all
   %  strata  of  meromorphic  quadratic  differentials with at most simple
   %  poles in genus $0$.
The     proof     of     Theorem~\ref{theorem:volume},    given    in
\S\ref{sec:induction},     consists    in    verifying    that    the
expression~\eqref{eq:volume}    for    the   volume   satisfies   the
combinatorial     identities    implied    by~\eqref{eq:carea:answer}
and~\eqref{eq:carea:in:terms:of:volumes}.  Part  of this verification
is performed in Appendix~\ref{sec:appendix:identity}.

\begin{remark}[Normalization conventions]
Note  that  the  convention  that  all zeroes and poles are ``named''
affects   the   normalization:   we   compute   the  volumes  of  the
corresponding  covers  over  strata with ``anonymous'' singularities.
For  example,  the  stratum  $\cQ(1,-1^5)$ of quadratic differentials
with  ``anonymous''  zeroes  and  poles  is isomorphic to the stratum
$\cH(2)$  of  holomorphic  Abelian  differentials;  by convention the
volume  elements  are  chosen to be invariant under this isomorphism.
However, by~\eqref{eq:volume} we have
$$
\Vol\cQ_1(1, -1^5) =  2\pi^2\cdot v(1)\cdot\left(v(-1)\right)^5=
2\pi^2\cdot\frac{\pi^2}{2}\cdot 1^5=
5!\cdot\frac{\pi^4}{120}=5!\cdot\Vol\cH_1(2)\,,
$$
which  corresponds to $5!$ ways to \textit{give names} to five simple
poles.

Similarly,
$$
\Vol\cQ_1(2, -1^6) =
2\pi^2\cdot v(2)\cdot\left(v(-1)\right)^6=
2\pi^2\cdot\frac{4\pi^2}{3}\cdot 1^6=
\frac{6!}{2!}\cdot\frac{\pi^4}{135}=\frac{6!}{2!}\cdot\Vol\cH_1(1,1)\,.
$$
This  time  there  is  an extra factor $\frac{1}{2!}$ responsible for
\textit{forgetting the names} of the two zeroes of $\cH(1,1)$.
\end{remark}

%--------------------------------------------------------------------
\subsection{Counting pillowcase covers}
\label{sec:subsec:counting:pillowcase:covers}

One  of  the  ways to compute the volumes of the strata of Abelian or
quadratic  differentials  (actually,  the only one before the current
paper)    is    to    count    \textit{square-tiled    surfaces}   or
\textit{pillowcase  covers},  see~\cite{Eskin:Okounkov},  \cite{EO2},
\cite{Eskin:Okounkov:Pandharipande}, \cite{Zorich:volumes}.
  In  the current paper we follow an alternative method,
and,  thus, our result implies an explicit expression for the leading
term  of the function counting associated pillowcase covers, when the
degree of the cover tends to infinity.

Namely, following~\cite{EO2} we define a \textit{pillowcase
cover} of degree
$4d$
as a ramified cover
\begin{equation}
\label{eq:Eskin:Okounkov:pillowcase:cover}
\pi: \hat\cP \rightarrow \cP
\end{equation}
over  the pillowcase orbifold $\cP=\big(\C{}/(\Z\oplus i\Z)\big)/\pm$
(as  in  Figure~\ref{fig:square:pillow}) with ramification data given
as  follows.  Let  $\eta$  be a partition and $\nu$ a partition of an
even  number  into \textit{odd} parts. Viewed as a map to the sphere,
$\pi$  has  profile $(\nu,2^{2d-|\nu|/2})$ over $0\in\cP$ and profile
$(2^{2d})$ over the other three corners of $\cP$. Additionally, $\pi$
has  profile  $(\eta_i,1^{4d-\eta_i})$ over $\ell(\eta)$ given points
of  $\cP$  and unramified elsewhere, where $\ell(\eta)$ is the number
of  parts  in $\eta$. This ramification data determines the genus $g$
of $\hat\cP$ by
$$
2-2g=\chi(\hat\cP) = \ell(\eta)+\ell(\nu) - |\eta|-|\nu|/2\,.
$$
We  consider only those ramification data for which $g=g(\hat\cP)$ in
the    above    formula    is    equal    to    zero,
\begin{equation}
\label{eq:condition:genus:0}
\ell(\eta)+\ell(\nu) - |\eta|-|\nu|/2=2\,.
\end{equation}
\begin{figure}[htb]
   %
   %  PILLOWCASE
   %
\includegraphics{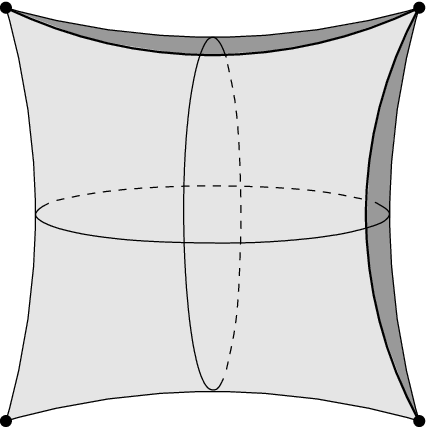}
\vspace{55pt}
\caption{
\label{fig:square:pillow}
Pillowcase orbifold.}
\end{figure}
Denote   by   $\operatorname{Cov}^0_{4d}(\eta,\nu)$   the  number  of
inequivalent  degree  $4d$  connected  covers  $\pi:\hat\cP \to \cP$
with ramification data $(\eta,\nu)$.

Denote by $\cQ(\eta, \nu)$ the moduli
space  of quadratic differentials with singularity data $\{\nu_i-2\}$
and      $\{2\eta_i-2\}$.      Condition~\eqref{eq:condition:genus:0}
guarantees  that  $\cQ(\eta,  \nu)$  is  nonempty, and corresponds to
genus zero.

Consider now the same partitions $\eta,\nu$ as above
and a ramified cover
$$
\pi_{\scriptscriptstyle \boxplus}: \hat\cP \rightarrow \cP
$$
of the same degree $4d$
over  the  pillowcase  orbifold  $\cP$
with  ramification  data  given  as
follows:     $\pi_{\scriptscriptstyle    \boxplus}$    has    profile
$(2\eta,\nu,2^{2d-|\eta|-|\nu|/2})$   over   $0\in\cP$   and  profile
$(2^{2d})$   over  the  other  three  corners  of  $\cP$.  The  cover
$\pi_{\scriptscriptstyle \boxplus}$ is unramified elsewhere. Applying the
Riemann--Hurwitz             formula             and            using
relation~\eqref{eq:condition:genus:0}  we  see  that covers with such
ramification  profile again have genus zero. The  corresponding  flat  surface  belongs  to the same
stratum  $\cQ(\eta,\nu)$ as before.
Denote          by         $\operatorname{Cov}^{0,{\scriptscriptstyle
\boxplus}}_{4d}(\eta,\nu)$  the  number  of  inequivalent degree $4d$
connected  covers $\pi_{\scriptscriptstyle \boxplus}:\hat\cP \to \cP$
with   ramification   data   $(\eta,\nu)$   as   above.

Theorem~\ref{theorem:volume} and the
Theorem~\ref{th:pillowcases:as:volume} below provide very simple
asymptotic formulae for the Hurwitz numbers
$\operatorname{Cov}^0_{4d}(\eta,\nu)$ and
$\operatorname{Cov}^{0,{\scriptscriptstyle \boxplus}}_{4d}(\eta,\nu)$.

\begin{theorem}
\label{th:pillowcases:as:volume}
For     any    ramification    data    $(\eta,    \nu)$    satisfying
condition~\eqref{eq:condition:genus:0}           the          numbers
$\operatorname{Cov}^0_{4d}(\eta,\nu)$
and

$\operatorname{Cov}^{0,{\scriptscriptstyle \boxplus}}_{4d}(\eta,\nu)$
of  pillowcase  covers of type
$(\eta,\nu)$ admit the following limits:
\begin{align}
\label{eq:volume:through:covers}
\lim_{\Deg \to \infty}
\frac{1}{\Deg^{\ell(\eta)+\ell(\nu)-2}}
\sum_{d=1}^{\Deg} \operatorname{Cov}^0_{4d}(\eta,\nu)
&=
2^{\ell(\eta)}\cdot
\frac{\Vol\cQ_1(\eta,\nu)}{2(\ell(\eta)+\ell(\nu)-2)}
\,,
\\
\lim_{\Deg \to \infty}
\frac{1}{\Deg^{\ell(\eta)+\ell(\nu)-2}}
\sum_{d=1}^{\Deg}
\operatorname{Cov}^{0,{\scriptscriptstyle \boxplus}}_{4d}(\eta,\nu)
&
=
\frac{\Vol\cQ_1(\eta,\nu)}{2(\ell(\eta)+\ell(\nu)-2)}
\,,
\end{align}
where $\Vol\cQ_1(\eta,\nu)$ is given by equation~\eqref{eq:volume}.
\end{theorem}

Theorem~\ref{th:pillowcases:as:volume}       is       proved
in~\S\ref{subsec:lattice:square:tiled:pillowcase}.

Note  that the more natural direct geometric approach to the counting
of pillowcase covers leads to rather involved combinatorial problems.
We   present  this  alternative  geometric  approach  in  a  separate
paper~\cite{AEZ:pillowcase:covers}.

\begin{remark}
There  are  several  different  combinatorial approaches to computing
volumes of strata, based on counting (pillowcase) covers.

For  the  strata  of  Abelian differentials, the problem is solved in
\cite{Eskin:Okounkov};  see  also~\cite{Zorich:volumes}  for  a  more
direct  but much less efficient approach. Many of these combinatorial
approaches  can be pushed to produce some complicated expressions for
the  volumes  in  Theorem~\ref{theorem:volume}.  Currently,  the most
efficient  approach  to calculation of volumes of strata of quadratic
differentials  (independently  of  genus) is suggested in~\cite{EO2}.
The  exact  values  of volumes of all strata up to dimension $11$ are
presented  in~\cite{Goujard}  based  on  the algorithm of~\cite{EO2};
this result is close to limits of current computational capacities of
modern  computers  in  manipulating huge tables of characters. For an
approach    based    on    Kontsevich'   solution   to   the   Witten
conjecture~\cite{Kontsevich}   see~\cite{AEZ:pillowcase:covers};  one
more  version  developing  ideas  of  Eskin and Okounkov is suggested
in~\cite{Rios-Zertushe:2012};   see  also~\cite{experimental:volumes}
for  yet another approach. Paper~\cite{Goujard} suggests a comparison
of various approaches.

However, we were not able to get the
simple expressions (\ref{eq:volume}) using any of these methods. In
fact, our proof of Theorem~\ref{theorem:volume} is not purely
combinatorial, but has analytic,
geometrical and dynamical inputs (and is motivated by consideration of
Lyapunov exponents). It thus remains a challenge to give a more direct
proof of Theorem~\ref{theorem:volume}, in particular bypassing
\cite{Eskin:Kontsevich:Zorich}.
\end{remark}

\subsection{Counting trajectories of right-angled billiards}
\label{sec:subsec:introduction:beginning}
Currently  it  is  not  known  whether  there  exists a single closed
billiard trajectory in every obtuse triangle (see~\cite{Schwartz} for
some  progress  in  this  direction  and for further references). The
situation  with  billiards  in \textit{rational} polygons (that is in
polygons  with  angles  which  are  rational  multiples  of $\pi$) is
understood much better: trajectories of such billiards are related to
geometry of certain compact flat surfaces with conical singularities,
which  are  thoroughly  studied  starting with the landmark papers of
H.~Masur~\cite{Masur:interval}  and  W.~Veech~\cite{Veech:Gauss}.  In
particular,  it  is known that a billiard in any rational polygon has
infinitely  many  closed trajectories~\cite{KMS}, and furthermore the
number  of trajectories of length at most $L$ is bounded between $c_1
L^2$  and  $c_2 L^2$ for some $0 < c_1 < c_2$  and for
$L$ large enough, see \cite{Masur:lower} and \cite{Masur:upper}.

In the current paper we study families of right-angled billiards like
the  ones in Figures~\ref{fig:non:embeddable:polygons} and
\ref{fig:family:of:rectangular:polygons}. Namely,
we  assume that the billiard table is a topological disk endowed with
a  flat  metric,  and  that  the  boundary  of  the disk is piecewise
geodesic  such  that  the angle at every corner of the boundary is an
integer multiple of $\frac{\pi}{2}$.
Note  that  by allowing integer multiples $k\pi/2$ with $k \geq 5$, we can obtain billiard tables which may
not be       embeddable         in       the         plane
(see~Figure~\ref{fig:non:embeddable:polygons}). In particular, we can
consider helical right-angled billiards.

\begin{figure}[htb]
\centering
%
%  EXAMPLES OF NON-EMBEDDABLE RIGHT-ANGLED POLYGONS
   %
\includegraphics{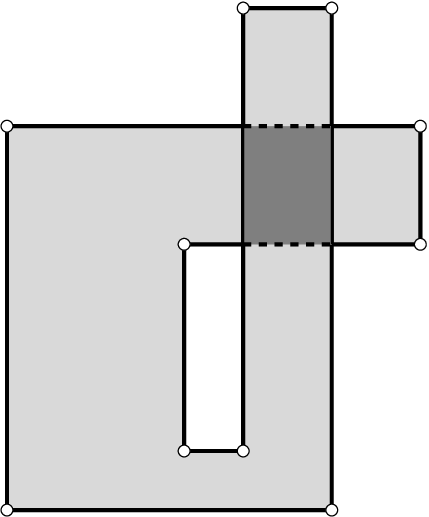}
\includegraphics{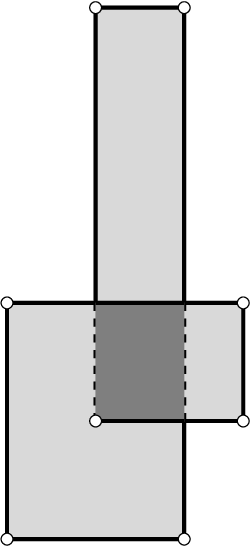}
\vspace{95bp}
\caption{
\label{fig:non:embeddable:polygons}
A Right-angled billiard  table  which is not
embeddable   into   the   plane.
   }
\end{figure}

We  consider  families  of  polygons  sharing  the same collection of
interior         corner        angles        $\left(\frac{\pi}{2}k_1,
\frac{\pi}{2}k_2,\ldots, \frac{\pi}{2}k_n\right) $. Actually, it will
be  convenient  to  consider a slightly larger space $\cB(k_1,\ldots,
k_n)$  of  ``directional  billiards'' distinguishing a billiard table
$\Pi$  and  the same table turned by angle $\phi$. The measure in the
space $\cB(k_1,\dots,k_n)$ is the product measure of Lebesgue measure
arising from the side lengths and the angular measure $d\phi$.

We  count  the  number  of generalized diagonals of bounded length in
such  billiards (that is, the number of trajectories of bounded length
which  start  in  some  fixed  corner  $P_i$ and arrive to some fixed
corner  $P_j$,  see  Figure~\ref{fig:family:of:rectangular:polygons})
and  the  number  of  closed billiard trajectories of bounded length.
Note, that closed regular trajectories are never isolated in rational
billiards: they always form bands of ``parallel'' closed trajectories
of             the             same            length,            see
Figure~\ref{fig:family:of:rectangular:polygons}.  Thus, when counting
closed  trajectories  one  actually  counts the number of such bands.
Sometimes,  it  is  natural  to  count  the bands with a weight which
registers     the     thickness     of    the    band,    see    e.g.
Theorem~\ref{theorem:Narea:weak:asymp}      at     the     end     of
\S\ref{sec:subsec:introduction:beginning}.
By  convention  we  always  count  \textit{non-oriented} generalized
diagonals and \textit{non-oriented} closed billiard trajectories.

To  give  an  idea  of  the  general  theorems  stated  in  detail  in
\S\ref{sec:configurations:counting}       and       developed      in
\S\ref{sec:Siegel:Veech}, we present the following  representative
results.

\begin{theorem}
\label{theorem:pi:2:general}
For  any  right-angled  billiard  $\Pi$  outside  of a zero
measure  set  in  any  family  $\cB(k_1,  \ldots,  k_n)$  the  number
$N_{ij}(\Pi,L)$  of  generalized  diagonals  of  length  at  most $L$
joining   a   pair   of   fixed   corners   $P_i,  P_j$  with  angles
$\frac{\pi}{2}$  has  the  following  quadratic  asymptotics as
$L\to\infty$:
\begin{equation}
\label{eq:theorem:pi:2:general}
N_{ij}(\Pi,L) \sim \cfrac{1}{2\pi}\cdot
\cfrac{L^2}{\text{Area of the billiard table}}\ .
\end{equation}
\end{theorem}

Theorem~\ref{theorem:pi:2:general}         is        proved        in
~\S\ref{sec:subsec:sv:pocket}, using the theorem
  proved by Jon Chaika in Appendix~\ref{sec:chaika}.

\begin{figure}[htb]
\centering
%
%  FAMILY OF RECTANGULAR POLYGONS
   %
\includegraphics{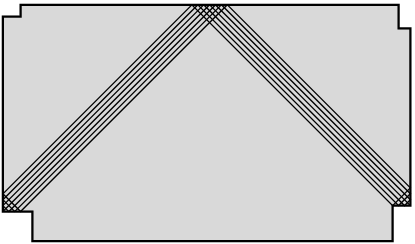}
\includegraphics{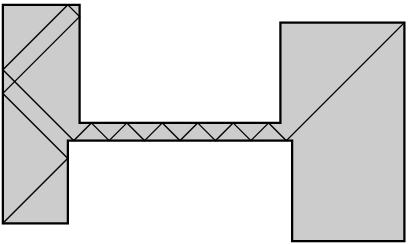}
\begin{picture}(0,0)(0,13)
\put(-171,-62){$P_i$}
\put(-27,-61){$P_j$}
\put(18,-66){$P_i$}
\put(164,11){$P_j$}
\end{picture}
\vspace{75bp}
\caption{
\label{fig:family:of:rectangular:polygons}
A  family  $\cB(k_1,\dots,k_n)$  of  right-angled polygons; a band of
periodic  trajectories on the left, and a generalized diagonal on the
right.
   }
\end{figure}
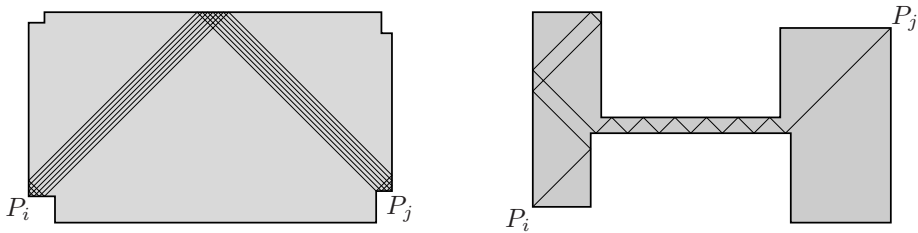

The fact that this asymptotics does not depend at all on the billiard
table  is  at  the  first glance counterintuitive. What is even more
surprising  is  that  it  is universal: it is the same not
only  for  almost all billiard tables inside each family, but it does
not  vary  even from one family to another! In particular, though the
shape    of    the    two    polygons    of    the   same   area   in
Figure~\ref{fig:family:of:rectangular:polygons}  is  quite different,
the  number  of  trajectories  of  length  at  most  $L$  joining the
right-angle   corner   $P_i$  to  the  right-angle  corner  $P_j$  is
approximately  the  same in both cases, and is approximately the same
as  the  number  of  trajectories  of  length at most $L$ joining two
corners of the usual rectangular billiard of the same area when $L\gg
1$.

The   situation  becomes  more  complicated  when  we  consider  other
types of corners  of  the  billiard.  Consider,  for  example, an
$\operatorname{L}$-shaped       billiard       table       as      on
Figure~\ref{fig:L:shaped:billiard}.
Let   $P_1,\dots,   P_5$   be   the   right-angle   corners   of  the
$\operatorname{L}$-shaped  billiard, and let $P_0$ be the corner with
the interior angle $\frac{3\pi}{2}$.

\begin{figure}[htb]
%\centering
   %
\includegraphics{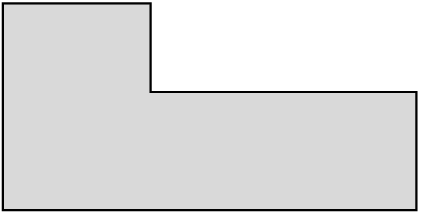}
\begin{picture}(0,0)(0,0)
\put(-64,-65){$P_3$}
\put(-64,-4){$P_4$}
\put(-16,-4){$P_5$}
\put(-16,-25){$P_0$}
\put(46,-25){$P_1$}
\put(46,-65){$P_2$}
\end{picture}
\vspace{60bp}
\caption{
\label{fig:L:shaped:billiard}
$\operatorname{L}$-shaped billiard.
   }
\end{figure}

\medskip
\begin{theorem}
\label{theorem:L:shaped:almost:all}
For almost any $\operatorname{L}$-shaped billiard $\Pi$ the number
$N_{i0}(\Pi,L)$  of  generalized  diagonals  of  length  at  most $L$
joining  a  fixed  corner  $P_i$  with  angle $\frac{\pi}{2}$ and the
corner  $P_0$  with  angle  $\frac{3\pi}{2}$  has  the following
quadratic asymptotics as $L\to\infty$:
\begin{equation}
\label{eq:theorem:L:shaped:almost:all}
N_{i0}(\Pi,L) \sim \cfrac{2}{\pi}\cdot
\cfrac{L^2}{\text{Area of the billiard table}}\ .
\end{equation}
\end{theorem}
The proof of this theorem also relies in part on
  Theorem~\ref{theorem:birkhoff:chaika} proved by Jon~Chaika in
  Appendix~\ref{sec:chaika}.

The  naive intuition does not help: the angle $\frac{3\pi}{2}$ at the
corner  $P_0$  is  \textit{three}  times  larger than in the previous
case,  while  the  constant  in  the  asymptotics  for  the number of
generalized  diagonals  is  \textit{four}  times  larger  than in the
previous  statement.  Currently  we  have  no idea how to obtain this
factor  $4$  without  using  techniques of the Teichm\"uller geodesic
flow,  Lyapunov  exponents  of  the  Hodge bundle, and the computation of
volumes  of  the moduli spaces of meromorphic quadratic differentials
with  at  most  simple poles on $\CP$.
Theorem~\ref{theorem:L:shaped:almost:all}      is      proved      in
\S\ref{sec:subsec:sv:typeI}.

Using  recently  developed  technology, one can prove weak asymptotic
formulas   similar to
Theorem~\ref{theorem:pi:2:general}      and
Theorem~\ref{theorem:L:shaped:almost:all}   for  individual  billiard
tables. In particular, the following holds:
\begin{theorem}
\label{theorem:L:shaped:precise}
Suppose $\Pi$ is an $\operatorname{L}$-shaped billiard table as in
Figure~\ref{fig:L:shaped:billiard}. Let
\begin{displaymath}
a = \frac{|P_3 P_4|}{|P_1 P_2|}, \qquad
b = \frac{|P_2 P_3|}{|P_4 P_5|}\,.
\end{displaymath}
Then,
\begin{itemize}
\item[{\rm (i)}] If $a$ and $b$ are both rational, or if there exists
a  non-square integer $D > 0$ such that $a,b \in \Q(\sqrt{D})$ and $a
+ \bar{b} = 1$ (where $\bar{b}$ is the Galois conjugate of $b$), then
\begin{equation}
\label{eq:Bainbridge}
N_{ij}(\Pi,L) \sim c_{ij} \cfrac{L^2}{\text{Area of the billiard table}}\,,
\end{equation}
\item[{\rm (ii)}] For any other $\operatorname{L}$-shaped billiard
  table, we have the ``weak asymptotic formulas''
\begin{displaymath}
N_{ij}(\Pi,L) \mysim \cfrac{1}{2\pi}\cdot
\cfrac{L^2}{\text{Area of the billiard table}}\ .
\end{displaymath}
and
\begin{displaymath}
N_{i0}(\Pi,L) \mysim \cfrac{2}{\pi}\cdot
\cfrac{L^2}{\text{Area of the billiard table}}\ .
\end{displaymath}
The meaning of the ``weak asymptotic $\mysim$ is defined in
\S\ref{sec:subsec:counting:theorems}.
\end{itemize}
\end{theorem}
In  the  case  (i)  the Siegel--Veech constants $c_{ij}$ for rational
values  of  parameters  $a,b$  can  be computed by the formula due to
E.~Gutkin  and  C.~Judge~\cite{Gutkin:Judge}.  For  $i,j\neq  0$  and
$a,b\in\Q(\sqrt{D})$   the   constants   $c_{ij}$   are  computed  by
M.~Bainbridge, see \cite[Theorem 1.5 and \S{14}]{Bainbridge:Euler}.
\begin{proof}
Theorem~\ref{theorem:L:shaped:precise}  is  a  compilation of several
different results. In case (i), the polygon $\Pi$ is a Veech polygon,
which gives rise to a Teichm\"uller curve,
see~\cite{Calta}, \cite{McMullen:Billiards}.
The  existence of an asymptotic formula such as (\ref{eq:Bainbridge})
for  such  a  situation was proved in the pioneering work of W.~Veech
\cite{Veech:surfaces}.

Let
\begin{displaymath}
U  = \begin{pmatrix} 1 & \ast
  \\ 0 & 1 \end{pmatrix} \subset \SL, \qquad
P = \begin{pmatrix} \ast & \ast \\ 0
  & \ast \end{pmatrix} \subset \SL.
\end{displaymath}
The fact that weak asymptotic formulas such as those of part (ii) hold
for       any       rational       billiard       table       follows
from~\cite[Theorem~2.12]{Eskin:Mirzakhani:Mohammadi},  which uses the
general       invariant      measure      classification      theorem
of~\cite{Eskin:Mirzakhani}  for  the  action  of $P$ on moduli space.
However,    to    evaluate    the    constant    for   an   arbitrary
$\operatorname{L}$-shaped  table,  one  also  has  to  appeal  to the
explicit classification of $\SL$-invariant affine submanifolds in the
moduli   space   of   Abelian  differentials  in  genus  $2$  due  to
\mbox{C.~McMullen}, \cite{McMullen:SL2R}.
\end{proof}

We  note  that  asymptotic counting formulas for individual billiards
are  associated with invariant measure classification theorems on the
action  of  subgroups  of  $\SL$  on  (certain subsets of) the moduli
space.  In  particular, when a measure classification theorem for the
action   of   the  subgroup  $U$  exists  (e.g.  in  the  case  of  a
Teichm\"uller  curve), one can get a strong asymptotic formula. Also,
a  measure  classification theorem for the action of the subgroup $P$
leads  to  a  weak  asymptotic  formula.

For  other  examples when a
classification  of invariant measures for the action of $U$ (and thus
strong    asymptotic    formulas)    are    known   see   \cite{EMS},
\cite{Eskin:Marklof:Morris},                    \cite{Calta:Wortman},
\cite{Bainbridge:L:shaped}.
All  examples  of  individual
billiard  tables  for  which  the  (strong) quadratic
asymptotics  was  known are, essentially, covered by several families
of  triangles depending on one integer parameter; by several sporadic
triangles beyond these families; by a square with a specially located
barrier;  and by a family of $\operatorname{L}$-shaped tables with or
without   a   wall   for   special   values   of  parameters  of  the
$\operatorname{L}$-shaped table.

In~\S\ref{sec:subsec:counting:theorems}      for      each     family
$\cB_1(k_1,\ldots,k_n)$  of  right-angled  billiards  we describe all
geometric  types  of  generalized  diagonals  and all closed billiard
trajectories  which  can  be  found on a billiard $\Pi$ outside of a
zero  measure set in $\cB_1(k_1,\ldots,k_n)$. For such $\Pi$, and
each such geometric
type  we  prove  (strong)  quadratic  asymptotics  for  the number of
associated
generalized  diagonals (or of bands of closed billiard trajectories),
and explicitly evaluate the constant in the quadratic asymptotics.

  %
  %  For the best currently known counting statements for other types of
  %  right-angled billiard tables,
  %  see \S\ref{sec:subsec:counting:theorems}.
  %
% In this paper, our focus is on the evaluation of
% the constants rather than on the proof of various forms of asymptotics
% for individual billiards.
% However, in~\S\ref{sec:billiards:qd} and~\S\ref{sec:ergodic}
% we give fairly self-contained proofs of weak asymptotics for almost all
% right-angled billiards in a family such as
% Theorem~\ref{theorem:pi:2:general} and
% Theorem~\ref{theorem:L:shaped:almost:all}. See the end of
% \S\ref{sec:subsec:intro:rab:qd} for a discussion of this argument.

Theorem~\ref{theorem:labyrinth}
below   illustrates   an   application   of   the general
Theorem~\ref{theorem:invididual:billiards}     and    of  the  general
Theorems~\ref{th:SV:I}--\ref{th:SV:IV}  to billiards more complicated
than the $\operatorname{L}$-shaped ones, see Figure~\ref{fig:B44:33}.

\begin{figure}[htb]
%
%  Labyrinth
   %
\includegraphics{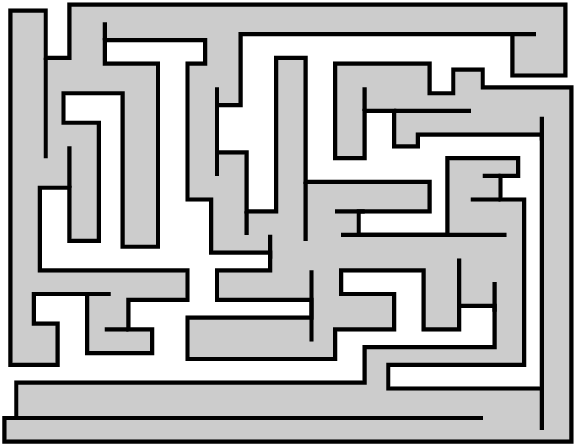}
\vspace{80bp}
\caption{
\label{fig:B44:33}
A billiard table from the family $\cB(4^m,3^n,1^{m+n+4})$.
   }
\end{figure}

By  $\cB(4^m,3^n,1^{m+n+4})$  we  denote  the  family of right-angled
billiards  with  $m$  corners  with  angles  $2\pi$ (endpoints of the
walls);  $n$  corners  with  interior  angles  $3\pi/2$, and with the
remaining $n+m+4$ corners with interior angles $\pi/2$.

\begin{theorem}
\label{theorem:labyrinth}
Consider  two  distinct  corners $P_i,P_j$ of a
billiard $\Pi$ in any family $\cB(4^m,3^n,1^{m+n+4})$. Assume that at
least  one  of  the  interior  angles  $k_i\pi/2$  and  $k_j\pi/2$ is
different  from  $\pi/2$ (i.e. $k_i,k_2$ are not simultaneously equal
to $1$).

For almost any $\Pi$, any  generalized  diagonal  $\delta$  joining  $P_i$ to $P_j$ and non
parallel   to  a  side  of  $\Pi$  never  bounds  a  band  of  closed
trajectories.  No  other  generalized diagonal in $\Pi$ has a segment
parallel  to  any  segment of $\delta$. For almost any $\Pi$, the
number $N_{ij}(\Pi,L)$ of
such  generalized  diagonals  of length at most $L$ has the following
asymptotics as $L\to+\infty$:
$$
N_{ij}(\Pi,L) \sim c_{ij}
\cdot
\cfrac{L^2}{\text{Area of the billiard table}}\,,
$$
where the constant $c_{ij}$ depends only on the angles $k_i\pi/2$ and
$k_j\pi/2$ at $P_i$ and $P_j$ correspondingly; its value is presented
in the following table:
$$
\begin{array}{r|lll}
\text{angle}   &\cfrac{4\pi}{2} &\cfrac{3\pi}{2}     &\cfrac{\pi}{2}\\
[-\halfbls]\\
\hline
\cfrac{4\pi}{2}&\cfrac{9}{10}   &\cfrac{45}{64}      &\cfrac{9}{32}\\
[-\halfbls]\\
\cfrac{3\pi}{2}&\cfrac{45}{64}  &\cfrac{16}{3\pi^2}  &\cfrac{2}{\pi^2}\\
[-\halfbls]\\
\cfrac{\pi}{2} &\cfrac{9}{32}   &\cfrac{2}{\pi^2}    &\cfrac{1}{2\pi^2}
\end{array}
$$
  %  2,2  (6!! * 3!! * 3!!) / (5!! * 2 * 2) *1/2 /4  \frac{9}{10}
  %  2,1  (5!! * 2!! * 3!!) / (4!! * 1 * 2) *1/2 /4  \frac{45}{64}
  %  2,-1 (3!! * 1 * 3!!) / (2!! * 1 * 2) *1/2   /4  \frac{9}{32}
  %  1,1  (4!! * 2 * 2) / (3!! * 1 * 1) *2/Pi^2  /4  \frac{16}{3\pi^2}
  %  1,-1 (2!! * 2 * 1) / (1 * 1 * 1) *2/Pi^2    /4  \frac{2}{\pi^2}
\end{theorem}

Note  that  the  values of the constants do not depend neither on the
numbers  $n$  or  $m$  of corners, nor on the particular shape of the
billiard.   The  proof  of  this  theorem  also  relies  in  part  on
Theorem~\ref{theorem:birkhoff:chaika}    proved   by   Jon~Chaika   in
Appendix~\ref{sec:chaika}.

We  complete  this  section  with an illustration of further counting
problems  where  one  can  apply  our techniques. Let $\Narea(\Pi,L)$
denote  the  number of bands of closed periodic billiard trajectories
of  length  at most $L$ counted with a weight given by the normalized
area  of  the  band. More precisely, we count the area of overlapping
domains  of  the band twice: the area of the band is naively measured
as  the  area of the associated cylinder on the flat sphere, that is,
the  width  of  the  band  times the length of the closed trajectory,
normalized  by  the  area  of the billiard table. Having measured the
area  of  the band, we divide it by the area of the billiard table to
get the weight of the band.

\begin{theorem}
\label{theorem:Narea:weak:asymp}
For   any  billiard  $\Pi$  in  any  family  $\cB(k_1,\dots,k_n)$  of
right-angled  billiards  the weighted number $\Narea(\Pi,L)$ of bands
of  closed  billiard trajectories of length at most $L$ satisfies the
following weak asymptotics as $L\to\infty$:
$$
\Narea(\Pi,L) \mysim
\cfrac{1}{16\pi}\,
\sum_{j=1}^n\left(\cfrac{4}{k_j}-k_j\right)
\cdot
\cfrac{L^2}{\text{Area of the billiard table}}\ .
$$
For almost any billiard $\Pi$ in the same family, the asymptotics is,
actually, exact:
\begin{equation}
\label{eq:billiard:carea}
\Narea(\Pi,L) \sim
\cfrac{1}{16\pi}\,
\sum_{j=1}^n\left(\cfrac{4}{k_j}-k_j\right)
\cdot
\cfrac{L^2}{\text{Area of the billiard table}}\ .
\end{equation}
\end{theorem}
The weak asymptotics for all billiards follows, as before, from~\cite[Theorem 2.12]{Eskin:Mirzakhani:Mohammadi}. The strong asymptotics (\ref{eq:billiard:carea}) is proved in
~\S\ref{sec:point:weak}, using Jon Chaika's
Theorem~\ref{theorem:birkhoff:chaika} which is proved in
Appendix~\ref{sec:chaika}.  The constant in the corresponding counting
function is directly related to the \textit{Siegel--Veech area
  constant} for the corresponding stratum of meromorphic quadratic
differentials on $\CP$ discussed in \S\ref{sec:subsec:intro:volumes}.

%-------------------------------------------------------------------
\subsection{Right-Angled billiard tables and quadratic differentials}
\label{sec:subsec:intro:rab:qd}

Given  a  right-angled  billiard $\Pi$ in $\cB(k_1,\dots,k_n)$ we can
glue  a  topological  sphere  from  two  superposed  copies  of $\Pi$
identifying  the  boundaries  of  the  two  copies by isometries, see
Figure~\ref{fig:billiard:and:sphere}.  By  construction the resulting
topological  sphere  is  endowed  with  a  flat metric. Note that the
metric  is  regular  on  interior  of  the  segments  coming from the
boundary  of  $\Pi$:  one can unfold a neighborhood of any such point
into a small regular flat domain.

\begin{figure}[htb]
\includegraphics{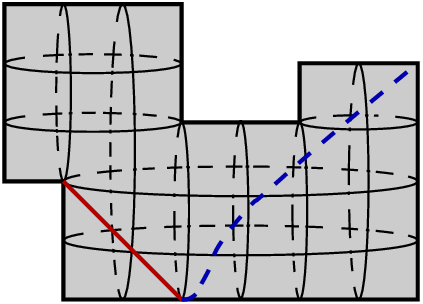}
\includegraphics{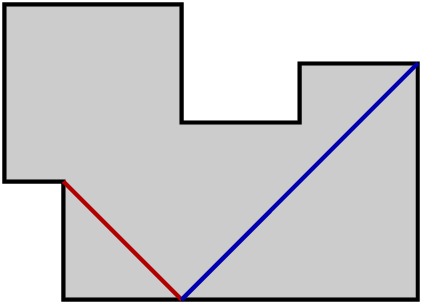}
\includegraphics{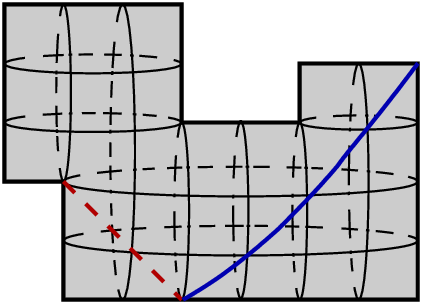}
\vspace{80pt}
\caption{
\label{fig:billiard:and:sphere}
Flat  spheres  glued  from two copies of a right-angled billiard. The
angle  by  which  the  billiard  table is rotated with respect to the
horizontal  position  encodes  the  ``phase''  of  the  corresponding
quadratic differential. A general generalized diagonal in the polygon
gives rise to two distinct saddle connections on the flat sphere.}
\end{figure}

However,  the  resulting  flat  metric has conical singularities with
cone  angles  $\pi  k_1,\dots, \pi k_n$ at the points coming from the
vertices  of  $\Pi$.  By construction the linear holonomy of the flat
metric with isolated singularities belongs to the group $\Z/2\Z$: the
parallel  transport  along  a  short  path encircling a conical point
$P_j$  brings a tangent vector $\vec v$ either to itself or to $-\vec
v$ depending on the parity of $k_j$.

It  is  known  that a flat metric with isolated conical singularities
and  with  holonomy in $\Z/2\Z$ on a closed surface defines a complex
structure  and  a  meromorphic  quadratic  differential  $q$  in this
complex   structure   defined   up  to  multiplication  by  a  scalar
$e^{i\phi}$.  Choosing  a line direction $\pm\vec v$ at some point of
the  resulting  flat  sphere as a ``horizontal'' direction we fix the
scalar  $e^{i\phi}$.  In  an  appropriate  flat  local coordinate $z$
outside  of  the  conical points the resulting quadratic differential
has  the  form  $(dz)^2$.  A  conical  singularity  with a cone angle
$k_i\pi$  corresponds  to  a  zero  of  the quadratic differential of
degree $k_i-2$, where a ``zero of degree $-1$'' is a simple pole.

Actually,  the two structures are synonymous: a meromorphic quadratic
differential  $q$  with  at  most  simple  poles on a Riemann surface
defines  a canonical flat metric with isolated conical singularities,
with  linear monodromy in $\Z/2\Z$ and with a distinguished foliation
by   straight   lines   in   the   flat   metric  (see  the  original
papers~\cite{Masur:interval}         and~\cite{Veech:Gauss}        or
surveys~\cite{Masur:Tabachnikov} and~\cite{Zorich:Houches}).

By   construction  closed  billiard  trajectories  in  $\Pi$  are  in
canonical  one-to-two correspondence with closed regular geodesics on
the associated flat sphere, and generalized diagonals on $\Pi$ are in
the   natural   one-to-two  correspondence  with  the  \textit{saddle
connections}     on     the     associated     flat    sphere,    see
Figure~\ref{fig:billiard:and:sphere}. Thus, the two counting problems
are closely related.

It   is   known   by   work   of   Veech~\cite{Veech:siegel}  and  of
Eskin-Masur~\cite{Eskin:Masur}  that  almost  all  flat  spheres in a
given  stratum  $\cQ(d_1, \dots, d_n)$ satisfy a quadratic asymptotic
formula  for  the  number  of  saddle connections. However, we cannot
immediately  translate  this  result  to  right-angled  billiards. An
elementary  count  shows  that  the  space  $\cB(k_1,\dots,k_n)$  has
\textit{real}   dimension   $n-2$,   while   the  associated  stratum
$\cQ(k_1-2,\dots,k_n-2)$  has \textit{complex} dimension $n-2$. Thus,
flat spheres constructed from right-angled billiards form a subset of
measure  zero,  and  ``almost  all''  results  for the strata are not
applicable to families of billiards. This is the common difficulty of
translating results valid for flat surfaces to billiards.

In  our  specific  case we are lucky enough to get a subspace of flat
spheres  ``of  billiard origin'' which is transversal to the unstable
foliation  of  the Teichm\"uller flow (see \S\ref{sec:billiards:qd}).
This  allows us to  apply  certain  techniques  of hyperbolic
dynamics to obtain some ergodic results in slightly weaker form. As a
corollary  we  obtain  the  desired  information  on quadratic
asymptotics  in the counting problems for almost all billiards. The
corresponding ergodic technique is presented in \S\ref{sec:ergodic}. A
key tool we use is Theorem~\ref{theorem:birkhoff:chaika} proved by Jon
Chaika in Appendix~\ref{sec:chaika}.

%--------------------------------------------------------------------
\subsection{Reader's guide}
The paper (like Caesar's Gaul) is composed of three parts.
The   reader   interested   only  in  the  billiards  may  read  only
\S\ref{sec:configurations:counting}
(and
optionally \S\ref{sec:billiards:qd} and \S\ref{sec:ergodic}).
The ergodic theorem we use in
  \S\ref{sec:ergodic} is due to Jon Chaika, and is proved in Appendix~\ref{sec:chaika}.

The   part   where   we   compute   the   volume   of   any   stratum
$\cQ_1(d_1,\dots,d_n)$ of meromorphic quadratic differentials with at
most   simple   poles   on   $\CP$   and   where   we   compute   the
\textit{Siegel--Veech constants} for these strata is independent from
the     rest     of     the     paper.    It    is    presented    in
\S\ref{sec:subsec:configurations},
\S\S\ref{sec:subsec:coordiantes:in:stratum}--\ref{sec:subsec:hat:homologous:saddle:connections}
and   in  \S\S\ref{sec:Siegel:Veech}--\ref{sec:induction}  (with  one
verification         in        Appendix~\ref{sec:appendix:identity}).
  % Appendix~\ref{a:SV:constants:for:hyperelliptic:loci}  compares  these
  % results  with  analogous  results for the hyperelliptic components of
  % the strata of Abelian differentials.

Finally,  Appendix~\ref{sec:pillowcase:covers}  devoted to pillowcase
covers is completely independent of the rest of the
paper.

%--------------------------------------------------------------------
\subsection{Historical remarks}
The  formula  for  the  volume  of  the strata of
quadratic  differentials  was  guessed  by  M.~Kontsevich more than a
decade  ago.  At this time formula~\eqref{eq:carea:answer} related to
Lyapunov   exponents  was  known  experimentally.  The  Siegel--Veech
constant~\eqref{eq:carea:in:terms:of:volumes}  has  especially simple
form for the strata $\cQ(d,-1^{d+4})$ of quadratic differentials with
a    single    zero    and    only    simple    poles    on    $\CP$.
Comparing~\eqref{eq:carea:answer}         and        a        version
of~\eqref{eq:carea:in:terms:of:volumes}   M.~Kontsevich   obtained  a
conjectural  formula  for  $\Vol\cQ_1(d,-1^{d+4})$.  Motivated by the
simplicity of the resulting expression as a function of $d$ he stated
a  guess that $\Vol\cQ_1(d_1,\dots,d_k)$ for any stratum in genus $0$
might  be expressed as a product of the corresponding expressions for
all $d_i$.

%--------------------------------------------------------------------
\subsection*{Acknowledgments}
The  authors  are grateful to M.~Kontsevich for the conjecture on the
volumes        and        for        collaboration       in       the
work on Lyapunov exponents essential for the current paper.

Part  of  this  paper  strongly  relies  on  techniques  developed  in
collaboration  with  H.~Masur.  We  are  grateful to him for his very
important contribution.

We  thank  C.~Boissy  for  the list of configurations of \^homologous
saddle  connections  in  genus  $0$,  elaborated specifically for our
needs, and for his pictures of configurations.

We  thank  P.~Paule  for  the kind permission to use the ``MultiSum''
package,  which  was helpful in developing certain intuition with the
combinatorial identity.

We thank E.~Goujard who carefully read the first version of this
paper, provided us with a list of typos,
and      indicated      to     us     a     missing     factor     in
formula~\eqref{eq:volume:through:covers}   and   necessity  to  prove
Lemma~\ref{lm:chessboard:coloring}        as       explained       in
Remark~\ref{rm:genus:0:only}.

We are grateful to P.~Hubert and to anonymous referee whose
suggestions helped to improve the presentation and clarity
of the arguments.

The  authors  are  happy  to  thank  IHES,  IMJ,  IUF,  MPIM,
MSRI, and the
Universities   of   Chicago,  of  Illinois  at  Urbana-Champaign,  of
Rennes~1,  and  of  Paris~7 for hospitality during the preparation of
this paper.

%####################################################################
%####################################################################
%####################################################################

\section{Configurations and Counting Theorems}
\label{sec:configurations:counting}

\subsection{Types of saddle connections and generalized diagonals}
\label{sec:subsec:configurations}
We  distinguish  the  following  four  ways  of  getting  generalized
diagonals  in  a right-angled billiard. They correspond to four types
of configurations of saddle connections on a flat sphere defined by a
meromorphic     quadratic    differential    with    simple    poles,
see~\cite{Eskin:Masur:Zorich}   and~\cite{Masur:Zorich}  for  general
information   on   \textit{configurations   of   saddle  connections}
and~\cite{Boissy:configurations} for specific case of $\CP$.
\smallskip

\textbf{I. Saddle connection joining distinct singularities.} In this
situation  (see  Figure~\ref{fig:I:distinct:singularities}) we have a
generalized  diagonal  joining  a  corner  $P_i$ with the inner angle
$k_i\frac{\pi}{2}$, where $k_i\ge 3$, to a distinct corner $P_j$.

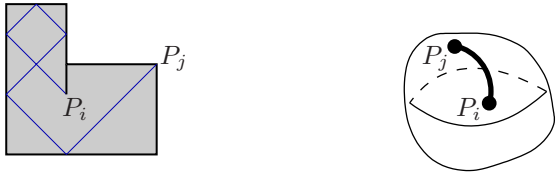
\begin{figure}[htb]
%
%  Saddle connection joining distinct singularities
   %
\includegraphics{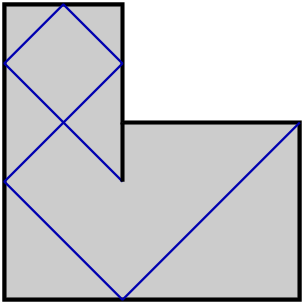}
\includegraphics{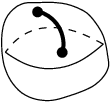}
\begin{picture}(0,0)(0,5)
\put(-49,-15){$P_j$}
\put(-86,-35){$P_i$}
\put(50,-15){$P_j$}
\put(63,-35){$P_i$}
\end{picture}
\vspace{60bp}
\caption{
\label{fig:I:distinct:singularities}
Type  I.  On  the  left:  a generalized diagonal joining two distinct
corners  of  the  billiard, where at least one of the two corners has
inner  angle  at  least $\frac{3\pi}{2}$. It does not bound a band of
closed  trajectories.  On  the  right:  a  saddle connection on $\CP$
joining a zero to a distinct zero (or to a pole).
   }
\end{figure}

The  induced flat metric on $\CP$ has an associated saddle connection
of  the  same  length joining the zero $P_i$ to the distinct zero (or
simple pole) $P_j$.
\smallskip

\textbf{II.  Saddle  connection  joining  a  zero  to  itself.}  This
situation  (see Figure~\ref{fig:II:zero:to:itself}) can happen only
when  we  have  a corner $P_i$ with a corner angle $k_i\frac{\pi}{2}$
with  $k_i\ge  4$.  In  this  case we can have a generalized diagonal
joining the corner $P_i$ to itself such that it does not bound a band
of  closed  regular  trajectories.

\begin{figure}[htb]
%
% II. Saddle connection joining zero to itself
   %
\includegraphics{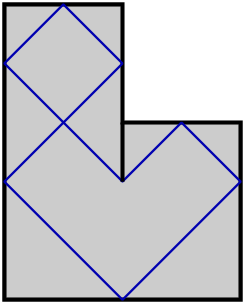}
\includegraphics{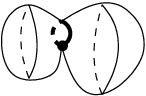}
\begin{picture}(0,0)(0,5)
\put(-86,-35){$P_i$}
\put(79,-35){$P_i$}
\end{picture}
\vspace{60bp}
\caption{
\label{fig:II:zero:to:itself}
Type  II.  On  the left: a generalized diagonal returning to the same
corner.  For this type, it does not bound closed trajectories. On the
right:  the  corresponding saddle connection joining a zero (of order
at least $2$) to itself.
   }
\end{figure}
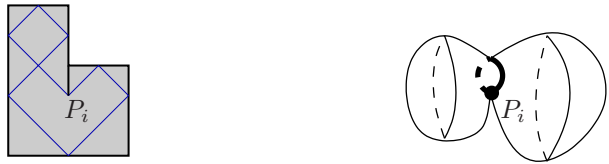

For  the  induced  flat metric on $\CP$ we get a corresponding saddle
connection  of  the same length joining the zero $P_i$ to itself such
that  the  total  angle $k_i\pi$ at the singularity $P_i$ is split by
the  separatrix  loop  into  two  sectors  having the angles strictly
greater  than  $\pi$  (which  is  equivalent  to  the  condition that
generically such a saddle connection does not bound a cylinder filled
with  periodic  geodesics).
\smallskip

\textbf{III.     A    ``pocket''.}    In    this    situation    (see
Figure~\ref{fig:III:pocket}) we have a band of periodic trajectories.
The  boundary  of  the band is composed of two generalized diagonals.
The  first  generalized diagonal joins a pair of corners $P_i$, $P_j$
with   inner  angles  $\frac{\pi}{2}$.  The  length  of  this  saddle
connection  is  twice  shorter  than  the length of periodic billiard
trajectory  in  the  band.  The  second  generalized diagonal joins a
corner  $P_i$  with inner angle $k_i\frac{\pi}{2}$ with $k_i\ge 3$ to
itself.  The  length  of  this  saddle  connection is the same as the
length of periodic billiard trajectory in the band.

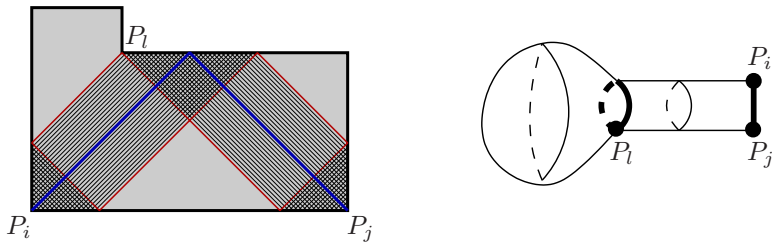
\begin{figure}[htb]
%
%  III. POCKET
   %
\includegraphics{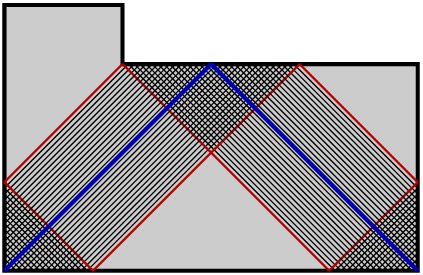}
\includegraphics{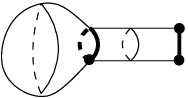}
\begin{picture}(0,0)(0,-13)
\put(-145,-93){$P_i$}
\put(-100,-21){$P_l$}
\put(-16,-93){$P_j$}
\end{picture}
\begin{picture}(0,0)(-150,0)
\put(-70,-52){$P_l$}
\put(-18,-16){$P_i$}
\put(-18,-52){$P_j$}
\end{picture}
\vspace{80bp}
\caption{
\label{fig:III:pocket}
Type  III.  On the left: a band of closed trajectories bounded by two
generalized   diagonals.  One  of  generalized  diagonals  joins  two
distinct  corners  with  angles $\frac{\pi}{2}$; the other returns to
the   same   corner.  On  the  right:  the  corresponding  ``pocket''
configuration  with  a  cylinder  bounded  on  one  side  by a saddle
connection  joining  two  simple  poles,  and  by a saddle connection
joining a zero to itself on the other side.
   }
\end{figure}

For  the associated flat metric on $\CP$ we get a cylinder filled with
closed  regular  trajectories.  One of the boundary components of the
cylinder  degenerates to a saddle connection joining two simple poles
$P_i$,  $P_j$.  Clearly, this saddle connection is twice shorter than
the length of the periodic trajectories. The other boundary component
is  a  saddle  connection joining the zero $P_l$ to itself. The total
angle  $k_l\pi$  at  the singularity $P_l$ is split by the separatrix
loop  into two sectors, such that the sector adjacent to the cylinder
has  angle $\pi$. The length of this saddle connection is the same as
the  length  of the periodic trajectories in the cylinder.
\smallskip

\textbf{IV.   A   ``dumbbell''.}   In   this   last   situation  (see
Figure~\ref{fig:IV:dumbbell})  we  again  have  a  band  of  periodic
trajectories.  The  boundary  of  the  band  is again composed of two
generalized  diagonals,  but this time the first generalized diagonal
joins the corner $P_i$ with inner angle $k_i\frac{\pi}{2}$ to itself,
and  the  second generalized diagonal joins the distinct corner $P_j$
with  inner  angle  $k_j\frac{\pi}{2}$  to itself. Both $k_i,k_j$ are
greater  than  or  equal  to  $3$.  The  length  of each of these two
generalized  diagonals  is  the  same as the length of every periodic
billiard trajectory in the band.

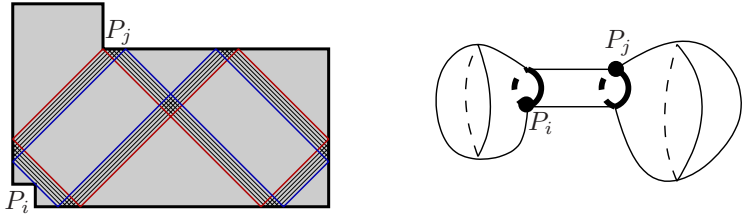
\begin{figure}[htb]
\centering
%
% IV. DUMBBELL
   %
\includegraphics{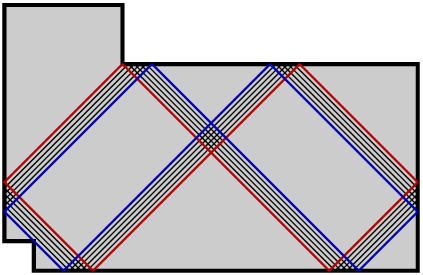}
\includegraphics{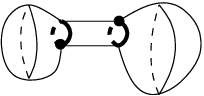}
\begin{picture}(0,0)(-25,10)
\put(-165,-65){$P_i$}
\put(-127,0){$P_j$}
\put(33,-35){$P_i$}
\put(62,-4){$P_j$}
\end{picture}
\vspace{70bp}
\caption{
\label{fig:IV:dumbbell}
Type IV. On the left: a band of periodic trajectories, such that each
of the two bounding generalized diagonals returns to the same corner.
On the right: a ``dumbbell'' composed of two flat spheres joined by a
cylinder.  Each  boundary  component  of  the  cylinder  is  a saddle
connection joining a zero to itself.
   }
\end{figure}

For the associated flat metric on $\CP$ we get a cylinder filled with
closed  regular  trajectories.  On each of the boundary components of
the  cylinder  we  have  a  saddle  connection joining the zero $P_i$
(correspondingly  $P_j$)  to  itself.  The  length of each of the two
saddle  connections  is  the  same  as  the  length  of  the periodic
trajectories in the cylinder.
\smallskip

The  following two Propositions explain why we distinguish these four
particular types of configurations
(see more details
in ~\S\ref{sec:subsec:hat:homologous:saddle:connections} which discusses
a homological interpretation of these statements).

\begin{proposition}
\label{prop:only:four:types}
Almost  any  flat  surface $S$ in any stratum $\cQ_1(d_1, \dots,d_n)$
different  from  the pillowcase stratum $\cQ_1(-1^4)$ does not have a
single  pair  of parallel saddle connections different from the pairs
involved in configurations of types $I,II,III,IV$.
\end{proposition}

Proposition~\ref{prop:only:four:types}       is       proved       in
~\S\ref{sec:subsec:hat:homologous:saddle:connections}.
An analogous statement can be formulated for right-angled billiards.

\begin{proposition}
\label{pr:no:other:for:billiards}

For    almost    any    right-angled    billiard    in   any   family
$\cB(k_1,\dots,k_n)$ the following property holds. Consider a pair of
trajectories,  where  each  trajectory  is  either  a  closed regular
trajectory or a generalized diagonal. Suppose that these trajectories
are  not  parallel to any side of the polygon. If some segment of the
first   trajectory   is  parallel  to  some  segment  of  the  second
trajectory, then both trajectories make part of one of configurations
I--IV described in~\ref{sec:subsec:configurations}.
\end{proposition}

Proposition~\ref{pr:no:other:for:billiards}  mimics  Proposition  7.4
in~\cite{Eskin:Masur:Zorich};       it       is       proved       in
~\S\ref{sec:subsec:The:subspace:of:billiards}.

\bold{Configurations  of saddle connections.} In addition to the type
I--IV of a saddle connection, we may specify some extra combinatorial
information, for example the indices (``names'') of all singularities
involved.  For  saddle  connections  of type IV, where a
cylinder is joining two spheres, we specify not only the zeroes $P_i$
and  $P_j$  at  the  boundary components of the cylinder, but we also
specify   the   subcollections   $P_{i_1},  \dots,  P_{i_{k_1}}$  and
$P_{j_1},  \dots,  P_{j_{k_2}}$ of numbered zeroes an poles which get
to  the  first and to the second sphere correspondingly. We call this
information the \textit{configuration} of a saddle connection (or the
configuration  of  saddle  connections, when there are several saddle
connections  involved  as  in  types  III and IV). By convention, the
configuration  of  saddle  connections  includes  its  type. See also
~\S\ref{sec:subsec:hat:homologous:saddle:connections}  for a homological
interpretation of a configuration of saddle connections.

\bold{Configuration    of    a    generalized   diagonal.}   By   the
\textit{configuration}  of  the  generalized  diagonal  we  mean  the
configuration of the associated saddle connections in $\CP$ described
in \S\ref{sec:subsec:intro:rab:qd}.

\subsection{Counting Theorems}
\label{sec:subsec:counting:theorems}

By the notation
\begin{displaymath}
N(L) \sim c L^2
\end{displaymath}
we mean as customary,
\begin{displaymath}
\lim_{L \to \infty} \frac{N(L)}{L^2} = c.
\end{displaymath}
For technical reasons, we will need to consider ``weak asymptotic
formulas''
\begin{displaymath}
N(L) \mysim c L^2
\end{displaymath}
which means
\begin{displaymath}
\lim_{L \to \infty} \frac{1}{L} \int_{0}^L N(e^t) e^{-2t} \, dt = c.
\end{displaymath}

The  following  theorem  (which  is  a  special  case  of  results of
\cite{Veech:siegel}  and  \cite{Eskin:Masur})  establishes  a  strong
asymptotic  formula  for  almost  all  flat surfaces in a stratum.
By   convention   we   always   count  \textit{non-oriented}  saddle
connections and \textit{non-oriented} closed flat geodesics.
\begin{theorem}
\label{theorem:SV:for:CP}
For almost any flat surface $S$ in any stratum $\cQ(d_1, \dots, d_n)$
of  meromorphic  quadratic differentials with at most simple poles on
$\CP$ the number $N_\cC(S,L)$ of occurrences of saddle connections of
length  at  most  $L$ and of fixed configuration $\cC$, has quadratic
asymptotics in $L$:
\begin{displaymath}
\label{eq:Th:N:cC:S:L}
N_\cC(S,L)\sim c_\cC\cdot
\frac{\pi L^2}{\text{Area of }S}
\,.
\end{displaymath}
The  constants $c_\cC$ are called Siegel-Veech constants. They depend
only  on  the  configuration  $\cC$  and  on $d_1, \dots, d_n$. Their
values are given in \S\ref{sec:Siegel:Veech}.
\end{theorem}
Theorem~\ref{theorem:SV:for:CP}          is         proved         in
\S\ref{sec:subsec:Reduction:to:ergodic:theory}.           Note           that
Theorem~\ref{theorem:SV:for:CP}  has  no  relation  to  billiards, it
concerns  only flat metrics on $\CP$ induced by meromorphic quadratic
differentials  with  simple poles. In \S\ref{sec:subsec:intro:rab:qd}
we  described  how  a  right-angled  billiard table $\Pi$ canonically
determines  a  meromorphic  quadratic differential on $\CP$. However,
since  the  image  of  the  resulting  map  $\cB(k_1, \dots, k_n) \to
\cQ(k_1-2,  \dots  k_n-2)$  has  measure  $0$  in  $\cQ(k_1-2, \dots,
k_n-2)$,  results  such  as  Theorem~\ref{theorem:SV:for:CP}  do  not
immediately    imply    anything    about   right-angled   billiards.
Nevertheless, we have the following:
\begin{theorem}
\label{theorem:SV:for:billiards}
For  almost  any  billiard table $\Pi$ in any family $\cB(k_1, \dots,
k_n)$   of   right-angled  billiards  the  number  $N_\cC(\Pi,L)$  of
occurrences  of  generalized  diagonals of configuration $\cC$ and of
length at most $L$ has quadratic asymptotics in $L$:
\begin{equation}
\label{eq:SV:for:billiards}
N_\cC(\Pi,L) \sim \frac{c_\cC}{4}\cdot
\frac{\pi L^2}{\text{Area of the billiard table }\Pi}
\,,
\end{equation}
where  the  constants  $c_\cC$  are  the  corresponding Siegel--Veech
constants   $c_\cC$   for  the  stratum  $\cQ(k_1-2,\dots,k_n-2)$  in
Theorem~\ref{theorem:SV:for:CP}.
\end{theorem}
\medskip
\noindent Theorem~\ref{theorem:SV:for:billiards}       is       proved       in
\S\ref{sec:point:weak}, using Jon Chaika's
  Theorem~\ref{theorem:birkhoff:chaika} which is proved in
  Appendix~\ref{sec:chaika}.

The   factor   of  $\tfrac{1}{4}$  in~\eqref{eq:SV:for:billiards}  is
explained  as  follows.  Note  that  any  generalized diagonal in the
billiard  table  $\Pi$  which  is not parallel to one of the sides of
$\Pi$  canonically determines two symmetric saddle connections of the
same type on the flat surface $S$ glued from the two copies of $\Pi$,
where   the   symmetry   is   the   antiholomorphic  involution,  see
Figure~\ref{fig:billiard:and:sphere}. Hence,
$$
N_\cC(\Pi,L)=\frac{1}{2} N_\cC(S,L)\,.
$$
Note  also, that by construction the area of $S$ is twice the area of
the billiard table $\Pi$.

\begin{figure}[htb]
\includegraphics{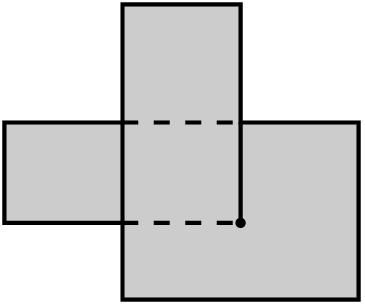}
\vspace{65bp}
\caption{
\label{fig:helical:billiard}
A  helical  billiard  corresponds  to  the  stratum $\cQ(d,-1^{d+4})$.
   }
\end{figure}

Note  that  our  billiard  table  does  not  need  to  be necessarily
embeddable   into   the   plane,  say,  we  can  consider  a  helical
right-angled  billiard  as in Figure~\ref{fig:helical:billiard}. More
precisely,  by  a  \textit{right-angled  billiard  table}  we  call a
topological  disc  endowed  with  a  flat metric having the following
properties.  The  flat  metric  is allowed to have isolated cone-type
singularities  in  the  interior  of the disc with cone angles of the
form   $l_i\pi$,  with  $l_i\in\N$.  The  boundary  of  the  disc  is
piecewise-geodesic  in  the  flat  metric, and the angles between the
geodesic segments have the form $k_j\pi/2$, with $k_j\in\N$.

In fact, some version of Theorem~\ref{theorem:SV:for:billiards}
holds for individual billiards:
\begin{theorem}
\label{theorem:invididual:billiards}  Suppose  $\Pi$  is  a  billiard
table  from  the  family  of  right-angled billiards $\cB(k_1, \dots,
k_n)$.  Furthermore,  suppose $\Pi$ is such that the flat surface $S$
glued  from  two  copies  of  $\Pi$  does  not  belong  to any proper
$\GL$-invariant     affine     submanifold     of     the     stratum
$\cQ(d_1,\dots,d_k)$. Then, for any choice
I--IV
of configuration $\cC$, the weak asymptotic formula
\begin{displaymath}
N_\cC(\Pi,L) \mysim \frac{c_\cC}{4}\cdot
\frac{\pi L^2}{\text{Area of the billiard table }\Pi}
\end{displaymath}
holds,
where   $c_\cC$   is  the  Siegel--Veech  constant
corresponding   to   the   configuration   $\cC$   in   the   stratum
$\cQ(k_1-2,\dots,k_n-2)$ (as in Theorem~\ref{theorem:SV:for:CP}).
\end{theorem}
  %
  %   We        note       that       a       complete       proof       of
  %   Theorem~\ref{theorem:invididual:billiards}  is  well  over  200 pages
  %   long.
  %
\begin{proof}
The   statement   is   an   immediate   corollary  of~\cite[Theorem 2.12]{Eskin:Mirzakhani:Mohammadi}.
\end{proof}
We     note     that     a    complete    proof    of  ~\cite[Theorem~2.12]{Eskin:Mirzakhani:Mohammadi}     involves     the     measure
classification  theorem  of~\cite{Eskin:Mirzakhani}  and is well over
200 pages long, and yields \emph{weak} asymptotic formulas.
  %
  %   We        note       that       a       complete       proof       of
  %   Theorem~\ref{theorem:invididual:billiards}  is  well  over  200 pages
  %   long.
  %
The proof of Theorem~\ref{theorem:SV:for:billiards} is
  much shorter,  and uses
special     features     of     right-angled     billiards,    namely
Proposition~\ref{prop:map:onto}. However,
Theorem~\ref{theorem:SV:for:billiards} is an almost everywhere
statement, and does not imply any type of asymptotic formula for an individual
billiard table.

We also note that for most other families of billiards,
almost-everywhere statements like
Theorem~\ref{theorem:SV:for:billiards} are not available (since the
analogue of Proposition~\ref{prop:map:onto} fails.)

% In  \S\ref{sec:ergodic}, we also prove the following $L^1$-statement.
% Denote  by $\cB_1(k_1,\dots,k_n)$ the hypersurface of billiard tables
% of  area  one  in  the  family  $\cB(k_1,\dots,k_n)$,  and  denote by
% $\mu_{\cB_1}$  the measure on $\cB_1(k_1,\dots,k_n)$ induced from the
% Lebesgue   measure   $\mu_{\cB}$   on  $\cB(k_1,\dots,k_n)$  to  this
% hypersurface.
%    %
% \begin{theorem}
% \label{theorem:l1poly}
% Let
%    %
% $$
% A \subset \cB_1(k_1, \ldots k_n)
% $$
%    %
% be   a   set   of   positive   $\mu_{\cB_1}$-measure.  Then  for  any
% configuration $\cC$ as $L \to \infty$,
%    %
% \begin{equation}
% \label{avepoly1}
% \frac{1}{\mu_{\cB_1}(A)} \int_A N_\cC(\Pi, L) \, d\mu_{\cB_1}
% \sim \frac{c_\cC}{4} \cdot
% \frac{\pi L^2}{\text{Area of the billiard table }\Pi}\,,
% \end{equation}
%   %
% where  $c_\cC$  is  the  Siegel--Veech  constant corresponding to the
% configuration $\cC$ in the stratum $\cQ(k_1-2,\dots,k_n-2)$.
% \end{theorem}

% Theorem~\ref{theorem:l1poly} is proved in \S\ref{sec:l1}.

%####################################################################
%####################################################################
%####################################################################

\section{Billiards   in  right-angled   polygons   and   quadratic
differentials}
\label{sec:billiards:qd}

In    ~\S\ref{sec:subsec:coordiantes:in:stratum}    we    describe   the
cohomological coordinates in a stratum of quadratic differentials. We
proceed  in ~\S\ref{sec:subsec:hat:homologous:saddle:connections} with a
reminder  of  the notions of \textit{\^homologous saddle connections}
and a \textit{configuration} of \^homologous saddle connections.

In    ~\S\ref{sec:subsec:The:subspace:of:billiards}   we   analyze   the
canonical  embedding  of  the  space  of  (directional)  right-angled
billiards   $\cB(k_1,\dots,k_n)$   into   the   corresponding   space
$\cQ(k_1-2,\dots,k_n-2)$  of  meromorphic  quadratic differentials on
$\CP$.  Namely,  we prove in Proposition~\ref{prop:map:onto} that its
image  projects  surjectively  onto  the  unstable  foliation  of the
Teichm\"uller geodesic flow, which allows us to apply certain ergodic
techniques  of  hyperbolic  dynamics  not  only to flat surfaces from
$\cQ(k_1-2,\dots,k_n-2)$ but to billiards from $\cB(k_1,\dots,k_n)$.

We   complete   ~\S\ref{sec:billiards:qd}   with   a   proof  of
Proposition~\ref{pr:no:other:for:billiards}.

%----------------------------------------------------------------
\subsection{Coordinates in a stratum of quadratic differentials.}
\label{sec:subsec:coordiantes:in:stratum}
Consider a meromorphic quadratic differential $\psi$ having zeroes of
arbitrary  multiplicities  but  only simple poles on $\CP$. Let $P_1,
\ldots,  P_n$  be  its  singular  points  (zeros  and  simple poles).
Consider  the minimal branched double covering $p:\hat S\to \CP$ such
that   the   induced   quadratic  differential  $p^\ast\psi$  on  the
hyperelliptic  surface  $\hat  S$  is  a square of an Abelian
differential $p^\ast\psi=\omega^2$.
\begin{figure}[htb]
%
%  HYPERELLIPTIC SURFACE
%
%
\includegraphics{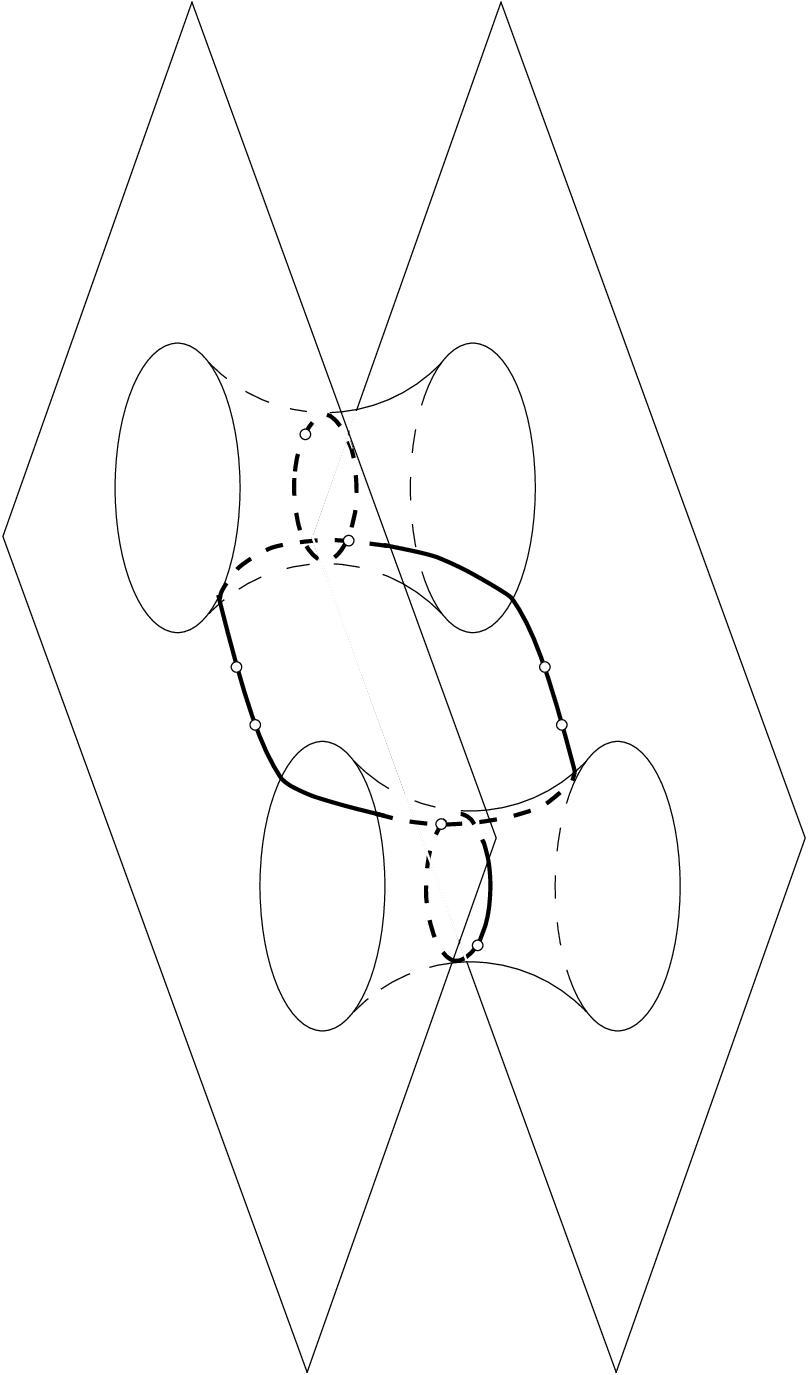}
\includegraphics{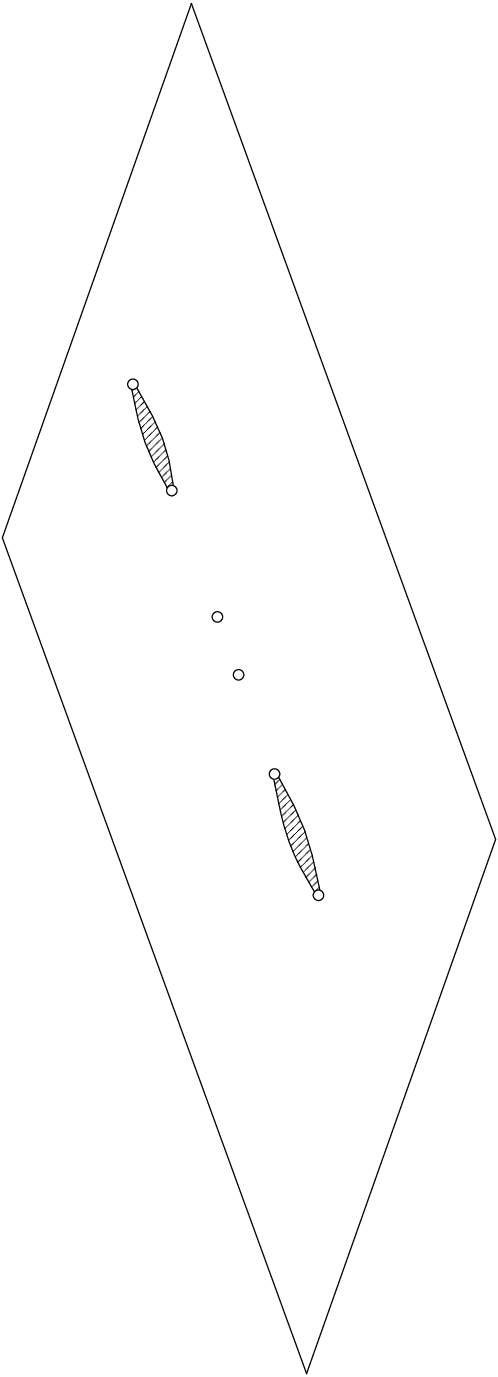}
\begin{picture}(0,0)(0,0)
\put(-53,-101){$\hat P_1$}
\put(-17,-88){$\hat P_2$}
\put(0,-58){$\hat P^+_i$}
\put(0,-118){$\hat P^-_i$}
\put(35,-65){$\hat P_{n-1}$}
\put(61,-59){$\hat P_n$}
\end{picture}
\begin{picture}(0,0)(0,0)
\put(-48,-223){$P_1$}
\put(-15,-213){$P_2$}
\put(3,-206){$P_i$}
\put(33,-194){$P_{n-1}$}
\put(65,-184){$P_n$}
\end{picture}
\vspace{240bp}
\caption{
\label{fig:hyperel}
Basis  of  cycles  in $H_1^-(\hat S,\{\hat P_1,\dots,\hat P_N\};\Z)$.
Note that the cycle corresponding to the very last slit is omitted.
}
\end{figure}
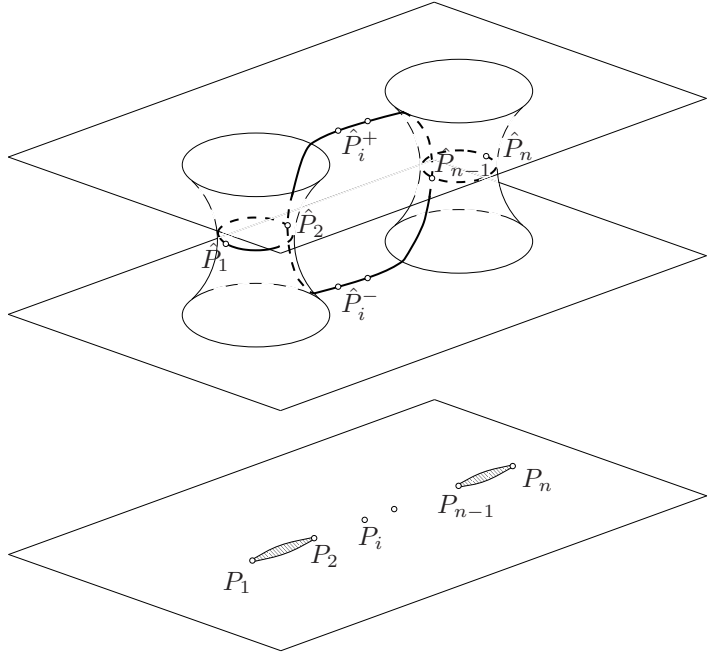

The  zeros  $\hat{P}_1,\ldots,\hat{P}_N$  of  the  resulting  Abelian
differential  $\omega$  correspond  to  the  zeros  of  $\psi$ in the
following  way:  every  zero  $P\in  \CP$ of $\psi$ of odd order is a
ramification  point  of  the  covering,  so it produces a single zero
$\hat{P}\in\hat  S$  of  $\omega$;  every zero $P\in\CP$ of $\psi$ of
even  order  is  a  regular point of the covering, so it produces two
zeros  $\hat{P}^+,\hat{P}^-\in\hat  S$ of $\omega$. Every simple pole
of  $\psi$ defines a branching point of the covering; this point is a
regular point of $\omega$.

Consider  the subspace $H_1^-(\hat S,\{\hat P_1,\dots,\hat P_N\};\Z)$
of  the relative homology of the cover with respect to the collection
of   zeroes  $\{\hat  P_1,\dots,\hat  P_N\}$  of  $\omega$  which  is
antiinvariant with respect to the induced action of the hyperelliptic
involution.  We  are  going to construct a basis in this subspace (in
complete  analogy  with  a  usual  basis  of  absolute  cycles  for a
hyperelliptic surface).

We  can  always  enumerate  the  singular  points $P_1,\dots, P_n$ of
$\psi$  in such a way that $P_n$ is a simple pole. Chose now a simple
oriented   broken   line   $P_1,\ldots,P_{n-1}$   on   $\CP$  joining
consecutively  all the singular points of $\psi$ except the last one.
For  every  arc $[P_i,P_{i+1}]$ of this broken line, $i=1,\dots,n-2$,
the  difference  of  their  two preimages defines a relative cycle in
$H_1^-(\hat  S,\{\hat P_1,\dots,\hat P_N\};\Z)$. By construction such
a   cycle   is   antiinvariant  with  respect  to  the  hyperelliptic
involution.  It  is immediate to see that the resulting collection of
cycles   forms   a   basis  in  $H_1^-(\hat  S,\{\hat  P_1,\dots,\hat
P_N\};\Z)$.

Note  that a  preimage  of a simple pole does not belong to the set
$\hat   P_1,\ldots,\hat   P_N$.   Thus,   a   preimage   of   an  arc
$[P_i,P_{i+1}]$ having a simple pole as one of the endpoints does not
define  a  cycle  in  $H_1(\hat  S,\{\hat  P_1,\dots,\hat P_N\};\Z)$.
However,  since  a  simple  pole  is  always  a  branching point, the
\textit{difference}  of  the  preimages  of  such  arc  is  already a
well-defined  relative  cycle  in  $H_1(\hat  S,\{\hat P_1,\dots,\hat
P_N\};\Z)$.

Let  $\cQ(d_1,\dots,d_n)$  be the ambient stratum for the meromorphic
quadratic   differential   $(\CP,\psi)$.   The  subspace  $H^1_-(\hat
S,\{\hat     P_1,\dots,\hat     P_N\};\C{})$    in    the    relative
\textit{cohomology}   antiinvariant   with  respect  to  the  natural
involution defines local coordinates in the stratum.
  %--------------------------------------------
  % INSERT BRIEF SECTION ON SL(2, R) ACTION
  %--------------------------------------------
\subsection{\^Homologous saddle connections}
\label{sec:subsec:hat:homologous:saddle:connections}

We  follow the exposition in~\cite{Masur:Zorich} introducing the notions
of   a   \textit{rigid}  collection  of  saddle  connections  and  of
\textit{\^homologous}  saddle connections. Consider a flat sphere $S$
corresponding  to  a  meromorphic quadratic differential $(\CP,\psi)$
with  at  most simple poles. Any saddle connection on the flat sphere
$S$ persists under small deformations of $S$ inside $\cQ(\alpha)$. It
might  happen  that  any  deformation  of  a given flat surface which
shortens  some  specific  saddle connection necessarily shortens some
other  saddle  connections.  We  say  that  a collection $\{\gamma_1,
\dots,  \gamma_n\}$  of  saddle  connections  is  \emph{rigid} if any
sufficiently small deformation of the flat surface inside the stratum
preserves  the  proportions $|\gamma_1|:|\gamma_2|: \dots:|\gamma_n|$
of the lengths of all saddle connections in the collection.

Consider  the  canonical  double  cover  $\hat S$ over $S$ defined in
~\S\ref{sec:subsec:coordiantes:in:stratum}.  Given  a  saddle connection
$\gamma$   on   $S$   choose  an  orientation  of  $\gamma$  and  let
$\gamma',\gamma''$  be its lifts to the double cover $\hat S$ endowed
with     the     orientation     inherited    from    $\gamma$.    If
$[\gamma']=-[\gamma'']$    as   cycles   in   $H_1(\hat   S,   \{\hat
P_1,\dots,\hat   P_N\};\,\Z)$   we   let   $[\hat\gamma]:=[\gamma']$,
otherwise         we         define         $[\hat\gamma]$         as
$[\hat\gamma]:=[\gamma']-[\gamma'']$. It immediately follows from the
above  definition  that  the cycle $[\hat\gamma]$ defined by a saddle
connection   $\gamma$   is  always  \textit{primitive}  in  $H_1(\hat
S,\{\hat  P_1,\dots,\hat  P_N\};\,\Z)$  and  belongs  to  $H_1^-(\hat
S,\{\hat P_1,\dots,\hat P_N\};\,\Z)$.

Following~\cite{Masur:Zorich} we introduce the following

\begin{definition}
\label{def:homologous}
The  saddle  connections  $\gamma_1,\gamma_2$  on  a flat surface $S$
defined  by a quadratic differential $q$ are \textit{\^homologous} if
$[\hat\gamma_1]=[\hat\gamma_2]$ in $H_1(\hat S,\hat P;\,\Z)$ under an
appropriate choice of orientations of $\gamma_1, \gamma_2$.
(The   notion  ``homologous  in  the  relative  homology  with  local
coefficients  defined  by  the  canonical  double  cover induced by a
quadratic  differential''  is  unbearably  bulky, so we introduced an
abbreviation ``\^homologous''. We stress that the circumflex over the
``h''  is quite meaningful: as it is indicated in the definition, the
corresponding cycles are homologous {\it on the double cover}.)
\end{definition}

Note  that since there is no canonical way to enumerate the preimages
$\gamma',  \gamma''$  of  a  saddle connection $\gamma$ on the double
cover,  the  cycle  $[\hat\gamma]$ is defined only up to a sign, even
when  we  fix  the  orientation  of  $\gamma$.  Thus,  $\gamma_1$  is
\^homologous      to      $\gamma_2$      if      and     only     if
$[\hat\gamma_1]=\pm[\hat\gamma_2]$.

\begin{NNProposition}[H.~Masur, A.~Z.]
\label{pr:rigid:configurations:hat:homologous}
Let  $S$  be  a flat surface corresponding to a meromorphic quadratic
differential  $q$  with at most simple poles. A collection $\gamma_1,
\dots, \gamma_n$ of saddle connections on $S$ is rigid if and only if
all saddle connections $\gamma_1, \dots, \gamma_n$ are \^homologous.
\end{NNProposition}

There   is  an  obvious  geometric  test  for  deciding  when  saddle
connections  $\gamma_1,  \gamma_2$  on  a translation surface $S$ are
homologous:   it   is   sufficient   to   check  whether  $S\setminus
(\gamma_1\cup\gamma_2)$     is    connected    or    not    (provided
$S\setminus\gamma_1$  and  $S\setminus\gamma_2$ are connected). It is
slightly  less obvious to check whether saddle connections $\gamma_1,
\gamma_2$  on  a flat surface $S$ with nontrivial linear holonomy are
\^homologous   or  not.  In  particular,  a  pair  of  closed  saddle
connections   might  be  homologous  in  the  usual  sense,  but  not
\^homologous;   a   pair   of  closed  saddle  connections  might  be
\^homologous  even  if  one  of  them represents a loop homologous to
zero,  and the other does not; finally, a saddle connection joining a
pair  of  {\it  distinct}  singularities  might  be \^homologous to a
saddle connection joining a singularity to itself, or joining another
pair  of  distinct  singularities. The following statement provides a
geometric  criterion  for  deciding  when  two saddle connections are
\^homologous.

\begin{NNProposition}[H.~Masur, A.~Z.]
Let  $S$  be  a  flat  surface  corresponding  to  a  meromorphic
quadratic differential $q$  with at most simple poles. Two saddle
connections $\gamma_1, \gamma_2$  on  $S$ are \^homologous if and
only  if  they have  no  interior intersections  and  one of  the
connected    components    of    the    complement    $S\setminus
(\gamma_1\cup\gamma_2)$ has trivial linear holonomy. Moreover, if
such a component exists, it is unique.
\end{NNProposition}

Now everything is ready for the proof of Proposition~\ref{prop:only:four:types}.

\begin{proof}[Proof of Proposition~\ref{prop:only:four:types}]
Configurations I and II involve a single saddle connection. Using the
above  criterion it is immediate to check that all saddle connections
involved  in  configurations III and IV are \^homologous. Thus, these
configurations  are  rigid, and we can find them on almost every flat
surface in the stratum.

Theorem~2.2  in~\cite{Boissy:configurations}  applies general results
from~\cite{Masur:Zorich}  to  classify all possible configurations of
\^homologous saddle connections on $\CP$, and shows that there are no
such configurations different from types I--IV.

To   complete   the   proof   it   remains   to  apply  Proposition~4
from~\cite{Masur:Zorich}  which  claims  that  for  almost every flat
surface  in  any  stratum, two saddle connections are parallel if and
only  if  they  are \^homologous. This statement is proved following
the  lines  of Proposition 7.4 in~\cite{Eskin:Masur:Zorich}; see also
the analogous proof of Proposition~\ref{pr:no:other:for:billiards} in
~\S\ref{sec:subsec:The:subspace:of:billiards} below.
\end{proof}

%----------------------------------------------------------------
\subsection{The subspace of billiards.}
\label{sec:subsec:The:subspace:of:billiards}

Consider  now   the  map
$$
\cB(k_1,\dots,k_n)\into\cQ(k_1-2,\dots,k_n-2)\,.
$$
In       the       chosen       coordinates       in      $H^1_-(\hat
S,\{\hat{P}_1,\ldots,\hat{P}_N\};\C{})$  the  image  of a directional
billiard $\Pi$ is presented by a point
\begin{multline}
\left(2\int_{P_1}^{P_2} dz, \ldots, 2\int_{P_{n-2}}^{P_{n-1}} dz\right)=
\\
\left(2|P_1 P_2|e^{i\phi},
2|P_2 P_3|e^{(k_2\pi)/2+i\phi},
\ldots,
2|P_{n-3} P_{n-2}|e^{(k_2+\ldots+k_{n-2})\pi/2+i\phi}\right)\,.
\end{multline}
The  components  of  the projection of this vector to the $H^1_-(\hat
S,\{\hat{P}_1,\ldots,\hat{P}_N\};\R{})$ are of the form
$$
\pm  2\sin(\phi)  |P_i P_{i+1}|\quad \mbox{ or }\quad \pm
2\cos(\phi)    |P_i   P_{i+1}|
$$
depending  on  the  parity  of  $k_2+\ldots+k_i$.  Thus,  for  $\phi$
different  from  an  integer  multiple of $\pi/2$ the composition map
$T_\ast\cB  \to H^1_-(\hat S,\{\hat{P}_1,\ldots,\hat{P}_N\};\R{})$ is
a surjective map. We have proved

\begin{proposition}
\label{prop:map:onto}
Consider        the        canonical       local        embedding
$$
\cB(k_1,\ldots,k_n) \into \mathcal{Q}(k_1-2,\ldots,k_n-2).
$$
For      almost      all       directional      billiards      in
$\cB(k_1,\ldots,k_n)$ the projection of the tangent  space
$T_\ast\cB(k_1,\ldots,k_n)$  to  the unstable  subspace of
the Teichm\"uller geodesic flow is a surjective map.
\end{proposition}
\medskip

% Using  standard  ergodic  theoretic  arguments on unstable foliations
% (see, e.g., \cite{Margulis, Eskin:McMullen}), we have the following:

% \begin{corollary}
% \label{cor:birkhoff:average}  Let $\mu_{\cB}$ denote Lebesgue measure
% (arising  from  lengths  of  sides  and  the  angle $\phi$) on $\cB =
% \cB(k_1,  \ldots,  k_n)$, and let $\mu_{\cQ}$ denote Lebesgue measure
% on  $\cQ(k_1-2,  \ldots,  k_n-2)$.  Then for $\mu_{\cB}$-almost every
% $\Pi  \in  \cB$, Lebesgue almost every $\theta \in S^1$, we have, for
% any $f \in C_c(\mu_{\cQ})$,
%    %
% \begin{equation}\label{eq:birkhoff:ave}
% \lim_{T \rightarrow \infty}\frac{1}{T}\int_{0}^{T} f(g_t r_{\theta} q_{\Pi})dt
% = \frac{1}{\mu_{\cQ}(\cQ)} \int_{\cQ} f d\mu_{\cQ}\,.
% \end{equation}
%    %
% Moreover,  for  any  set  of  $\mu_{\cB}-$positive measure $A \subset
% \cB$, we have
%    %
% \begin{equation}\label{eq:leaf:ave}
% \lim_{T \rightarrow \infty} \frac{1}{2\pi\mu_{\cB}(A)} \int_A \int_0^{2\pi} f(g_T r_{\theta} q_{\Pi}) d\mu_{\cB}(\Pi)
% =  \frac{1}{\mu_{\cQ}(\cQ)}  \int_{\cQ} f d\mu_{\cQ}\,.
% \end{equation}
%    %
% \end{corollary}

We      complete     this     section     with     a     proof     of
Proposition~\ref{pr:no:other:for:billiards}.

\begin{proof}[Proof of Proposition~\ref{pr:no:other:for:billiards}]
By  assumption  we  do  not consider generalized diagonals and closed
billiard  trajectories  parallel  to  the sides of the polygon. First
note that without loss of generality we can consider only generalized
diagonals:  any  closed regular trajectory makes part of a band which
is  bounded  on both sides by a (chain of) generalized diagonals, see
Figure~\ref{fig:family:of:rectangular:polygons}.

Let  $l_m=|P_m  P_{m+1}|$  for $m=1,\dots,n-2$. Recall that $l_i$ are
the   independent  coordinates  in  the  space  $\cB(k_1,\dots,k_n)$.
Unfolding the billiard along a generalized diagonal we see that every
generalized  diagonal  (non  parallel  to  one  of  the  sides of the
polygon) defines a relation
$$
\frac{\sum a_i l_i}{\sum  b_j l_j}=\tan(\phi)\,,
$$
where  $0<\phi<\pi/2$;  the  sum  in  the numerator is taken over the
vertical  sides  of  the polygon; the sum in the denominator is taken
over  the  horizontal  sides;  and  all $a_i$ and $b_j$ are integers.
Since the second generalized diagonal has a segment going in the same
direction $\phi$, it also defines a relation
$$
\frac{\sum c_i l_i}{\sum  d_j l_j}=\tan(\phi)\,,
$$
where  the  sum  in the numerator is taken over the vertical sides of
the  polygon; the sum in the denominator is taken over the horizontal
sides; and all $c_i$ and $d_j$ are integers.

Each  generalized diagonal determines a saddle connection $\gamma$ on
the  corresponding  flat  sphere,  which  in  turn  defines  a  cycle
$\pm\hat\gamma\in   H_1^-(\hat  S,\{\hat  P_1,\dots,\hat  P_N\};\Z)$.
Moreover,  up  to appropriate choice of signs of the basic vectors in
the  basis  from  ~\S\ref{sec:subsec:coordiantes:in:stratum}  the  cycle
corresponding  to  the  first generalized diagonal has the form $\hat
c_1:=\sum  a_i  \hat  \gamma_i  +  \sum b_j \hat\gamma_j$ and the
cycle  corresponding  to the second generalized diagonal has the form
$\hat c_2:=\sum c_i \hat \gamma_i + \sum d_j \hat\gamma_j$.

Assume  that the two generalized diagonals do not make part of any of
configurations        I--IV.       By       the       result       of
Boissy~\cite{Boissy:configurations}  there  are  no configurations of
\^homologous  saddle  connections  on $\CP$ other than configurations
I--IV. This implies that the corresponding saddle connections are not
\^homologous,  and,  hence,  the cycles $\hat c_1$ and $\hat c_2$ are
not proportional. This implies that the relation
$$
\frac{\sum a_i l_i}{\sum  b_j l_j}=\frac{\sum c_i l_i}{\sum  d_j l_j}
$$
is  a  nontrivial  relation on coordinates $l_1,\dots,l_{n-2}$. Thus,
the  set, satisfying this condition, has measure zero. Taking a union
over  the  countable  collection  of  possible conditions (countable,
because  we  have  to  consider  all possible collections of integers
$a_i,b_j,c_i,d_j$) we still get a set of measure zero.
\end{proof}

%####################################################################
%####################################################################
%####################################################################

\section{Values of the Siegel--Veech constants}
\label{sec:Siegel:Veech}

In  this section, we derive formulas for the Siegel--Veech constant of
each  configuration  of  saddle  connections.  There are two kinds of
formulas.  The  first  kind  expresses the Siegel--Veech constant as a
ratio of volumes of strata, with explicit combinatorial coefficients.
These  formulas  will be stated and proved in this section. The second
kind   of   formula  gives  the  Siegel--Veech  constants  as  numbers
(depending  only  on  the  stratum  and  the configuration). They are
proved    by    plugging    the   expression~\eqref{eq:volume}   from
Theorem~\ref{theorem:volume}  into  the formula of the first kind. We
also       present       these      formulas      here;      however,
Theorem~\ref{theorem:volume}     will     only     be    proved    in
\S\ref{sec:induction}.  For this reason we have attempted to separate
the  formulas  which  depend on Theorem~\ref{theorem:volume} from the
formulas which do not.

The results obtained in this section are based on techniques developed
in  the    papers~\cite{Eskin:Masur},       \cite{Eskin:Masur:Zorich},
and~\cite{Masur:Zorich} written in collaboration with H.~Masur.

%----------------------------------------------------------------
\subsection{Normalization of the volume element.}
\label{sec:subsec:notmalization}

Recall   that   for   any   flat   surface   $S$   in   any   stratum
$\cQ(d_1,\dots,d_k)$  we  have  a  canonical  ramified  double  cover
$\hat{S}\to  S$  such  that the induced quadratic differential on the
Riemann  surface $\hat S$ is a global square of a holomorphic Abelian
differential. We have seen in ~\S\ref{sec:subsec:coordiantes:in:stratum}
that the subspace
$H^1_-(\hat S,\{\hat{P}_1,\ldots,\hat{P}_N\};\C{})$
antiinvariant with respect to the induced action of the hyperelliptic
involution  on  relative cohomology provides local coordinates in the
corresponding     stratum     $\cQ(d_1,\dots,d_n)$    of    quadratic
differentials.     We     define    a    lattice    in    $H^1_-(\hat
S,\{\hat{P}_1,\ldots,\hat{P}_N\};\C{})$ as the subset of those linear
forms  which  take  values  in $\Z\oplus i\Z$ on $H^-_1(\hat S,\{\hat
P_1,\dots,\hat P_N\};\Z)$.

We  define  the  volume element $d\mu$ on $\cQ(d_1,\dots,d_k)$ as the
linear      volume      element      in      the     vector     space
$H^1_-(\hat{M}^2_g,\{\hat{P}_1,\ldots,\hat{P}_N\};\C{})$
normalized  in  such  way  that  the  fundamental domain of the above
lattice has volume $1$.

We warn the reader that for $N>1$ this lattice is a proper sublattice
of index $4^{N-1}$ of the lattice
$$
H^1_-(\widehat S,\{\widehat P_1, \dots, \widehat P_N\};\cx)\ \cap\
H^1(\widehat S,\{\widehat P_1, \dots, \widehat P_N\};\Z\oplus i\Z)\,.
$$
Indeed,  if a flat surface $S$ defines a lattice point for our choice
of  the  lattice,  then the holonomy vector along a saddle connection
joining  distinct  singularities can be half-integer. (However, the
holonomy  vector along any \textit{closed} saddle connection is still
always integer.)

The  choice  of one or another lattice is a matter of convention. Our
choice  makes  formulae  relating enumeration of pillowcase covers to
volumes  simpler;  see  Appendix~\ref{sec:pillowcase:covers}. Another
advantage   of   our  choice  is  that  the  volumes  of  the  strata
$\cQ(d,-1^{d+4})$   and   of  the  hyperelliptic  components  of  the
corresponding strata of Abelian differentials are the same (up to the
factors responsible for the numbering of zeroes and of simple poles).

\begin{Convention}
\label{con:area:1:2}
Similar  to  the  case  of Abelian differentials we now choose a real
hypersurface $\cQ_1(d_1,\dots,d_k)$ of flat surfaces of fixed area in
the  stratum  $\cQ(d_1,\dots,d_k)$.
We abuse notation by denoting by
$\cQ_1(d_1,\dots,d_k)$  the  space of flat surfaces of area $1/2$ (so
that the canonical double cover has area $1$).
\end{Convention}

The volume element $d\mu$ in the embodying space $\cQ(d_1,\dots,d_k)$
induces  naturally  a  volume  element  $d\mu_1$  on the hypersurface
$\cQ_1(d_1,\dots,d_k)$  in  the  following  way.  There  is a natural
$\cx^\ast$-action        on        $\cQ(d_1,\dots,d_k)$:       having
$\lambda\in\cx^\ast$ we associate to the flat surface $S=(\CP,q)$ the
flat  surface
\begin{equation}
\label{eq:Cstar:action}
\lambda\cdot S:=(\CP,\lambda^2\cdot q)\,.
\end{equation}
In  particular, we can represent any $S\in\cQ(d_1,\dots,d_k)$ as $S =
r  S_{(1)}$, where $r\in\reals_+$, and where $S_{(1)}$ belongs to the
``hyperboloid'': $S_{(1)}\in\cQ_1(d_1,\dots,d_k)$. Geometrically this
means that the metric on $S$ is obtained from the metric on $S_{(1)}$
by  rescaling  with  linear  coefficient  $r$. In particular, vectors
associated  to  saddle connections on $S_{(1)}$ are multiplied by $r$
to  give  vectors  associated  to corresponding saddle connections on
$S$.  It  means  also that $\area(S) = r^2\cdot\area(S_{(1)})=r^2/2$,
since  $\area(S_{(1)})  = 1/2$. We define the \textit{volume element}
$d\mu_1$    on    the   ``hyperboloid''   $\cQ_1(d_1,\dots,d_k)$   by
disintegration of the volume element $d\mu$ on $\cQ(d_1,\dots,d_k)$:
\begin{equation}
\label{eq:disintegration}
d\mu = r^{2n-1} \, dr\, d\mu_1\, ,
\end{equation}
where
$$
2n=\dim_\reals\cQ(d_1,\dots,d_k)=
2\dim_{\C{}}\cQ(d_1,\dots,d_k)=2(k-2)\,.
$$
Using  this  volume element we define the total \textit{volume of the
stratum} $\cQ_1(d_1,\dots,d_k)$:
\begin{equation}
\label{eq:int:cF:nu1}
\Vol\cQ_1(d_1,\dots,d_k):= \int_{\cQ_1(d_1,\dots,d_k)}d\mu_1\,.
\end{equation}

For     a     subset     $E\subset\cQ_1(d_1,\dots,d_k)$     we    let
$C(E)\subset\cQ_1(d_1,\dots,d_k)$ denote the ``cone'' based on $E$:
\begin{equation}
\label{eq:cone}
 C(E):=\{S=rS_{(1)}\,|\, S_{(1)}\in E,\ 0<r\le 1\}\,.
\end{equation}
Our  definition  of  the  volume element on $\cQ_1(d_1,\dots,d_k)$ is
consistent with the following normalization:
\begin{equation}
\label{eq:normalization}
\Vol(\cQ_1(d_1,\dots,d_k)) =
\dim_{\R{}} \cQ(d_1,\dots,d_k)\cdot\mu(C(\cQ_1(d_1,\dots,d_k))\,,
\end{equation}
where  $\mu(C(\cQ_1(d_1,\dots,d_k))$  is  the  total  volume  of  the
``cone''  $C(\cQ_1(d_1,\dots,d_k))\subset\cQ(d_1,\dots,d_k)$ measured
by means of the volume element $d\mu$ on $\cQ(d_1,\dots,d_k)$ defined
above.

\subsection{$\SL$-action}\label{subsec:sl2:affine}

There  is  an  action  of  $\SL$  on  the  moduli  space of quadratic
differentials   that  preserves  the  stratification,  and  moreover,
preserves (~\cite{Masur:interval, Veech:Gauss}) the measures on $\cQ$
and $\cQ_1$ described above. Recall that a quadratic differential $q$
determines  (and  is determined by) by an atlas of charts to $\mathbb
C$  whose  transition maps are of the form $z \mapsto \pm z+c$. Since
$\SL$  acts  on  $\mathbb  C$ via linear maps on $\R^2$, given a
quadratic differential $q$ and a matrix $g \in \SL$, define the
quadratic differential $g\cdot q$ via post-composition of charts with
$g$. This action generalizes the action of $\SL$ on the space of
(unit-area)  flat  tori  $\SL/\SLZ$.  Note  that $\SL$
preserves  the  area  of  the  quadratic  differential  $q$,  and  in
particular it preserves the level surface $\cQ_1(d_1,\dots,d_k)$.

%----------------------------------------------------------------
\subsection{Strata of surfaces with marked points}
In  this section we shall also consider the strata $\cQ_1(\alpha)$ of
flat  surfaces  $S=(\CP,q)$  where  we  mark  a  regular point on the
surface.  Say,  $\cQ_1(2,1^2,0,-1^8)$  will  denote  the  stratum  of
meromorphic  quadratic  differentials on $\CP$ with one zero of order
$2$,  two  zeroes  of  order $1$ denoted by $1^2$,
eight simple poles $-1^8$,
and one additional marked point: ``zero of order $0$''.

Let  $\alpha=\{d_1,\dots,d_k\}$  be  a set with multiplicities, where
$d_i\in\{-1,  1,  2,  3,  \dots\}$ for $i=1, \dots, k$, and $\sum d_i
=-4$. A stratum with a marked point $\cQ(0, d_1, \dots, d_k)$ has the
natural  structure  of  a fiber bundle over the corresponding stratum
without marked points $\cQ(d_1, \dots, d_k)$. This bundle has the
surface  $S$  (punctured at all singularities $P_1, \dots, P_k$) as a
fiber  over  the  ``point''  $S\in\cQ(d_1, \dots, d_k)$. Clearly, the
dimension  of  the  ``universal  curve''  $\cQ(0,  d_1,  \dots, d_k)$
satisfies
\begin{equation}
\label{eq:dim:marked:points}
\dim_\cx\cQ(0,d_1\dots,d_k)=
\dim_\cx\cQ(d_1\dots,d_k)+1=k-1\,.
\end{equation}
By convention  we {\it  always} mark a point on  a flat torus. We
denote the corresponding stratum $\cH(0)$; it has dimension two:
$\dim_\cx\cH(0)=2$.

The  natural measure on the stratum $\cQ(0,d_1\dots,d_k)$ with marked
points  disintegrates  into  a  product  measure,  where  the measure
$d\mu_0$  along  the fiber is proportional to the Lebesgue measure on
$S$  induced  by  the flat metric on $S$, and the measure on the base
$\cQ(d_1\dots,d_k)$   is   the   natural   measure  $d\mu_1$  on  the
corresponding stratum taken without marked points.

When  the  flat  structure  on $S$ is defined by a \textit{quadratic}
differential  the  measure  of  the  fiber  $S$ is different from the
measure of the analogous fiber $S$ with the flat structure defined by
an        \textit{Abelian}       differential.       Namely,       by
Convention~\ref{con:area:1:2} the area of the surface $S$ in terms of
our  flat metric defined by the quadratic differential is $1/2$. Note
also,  that  a saddle connection $\gamma$ joining a zero and a marked
point      and      having      half-integer      linear     holonomy
$\pm\mathit{hol}(\gamma)\in\R^{2}$    defines    an   integer   cycle
$\hat\gamma\in  H_1^-(\hat  M^2_g,\{\hat  P_1,\dots, \hat P_N\};\Z)$.
Hence,  our  choice  of  the fundamental domain of the lattice in the
relative   cohomology   $H^1_-(\hat   M^2_g,\{\hat   P_1,\dots,  \hat
P_N\};\C{})$   described   in  ~\S\ref{sec:subsec:notmalization}
implies that the component $d\mu_0$ of the disintegrated measure along
the fiber $S$ is
\begin{equation}
\label{eq:4dxdy}
d\mu_0 = 4dx\hspace{1pt}dy\,,
\end{equation}
i.e.  $4$  times  the  standard Lebesgue measure coming from the flat
metric.  This gives $\mu_0(S)=2$ for the total measure of each fiber,
which  implies  the  following  relation  between  the volumes of the
strata:
\begin{equation}
\label{eq:vol:with:marked:points}
\Vol \cQ_1(0, d_1, \dots, d_k) = 2 \Vol \cQ_1(d_1, \dots, d_k)\,.
\end{equation}
Recall  that  $v(0)=2$,  see~\eqref{eq:v};  so  this is coherent with
formula~\eqref{eq:volume} for the volume.

%-------------------------------------------------------------------
\subsection{Volume of a stratum of disconnected flat surfaces}
\label{sec:subsec:Strata:of:Disconnected:Surfaces}
It     will     be     convenient     to    consider    the    strata
$\cQ(\alpha')=\cQ(\alpha'_a)\times\cQ(\alpha'_b)$,   of  closed  flat
surfaces  $S$  having  two  components  $S_a\sqcup S_b$ of prescribed
types.  Such  strata play especially important role in the context of
the        \textit{principal       boundary}       discussed       in
~\S\ref{sec:subsec:principal:boundary}.  In the consideration below each
of  $\alpha'_a,  \alpha'_b$ might contain an entry ``$0$'' or not. In
other  words, the strata $\cQ(\alpha'_a), \cQ(\alpha'_b)$ are allowed
to have a marked point.

\begin{Convention}
\label{conv:disconn:strata}
Using  notation  $\alpha'=\alpha'_a\sqcup\alpha'_b$  for  the  strata
$\cQ(\alpha')$  of disconnected surfaces we assume that we keep track
of how $\alpha'$ is partitioned into $\alpha'_a$ and $\alpha'_b$.
\end{Convention}

We  shall  need  the  expressions  for the volume element and for the
total  volume  of  such strata. The corresponding expressions for the
strata  of  Abelian  differentials were obtained in \S 6.2 pp. 81--82
in~\cite{Eskin:Masur:Zorich}.   Though   the   corresponding  formula
translates  to  the  strata  of  quadratic  differentials without any
difficulties  we  present  this  simple  calculation since it is very
instructive   in  view  of  calculation  of  Siegel--Veech  constants
performed below.

We  write  $S_i  =  r_i  S_i^{(1)}$, where
$area\big(S_i^{(1)}\big)    =   \frac{1}{2}$;   $i\in\{a,b\}$.   Then
$area(S_i) = r_i^2\cdot\frac{1}{2}$. Let
$$
n_i:=\dim_\C{}\cQ(\alpha'_i);\qquad
n:=\dim_\C{}\cQ(\alpha')=n_a+n_b\,.
$$
Let  $d\mu^a$ (correspondingly $d\mu^b$) be the volume element on the
stratum   $\cQ(\alpha'_a)$  (correspondingly  $\cQ(\alpha'_b)$).  Let
$d\mu^a_1$  (correspondingly  $d\mu^b_1$)  be the hypersurface volume
element    on    the    ``unit    hyperboloid''    $\cQ_1(\alpha'_a)$
(correspondingly $\cQ_1(\alpha'_b)$). We have
$$
d\mu(S) =d\mu^a(S_a)\cdot d\mu^b(S_b)  =
r_a^{2n_a-1} r_b^{2n_b-1} dr_a\,dr_b\, d\mu^a_1\,d\mu^b_1\,.
$$
Set
$$
W=\Vol\cQ_1(\alpha'_a)\cdot\Vol\cQ_1(\alpha'_b)\,.
$$
Then,
$$
\mu(C(\cQ_1(\alpha'))  =
W\cdot
\int_{\substack{r_a^2+r_b^2\le 1\\r_a> 0;\ r_b>0}}
r_a^{2n_a - 1} r_b^{2n_b - 1}\,dr_a\,dr_b=
W\cdot\frac{1}{4}\frac{ (n_a-1)! (n_b-1)!}{n!}\,,
$$
where  we  have left the computation of the integral over the disk as
an exercise. Hence, applying~\eqref{eq:normalization} we get
\begin{multline}
\label{eq:total:volume:of:nonprimitive:stratum}
\Vol\cQ_1(\alpha')=2n\cdot \mu(C(\cQ_1(\alpha') )
=\\=
\frac{1}{2}\cdot
\frac{(\dim_\cx\cQ(\alpha'_a)-1)! (\dim_\cx\cQ(\alpha'_b)-1)!}
{(\dim_\cx\cQ(\alpha')-1)!}\cdot
\Vol\cQ_1(\alpha'_a)\cdot\Vol\cQ_1(\alpha'_b)\,.
\end{multline}

Repeating  literally   the   same   arguments   we   obtain  the
corresponding formula for the volume elements:
\begin{equation}
\label{eq:volume:element:of:nonprimitive:stratum}
d\mu_1=
\frac{1}{2}\cdot
\frac{(\dim_\cx\cQ(\alpha'_a)-1)! (\dim_\cx\cQ(\alpha'_b)-1)!}
{(\dim_\cx\cQ(\alpha')-1)!}\cdot
d\mu^a_1\,d\mu^b_1\,.
\end{equation}

%--------------------------------------------------------------------
\subsection{Reduction to ergodic theory}
\label{sec:subsec:Reduction:to:ergodic:theory}

In this section we recall the strategy given in~\cite{Eskin:Masur} to
obtain the quadratic asymptotics in Theorem~\ref{theorem:SV:for:CP}.

Fix  an  unordered collection
$(d_1,   \ldots,   d_n)$  of  integers  $d_i  \in  \N  \cup  \{-1\}$,
$i=1,\dots,n$,  satisfying  $\sum_{i=1}^n  d_i = -4$, and let $\cQ_1$
denote  the  stratum  $\cQ_1(d_1, \ldots, d_n)$. Note that every such
stratum  is  nonempty and connected. Let $\mu_1$ denote the canonical
$\PSL$-invariant  measure on $\cQ_1$. Fix a configuration $\cC$ as in
\S\ref{sec:subsec:configurations}.   To  each  saddle  connection  we
associate  a  holonomy  vector in the Euclidean plane $\R^{2}$ having
the same length and the same line direction as the saddle connection.
By  convention  the  configuration  III  is represented by the closed
saddle  connection  joining  a  zero  to  itself (the holonomy vector
associated  to the partner saddle connection joining two simple poles
is  parallel  but  twice  shorter).  Since  by  convention the saddle
connections  are not oriented, the holonomy vector is defined up to a
sign,  so  we actually consider a pair of opposite holonomy vectors
$\pm\vec  v$.  Given  a  flat surface $S=(\CP,q)  \in \cQ_1$, let
$V_\cC(S)$ be the set of holonomy vectors of saddle connections whose
configuration  is  $\cC$. For any flat surface $S$
the set $V_\cC(S)$ is a discrete subset of
$\R^2$. We are interested in the asymptotics of the number
\begin{equation}
\label{eq:N:C:S:L}
N_\cC(S, L) = \frac{1}{2}\big|V_\cC(S) \cap B(0, L)\big|\,,
\end{equation}
of saddle connections of type $\cC$ on the flat surface $S$ of length
at most $L$. The weight $1/2$ in the above expression compensates the
fact  that  each  saddle  connection  is  represented by two holonomy
vectors $\pm\vec v$.

In  the  remainder  of \S\ref{sec:ergodic}, the stratum $\cQ$ and the
configuration  $\cC$  are  fixed.  We  will often omit $\cC$ from the
notation,  and  we will use the abbreviated notation $q$ for the flat
surface $S=(\CP,q)$.

%--------------------------------------------------------------------
\subsubsection{Siegel--Veech formulas}
\label{Siegel:Veech:formulas}
\noindent Given $f \in C_c(\R^2)$, define the \textit{Siegel--Veech}
transform $\widehat{f}:\cQ_1 \rightarrow \R$ by

\begin{equation}
\label{eq:siegel:veech:def}
\widehat{f}(q) = \frac{1}{2}\sum_{v \in V_\cC(q)} f(v)\,.
\end{equation}

\noindent We have the \textit{Siegel--Veech formula}
(\cite{Veech:siegel}, Theorem 0.5). There is a constant (called the
\textit{Siegel--Veech} constant) $b_\cC(\cQ)$ so that:

\begin{equation}
\label{eq:siegel:veech}
\frac{1}{\mu_1(\cQ_1)}\int_{\cQ_1} \widehat{f}(q)d\mu_1(q) =
b_\cC(\cQ_1)\int_{\R^2} f(x,y)\, dx dy\,.
\end{equation}

\noindent  Let
$$g_t  = \left( \begin{array}{cc} e^{t}  & 0 \\ 0 & e^{-t}
\end{array}\right)$$

$$r_{\theta} =\left( \begin{array}{cc} \cos
\theta & \sin \theta \\ -\sin \theta & \cos \theta \end{array}\right).$$

\noindent Let $f$ be (a smoothed version of) the indicator function of the trapezoid $\mathcal T$ defined by the points
$$(1,1), (-1, 1), (1/2, 1/2), (-1/2, 1/2).$$

\noindent Note that the area of this trapezoid is $3/4$.

We then have, for $t \gg 0$, and any $v \in \R^2$ (\cite{Eskin:Masur}, Lemma 3.4):

\begin{equation}
\label{eq:eskinmasur:3.4}
\frac{1}{2\pi}\int_0^{2\pi} f(g_t r_{\theta} v) d\theta \approx \begin{cases}
\frac{e^{-2t}}{\pi} & e^{t}/2 \le \|v\| \le e^t\\ 0 & \mbox{otherwise}
\end{cases}.
\end{equation}
(See \cite{Eskin:Masur} for the exact meaning of $\approx$). Heuristically, the integral measures the proportion of angles $\theta$ so that $r_{\theta} v \in g_{-t} \mathcal T$. The trapezoid $g_{-t} \mathcal T$ has vertices at $$(e^{-t},e^t), (-e^{-t}, e^t), (e^{-t}/2, e^t/2), (-e^{-t}/2, e^t/2).$$ The range of (inverse) slopes is of size $2e^{-2t}$, and thus the length of the interval of $\theta$'s satisfying $r_{\theta} v \in g_{-t} \mathcal T$ is also of size $2e^{-2t}$, if $v$ has length in between $e^t/2$ and $e^t$, and zero otherwise. Dividing by $2\pi$ to get the proportion, we obtain (\ref{eq:eskinmasur:3.4}).
Combining (\ref{eq:siegel:veech}) and (\ref{eq:eskinmasur:3.4}), we obtain
\begin{equation}
\label{eq:circ:av}
\frac{e^{2t}}{2\pi} \int_{0}^{2\pi} \widehat{f}(g_t r_{\theta} q) d\theta \approx \frac{1}{\pi} \left(
N_\cC(q, e^{t}) - N_\cC(q, e^t/2)\right).
\end{equation}

\subsubsection{Equidistribution results}

The equation (\ref{eq:circ:av}) reduces the problem of studying
\begin{displaymath}
\lim_{t \rightarrow \infty} e^{-2t} N_\cC(q, e^t)
\end{displaymath}
to that of studying the limiting behavior of
\begin{displaymath}
\frac{1}{2\pi}\int_{0}^{2\pi} \widehat{f}(g_t r_{\theta} q) d\theta
\end{displaymath}
Assuming this limit exists, and is equal to $c$, a geometric series
calculation shows
\begin{displaymath}
\lim_{t \rightarrow \infty}e^{-2t} N_\cC(q, e^t) = \frac{4}{3} \pi c.
\end{displaymath}
Assuming further that Lebesgue measure supported on the circles $\{g_t
r_{\theta} q\}_{0 \le \theta < 2\pi}$ converges, as $t \rightarrow
\infty$, to
the absolutely continuous $\SL$-invariant
measure $\mu_1$ on $\cQ_1$, we would have that $c =
\frac{1}{\mu_1(\cQ_1)}\int_{\cQ_1} \widehat{f}(q)d\mu_1(q)$, and then using
(\ref{eq:siegel:veech}), we would obtain, since the area of the
trapezoid is $3/4$,
\begin{displaymath}
\lim_{t \rightarrow \infty}e^{-2t} N_\cC(q, e^t) = \pi b_\cC(\cQ).
\end{displaymath}
In fact, this is the approach used in~\cite{Eskin:Masur}. There, the
key tool is a general ergodic theorem on $\SL$-actions, proved
by A.~Nevo~\cite{Nevo} which shows
\begin{displaymath}
\lim_{t \rightarrow \infty}
\int_{0}^{2\pi} \widehat{f}(g_t r_{\theta} q)\,d\theta =
\frac{1}{\mu_1(\cQ_1)}\int_{\cQ_1} \widehat{f}(q)\,d\mu_1(q),
\end{displaymath}
for almost every $q \in \cQ$. However, since the set of billiards has
measure $0$, this does not yield any information about them. We will
instead use Theorem~\ref{theorem:birkhoff:chaika} to
obtain our results.

%--------------------------------------------------------------------
\subsection{Siegel--Veech constants and the principal boundary of strata}
\label{sec:subsec:principal:boundary}

In  this  section  we  present a strategy for evaluation Siegel--Veech
constants.     This     strategy     was     successfully     applied
in~\cite{Eskin:Masur:Zorich}  to  compute all Siegel--Veech constants
for  all connected components of the strata of Abelian differentials.
In   this   section   we   present   the  general  scheme  elaborated
in~\cite{Eskin:Masur:Zorich} and developed in~\cite{Masur:Zorich}. In
the further sections we adjust it to the concrete cases of configurations
of       saddle       connections       I--IV       described      in
~\S\ref{sec:subsec:configurations}.

Fix a stratum $\cQ(\alpha)$ of meromorphic quadratic differentials on
$\CP$,  where  $\alpha=\{d_1,\dots,d_k\}$.  Consider  a configuration
$\cC$  of  one  of  the types I--IV (in the case of general strata in
higher  genus  it  would  be any configuration of \^homologous saddle
connections).           We           have           seen           in
~\S\ref{sec:subsec:Reduction:to:ergodic:theory}   that   to   each  flat
surface   $S\in\cQ(\alpha)$   we  can  associate  a  discrete  subset
$V_\cC(S)\subset  \R^{2}$  of  holonomy vectors of saddle connections
whose  configuration  is $\cC$. By construction the set $V_\cC(S)$ is
centrally  symmetric  with respect to the origin. To any function $f$
with  compact  support  on $\R^2$ formula~\eqref{eq:siegel:veech:def}
associates  its \textit{Siegel--Veech transform} $\widehat{f}$ defined
on  the  stratum  $\cQ$.  By  definition~\eqref{eq:siegel:veech:def},
choosing  the  characteristic  function  $\chi_L$ of a closed disk of
radius $L$ centered at the origin of $\R^2$ as a function $f$, we get
as  $\widehat\chi_L(S)$  the  counting  function $N_\cC(S, L)$ of the
number  of saddle connections of type $\cC$ and of length at most $L$
on the flat surface $S$ defined by~\eqref{eq:N:C:S:L}.

Applying Siegel--Veech formula~\eqref{eq:siegel:veech} we obtain
\begin{equation}
\label{eq:SV:formula:for:chi}
\frac{1}{\mu_1(\cQ_1)}\int_{\cQ_1} \widehat\chi_L(S)\,d\mu_1(S) =
b_\cC(\cQ)\int_{\R^2} \chi_L(x,y)\, dx dy=
b_\cC(\cQ)\cdot\pi L^2\,.
\end{equation}

By  the  results  of  A.~Eskin  and  H.~Masur~\cite{Eskin:Masur}, for
almost all flat surfaces $S$ in the stratum $\cQ_1$ one has
\begin{equation}
\label{eq:N:equals:b:pi:L:2}
N_\cC(S, L)=\widehat{\chi}_L(S)\sim b_\cC\cdot \pi L^2
\end{equation}
with  the  same constant $b_\cC$ as in~\eqref{eq:SV:formula:for:chi}.

Formula~\eqref{eq:SV:formula:for:chi}     can     be    applied    to
$\widehat{\chi}_L$  for  any  value of $L$. In particular, instead of
taking  large  $L$  we can choose a very small $L=\varepsilon \ll 1$.
The corresponding function $\widehat{\chi}_\varepsilon(S)$ counts how
many  (collections of) $\varepsilon$-short saddle connections (closed
geodesics) of the type $\cC$ we can find on a flat surface $S\in\cQ$.

Consider  a subset $\cQ_1^{\varepsilon}(\cC)\subset\cQ_1$ of surfaces
of  area $1/2$ having a saddle connection shorter than $\varepsilon$.
Consider                a                smaller               subset
$\cQ_1^{\varepsilon,thin}\subset\cQ_1^{\varepsilon}$     of     those
surfaces   of   area  $1/2$  in  $\cQ_1$  which  have  at  least  two
distinct collections of \^homologous
saddle  connections  of  type $\cC$ and of length at
most  $\epsilon$.  Finally, define $\cQ_1^{\varepsilon,thick}$ as the
complement $\cQ_1^{\varepsilon}-\cQ_1^{\varepsilon,thin}$.

For    the    flat    surfaces    $S$    outside    of   the   subset
$\cQ^{\varepsilon}_1(\cC)$  there  are  no  saddle connections of the
type        $\cC$       shorter       than       $\epsilon$,       so
$\widehat{\chi}_\varepsilon(S)=0$ for such surfaces. For surfaces $S$
from the subset $\cQ^{\varepsilon,thick}_1(\cC)$ there is exactly one
collection  like this, so $\widehat{\chi}_\varepsilon(S)=1$. Finally,
for     the     surfaces    $S$    from    the    remaining    subset
$\cQ^{\varepsilon,thin}_1(\cC)$                one                has
$\widehat{\chi}_\varepsilon(S)\ge  1$.  A,~Eskin  and  H.~Masur  have
proved           in~\cite{Eskin:Masur}           that          though
$\widehat{\chi}_\varepsilon(S)$  might be large on $\cQ^{\varepsilon,
thin}_1$ the measure of this subset is so small that
$$
\int_{\cQ^{\varepsilon, \mathit{thin}}_1(\cC)}
\widehat{\chi}_\varepsilon(S)\ d\mu_1 = o(\varepsilon^2)
$$
and hence
$$
\int_{\cQ_1} \widehat{\chi}_\varepsilon(S)\ d\mu_1 =
\Vol\cQ^{\varepsilon, thick}_1(\cC)\  +\ o(\varepsilon^2)\,.
$$
This   latter   volume   is   almost   the   same   as   the   volume
$\Vol\cQ^\varepsilon_1(\cC)$, namely, by~\cite{Masur:Smillie} one has
$\Vol\cQ^\varepsilon_1(\cC)=
\Vol\cQ^{\varepsilon, \mathit{thick}}_1(\cC) + o(\varepsilon^2)$.
Taking into consideration that
$$
\int_{\R^2} \chi_\varepsilon(x,y)\, dx\, dy=\pi\varepsilon^2
$$
and  applying  Siegel--Veech formula~\eqref{eq:SV:formula:for:chi} we
get
$$
\cfrac{\Vol\cQ^{\varepsilon}_1(\cC)}
{\Vol\cQ_1}
+o(\varepsilon^2)=
b_\cC\cdot\pi\varepsilon^2
$$
which  implies  the  following formula for the Siegel--Veech constant
$b_\cC$:
\begin{equation}
\label{eq:bC:equals:Vol:over:Vol}
b_\cC=\lim_{\epsilon\to 0} \frac{1}{\pi\epsilon^2}\cdot
\cfrac{\Vol\cQ^{\varepsilon}_1(\cC)}
{\Vol\cQ_1}\,.
\end{equation}

We  complete  this  section  by  establishing  an elementary relation
between     the    Siegel--Veech    constant    $b_\cC$    used    in
\S\ref{sec:ergodic}  and  in  \S\ref{sec:subsec:principal:boundary}  and  the
Siegel--Veech         constant         $c_\cC$         used        in
\S\ref{sec:configurations:counting}.     Recall     that     counting
function~\eqref{eq:N:equals:b:pi:L:2}
$$
N_\cC(S, L)\sim b_\cC\cdot \pi L^2
$$
counts  the  number  of saddle connections of type $\cC$ of length at
most     $L$     on     the     flat    surface    $S\in\cQ_1$.    By
convention~\ref{con:area:1:2}  surfaces from $\cQ_1$ have area $1/2$.
Thus,  applying  the  asymptotic  formula~\eqref{eq:Th:N:cC:S:L} from
Theorem~\ref{theorem:SV:for:CP}  to  the  flat surface $S\in\cQ_1$ we
get
$$
N_\cC(S,L)\sim c_\cC\cdot
\frac{\pi L^2}{\text{Area of }S}=2c_\cC\cdot\pi L^2\,,
$$
which implies that
\begin{equation}
\label{eq:b:equals:2c}
b_\cC=2c_\cC\,.
\end{equation}

%----------------------------------------------------------------
\subsection{Principal boundary}

When  saddle  connections  of configuration $\cC$ are contracted by a
continuous deformation, the limiting flat surface decomposes into one
or  several  connected  components  represented by nondegenerate flat
surfaces $S'_1, \dots, S'_m$. Let the initial surface $S$ belong to a
stratum  $\cQ(\alpha)$, where $\alpha$ is the set with multiplicities
$\{d_1,\dots,d_k\}$.  Let $\cQ(\alpha'_j)$ be the stratum ambient for
$S'_j$.   The  stratum  $\cQ(\alpha'_\cC)=\cQ(\alpha'_1)\sqcup  \dots
\sqcup  \cQ(\alpha'_m)$  of  disconnected  flat  surfaces $S'_1\sqcup
\dots\sqcup  S'_m$  is  referred  to as a \textit{principal boundary}
stratum  of  the stratum $\cQ(\alpha)$. The principal boundary of any
connected  component  of  any  stratum  of  Abelian  differentials is
described in~\cite{Eskin:Masur:Zorich}; the principal boundaries of
strata      of      quadratic      differentials     are    described
in~\cite{Masur:Zorich}.

The    papers~\cite{Eskin:Masur:Zorich},   \cite{Masur:Zorich}   also
present   the   inverse   construction.  Consider  any  flat  surface
$S':=S'_1\sqcup   \dots   \sqcup  S'_m\in  \cQ(\alpha'_\cC)$  in  the
principal   boundary   of  $\cQ(\alpha)$;  consider  a  vector  $\vec
v\in\R^2\simeq\C{}$   such  that  $\|\vec  v\|\le\epsilon$.  One  can
reconstruct   a   flat   surface  $S\in\cQ(\alpha)$  endowed  with  a
collection  of  saddle  connections  of  the type $\cC$ such that the
linear  holonomy  along saddle connections is represented by $\pm\vec
v$,  and  such  that  degeneration  of  $S$  contracting  the  saddle
connections  in  the  collection  gives  the  surface  $S'$. When the
configuration  $\cC$ does not involve any cylinders, any flat surface
$S'\in\cQ_1$  and  any  holonomy  vector  $\vec v$ define the surface
$S\in\cQ_1^\epsilon(\cC)$,   basically,   up  to  some  finite  order
ambiguity  which can be explicitly computed. Moreover, the measure in
$\cQ_1^\epsilon(\cC)$     disintegrates    as    the    measure    in
$\cQ_1(\alpha'_\cC)$  times  the  measure  $d\mu_0$  in  the space of
parameters  of  the  deformation. The latter space can be viewed as a
finite  cover  of  the space of holonomy vectors $\pm\vec v$, that is
the quotient of the disk $D^2_\epsilon/\pm$ of radius $\epsilon$ over
the central symmetry. As a result we get
\begin{equation}
\label{eq:volume:of:a:cusp:1}
\Vol\left(\cQ_1^\varepsilon(\cC)\right)=
(\text{explicit factor})\cdot\pi\varepsilon^2\cdot
\Vol\cQ_1(\alpha'_\cC)\ +\ o(\varepsilon^2) .
\end{equation}
Thus,    in    order    to    compute   the   constant   $b_\cC$   by
formula~\eqref{eq:bC:equals:Vol:over:Vol} it is sufficient to express
the   volume   of   $\Vol\cQ_1(\alpha')$  in  terms  of  the  volumes
$\Vol\cQ_1(\alpha'_1),  \dots,  \Vol\cQ_1(\alpha'_m)$, and to compute
the  explicit factor, responsible for the fixed finite number of flat
surfaces  $S\in\cQ_1^\epsilon(\alpha)$  which  correspond  to a fixed
flat surface $S'\in\cQ(\alpha'_\cC)$ in the boundary stratum and to a
fixed  holonomy  vector  $\vec  v$.  The first problem is simple; the
answer       to       this       problem       is       given      in
~\S\ref{sec:subsec:Strata:of:Disconnected:Surfaces};  the second problem
is   solved  for  configurations  I--IV  in  the  remaining  part  of
~\S\ref{sec:Siegel:Veech}.

The situation for configurations which involve a cylinder is slightly
more  complicated,  but  similar  to the previous one. In both cases,
applying                    formula~\eqref{eq:bC:equals:Vol:over:Vol}
and~\eqref{eq:volume:of:a:cusp:1} we express the constant $b_\cC$ as
\begin{equation}
\label{eq:volume:of:a:cusp:2}
b_\cC=
(\text{explicit combinatorial factor})\cdot
\frac{\prod_{j=1}^k\Vol\cQ_1(\alpha'_j)}
{\Vol\left(\cQ_1(\alpha)\right)}\,.
\end{equation}

%--------------------------------------------------------------------
\subsection{Surgeries on a flat surface}

Consider  a  flat  surface  $S'\in\cQ_1(\alpha'_\cC)$  in  a  stratum  of
meromorphic  quadratic  differentials  with  at  most simple poles on
$\CP$,  possibly with a marked point. Fix some zero, or a simple pole
(or  the  marked regular point) $P_i$. Consider a vector
$\pm\vec  v\in\R^2$,  defined  up  to reversing the direction. Assume
that  $\vec v$ is much shorter than the shortest saddle connection on
$S$.

The papers~\cite{Eskin:Masur:Zorich} and~\cite{Masur:Zorich} describe how
to  perform  a  small deformation of the surface $S'$ breaking up the
chosen  singularity  $P_i$  of  degree  $d_i$  into two singularities
$P_i',  P_i''$  of  any  two  prescribed  degrees  $d_i'$ and $d''_i$
satisfying   the   relation   $d_i'+d_i''=d_i$,   where  $d_i,  d_i',
d_i''\in\{-1,0,1,2,\dots\}$. The deformation can be performed in such
way  that the holonomy vector of the resulting tiny saddle connection
joining  the  newborn singularities $P_i', P_i''$ is exactly $\pm\vec
v$.  This  deformation  is  described in details in sections 8.1--8.2
in~\cite{Eskin:Masur:Zorich}       and       in      section      6.3
in~\cite{Masur:Zorich}.  When  at least one of $d_i', d_i''$ is even,
the  deformation is local: it does not change the metric outside of a
small  neighborhood  of  $P_i$ and it does not change the area of the
flat  surface.  When  both  $d_i',  d_i''$  are  odd  the deformation
involves  some  arbitrariness  and  involves some small change of the
area   of   the   flat   surface.   A   discussion  in  the  original
papers~\cite{Eskin:Masur:Zorich} and~\cite{Masur:Zorich} explains why
both issues might be neglected in our calculations.

The  cone  angles  at  the  distinguished  singularity  is  equal  to
$\pi(d_i+2)$.  Thus,  there  are  $(d_i+2)$  geodesic  rays in linear
direction   $\pm\vec   v$  adjacent  to  $P_i$.  Take  a  small  disk
$D^2_\epsilon$  of  radius  $\epsilon$  centered  in  the  origin and
consider  its  quotient $D^2_\epsilon/\pm$ over the action of central
symmetry.  Letting  the vector $\pm\vec v$ vary in $D^2_\epsilon/\pm$
and  taking  care  of  normalization~\eqref{eq:4dxdy}  of the measure
$d\mu_0$ on $D^2_\epsilon/\pm$ we get a set of parameters of measure
\begin{equation}
\label{eq:disk:measure:in:slit:construction}
(d+2)\cdot 4\cdot \frac{\pi\epsilon^2}{2}=
2(d+2)\cdot\pi\epsilon^2\,.
\end{equation}
For    this    configuration   the   ``$(\text{explicit   factor})$''
in~\eqref{eq:volume:of:a:cusp:1} equals $2(d+2)$.

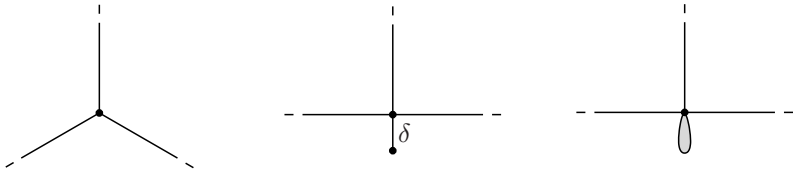
\begin{figure}
\includegraphics{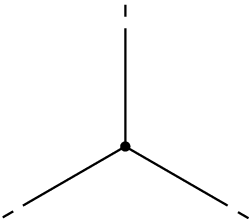}
\includegraphics{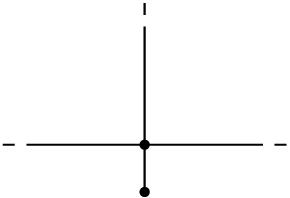}
\includegraphics{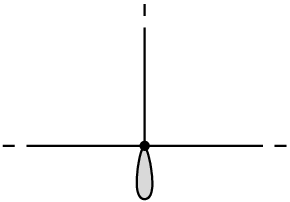}
\begin{picture}(0,0)(0,0)
\put(-4,-65){$\delta$}
\end{picture}
\vspace{80bp}
\caption{
\label{fig:breakin:up:a:zero}
Breaking  up a zero into two. In the particular case, when one of the
newborn  singularities  is  a  simple  pole,  we  can  slit along the
resulting  saddle connection of length $\delta$ to get a surface with
geodesic boundary of length $2\delta$.
   }
\end{figure}

Consider now a particular case, when one of the newborn singularities
$P_i',    P_i''$,    say,   $P_i''$   is   a   simple   pole.   Since
$d_i'+d''_i=d_i\ge  -1$,  the  singularities  $P_i', P_i''$ cannot be
simple  poles  simultaneously.  Making  a slit along the short saddle
connection   joining   $P_i'$   to   $P_i''$   we  create  a  surface
${\mathring{\!S}}$  with  geodesic boundary. Note that the cone angle
at  the singularity $P_i''$ was $\pi$. This means, that after opening
up  a slit, the point $P_i''$ becomes a regular point of the boundary
of  ${\mathring{\!S}}$,  see  Figure~\ref{fig:breakin:up:a:zero}.  In
other  words,  the  boundary  of  ${\mathring{\!S}}$ corresponds to a
single closed geodesic with linear holonomy $\pm\vec v$.

%--------------------------------------------------------------------
\subsection*{Parallelogram construction}
In  order to construct the subset $\cQ_1^\epsilon(\cC)$ corresponding
to  configuration  II,  we need another surgery. Given a flat
surface  $S'\in\cQ_1(\alpha')$  in a stratum of meromorphic quadratic
differentials  with  at  most  simple  poles  on  $\CP$, given a pair of
singularities  $P',P''$  on  $S'$  and  given a short vector $\pm\vec
v\in\R^{2}$,  we  construct  a  surface  with two boundary components
creating  a  pair of small holes adjacent to the chosen singularities
$P',P''$.  The  surgery  is performed in such way that the holes have
geodesic  boundary  with linear holonomy $\pm\vec v$. Let $d',d''$ be
the   degrees   of   singularities   $P',   P''$   respectively.  The
corresponding  cone  angles  are  $\pi(d'+2)$ and $\pi(d''+2)$. Thus,
there  are  $(d'+2)$  geodesic  rays  in linear direction $\pm\vec v$
adjacent  to  $P'$  and  $(d''+2)$  geodesic rays in linear direction
$\pm\vec v$ adjacent to $P''$.

The    corresponding   surgery   is   described   in   section   12.2
in~\cite{Eskin:Masur:Zorich}       and       in      section      6.1
in~\cite{Masur:Zorich}       as      the      ``\textit{parallelogram
construction}''.  This  is  a  nonlocal  construction,  so  it is not
canonical,  and  it  changes  slightly the area of the surface. Up to
this  ambiguity  (which  can  be  neglected  in  our  computations as
explained  in~\cite{Eskin:Masur:Zorich}  and in~\cite{Masur:Zorich}),
given  the data as above, there are $(d'+2)(d''+2)$ ways to construct
the  described surface with boundary ${\mathring{\!S}}$. Take a small
disk  $D^2_\epsilon$  of radius $\epsilon$ centered in the origin and
consider  its  quotient $D^2_\epsilon/\pm$ over the action of central
symmetry. Let the vector $\pm\vec v$ vary in $D^2_\epsilon/\pm$. Note
that  in  the contrary to the previous case, the saddle connection is
now \textit{closed}. Thus the measure along the fiber has the form
$$
d\mu_0=dx\,dy
$$
and  not  the  form~\eqref{eq:4dxdy} as before. This implies that for
this  configuration  the  set  of  parameters  of  deformation having
holonomy vectors in $D^2_\epsilon/\pm$ has the measure
\begin{equation}
\label{eq:disk:measure:in:parallelogram:construction}
(d'+2)(d''+2)\cdot \frac{\pi\epsilon^2}{2}\,.
\end{equation}
For    this    configuration   the   ``$(\text{explicit   factor})$''
in~\eqref{eq:volume:of:a:cusp:1} equals $\cfrac{(d'+2)(d''+2)}{2}$\,.

%--------------------------------------------------------------------
\subsection{Type I: A simple saddle connection joining a fixed zero to a fixed pole or to a distinct fixed zero.}
\label{sec:subsec:sv:typeI}

Now  we  finally  pass  to  explicit computation of the Siegel--Veech
constants following the strategy described above.

Throughout  the rest of this section $\cQ(d_1,\dots,d_k)$ denotes any
stratum  of  meromorphic  quadratic differentials with at most simple
poles on $\cP$ different from the stratum $\cQ(-1^4)$ of pillowcases.

\begin{theorem}
\label{th:SV:I}
For the configuration  $\cC$  of  saddle connections of type I, i.e. for
saddle  connections  joining  a  fixed  pair  $P_i,  P_j$ of distinct
singularities of orders $d_i,d_j$, the Siegel--Veech constant $c_\cC$
is expressed as follows:
\begin{equation}
\label{eq:ratio:Siegel:Veech:constant:sc}
c_\cC=(d_i+d_j+2) \cfrac{\Vol \cQ_1(d_i+d_j,d_1,d_2, \dots,
\widehat{d_i}, \dots, \widehat{d_j}, \dots, d_k)} {\Vol \cQ_1(d_1,
\dots, d_k)}.
\end{equation}
After plugging in Theorem~\ref{theorem:volume}, we get:
\begin{equation}
\label{eq:Siegel:Veech:constant:sc}
c_\cC =
\cfrac{(d_i+d_j+2)!!\ (d_i+1)!!\ (d_j+1)!!}
{(d_i+d_j+1)!!\  d_i!!\  d_j!!}\cdot
\begin{cases}\cfrac{2}{\pi^2}&\text{ when both $d_i,d_j$ are odd}\\ & \\
\cfrac{1}{2}&\text{ otherwise}
\end{cases}
\end{equation}
\end{theorem}
\begin{proof}[Proof of (\ref{eq:ratio:Siegel:Veech:constant:sc})]
The  principal  boundary  $\cQ_1(\alpha'_\cC)$  for  this  particular
configuration  $\cC$  is obtained by collapsing the saddle connection
joining  singularities  of  degrees  $d_i$  and $d_j$. This operation
merges two singularities to a single one of degree $d=d_i+d_j$. Thus,
$$
\alpha'_\cC=
\{d_i+d_j,d_1,d_2, \dots,\widehat{d_i}, \dots, \widehat{d_j}, \dots, d_k\}\,.
$$
By~\eqref{eq:volume:of:a:cusp:1}
$$
\Vol\left(\cQ_1^\varepsilon(\cC)\right)=
(\text{explicit factor})\cdot\pi\varepsilon^2\cdot
\Vol\cQ_1(\alpha'_\cC)\ +\ o(\varepsilon^2)\,,
$$
where   the  ``$(\text{explicit  factor})\cdot\pi\varepsilon^2$''  in
formula~\eqref{eq:volume:of:a:cusp:1}  stands  for the measure of the
space  of parameters of deformation corresponding to holonomy vectors
in      $D^2_\epsilon/\pm$.     This     measure     was     computed
in~\eqref{eq:disk:measure:in:slit:construction}.    Thus,    we   can
rewrite~\eqref{eq:volume:of:a:cusp:1} as
$$
\Vol\left(\cQ_1^\varepsilon(\cC)\right)=
2(d+2)\cdot\pi\varepsilon^2\cdot
\Vol\cQ_1(d_i+d_j,d_1,d_2, \dots,\widehat{d_i}, \dots, \widehat{d_j}, \dots, d_k)\ +\ o(\varepsilon^2)\, .
$$
Applying~\eqref{eq:bC:equals:Vol:over:Vol} and~\eqref{eq:b:equals:2c}
to            the            above            expression           we
obtain~\eqref{eq:ratio:Siegel:Veech:constant:sc}.
\end{proof}

We       are       ready       to      give      a      proof      of
Theorem~\ref{theorem:L:shaped:almost:all}                      (based
on~\eqref{eq:Siegel:Veech:constant:sc}   which  would  be  proved  in
~\S\ref{sec:induction}).

\begin{proof}[Proof of Theorem~\ref{theorem:L:shaped:almost:all}]
Let    $\cQ(d_1,\dots,d_k)=\cQ(1,-1^5)$.   Let   $d_i=1$,   $d_j=-1$.
Applying~\eqref{eq:Siegel:Veech:constant:sc}  we  get  $c_I=8/\pi^2$.
Applying      Theorem~\ref{theorem:SV:for:billiards}      to      the
$\operatorname{L}$-shaped           billiard           as          in
Figure~\ref{fig:L:shaped:billiard} we get the coefficient $\frac{2}{\pi}$
in  the  weak  asymptotics  of  the  number  of generalized diagonals
joining  a  fixed  corner  with angle $\frac{\pi}{2}$ with the corner
with       angle      $\frac{3\pi}{2}$,      and      thus      prove
formula~\ref{eq:theorem:L:shaped:almost:all}                      and
Theorem~\ref{theorem:L:shaped:almost:all}.
\end{proof}

%--------------------------------------------------------------------
\subsection{Type II: A simple saddle connection joining a zero to itself.}
\label{sec:subsec:sv:typeII}

The configuration  $\cC$  of type II consists of a single separatrix loop
emitted  from  a  fixed zero $P_i$ of order $d_i$ such that the total
angle   $(d_i+2)\pi$  at  the  singularity  $P_i$  is  split  by  the
separatrix  loop into two sectors having the angles $(d'_i+3)\pi$ and
$(d''_i+3)\pi$. We assume that $d'_i, d''_i \ge -1$, so we \textit{do
not}  have  any  cylinders  filled  with  periodic geodesics for this
configuration. The angles satisfy the natural relation
$$
d'_i+d''_i=d_i-4 \qquad d'_i, d''_i \ge -1
$$
which implies, in particular, that $d_i\ge 2$.

Our saddle connection separates the original surface $S$ into two
parts.  Let $P_{i_1}, \dots, P_{i_{k_1}}$ be the list of
singularities (zeroes and poles) which belong to the first part and
let $P_{j_1}, \dots, P_{j_{k_2}}$ be the list of singularities
(zeroes and poles) which belong to the second part. This information is
part of the configuration of this saddle connection.

We  assume  that  the  initial  surface  $S$ does not have any marked
points;  as  usual  we  denote  by $d_n$ the order of the singularity
$P_n$. The set with multiplicities $\{d_1, \dots, d_k\}$ representing
the  orders  of  all  singularities  (zeroes and poles) on $S$ can be
obtained as a disjoint union of the following subsets:
$$
\{d_1, \dots, d_k\}=
\{d_{i_1}, \dots d_{i_{k_1}}\}\sqcup
\{d_{j_1}, \dots, d_{j_{k_2}}\}\sqcup
\{d_i\}
$$

\begin{theorem}
\label{th:SV:II}
The Siegel--Veech constant $c_\cC$ for this configuration
is expressed as follows:
\begin{multline}
\label{eq:ratio:Siegel:Veech:constant:zero:to:itself}
c_\cC=\cfrac{(d'_i+2)(d''_i+2)}{8}\ \cdot\\
\cdot\cfrac{
  (\dim_{\C{}}\cQ(d'_i,d_{i_1},  \dots, d_{i_{k_1}})-1)!\
  (\dim_{\C{}}\cQ(d''_i,d_{j_1}, \dots, d_{j_{k_2}})-1)!}
{(\dim_{\C{}}\cQ(d_1,d_2,\dots,d_k)-2)!}\ \cdot\\
\cdot\cfrac{\Vol \cQ_1(d'_i,d_{i_1}, \dots, d_{i_{k_1}})\,\cdot\,
       \Vol \cQ_1(d''_i,d_{j_1}, \dots, d_{j_{k_2}})}
{\Vol \cQ_1(d_1, \dots, d_k)}%\\
\end{multline}
After plugging in Theorem~\ref{theorem:volume} we get:
\begin{multline}
\label{eq:Siegel:Veech:constant:zero:to:itself}
c_\cC =\cfrac{1}{8}\,\cdot\cfrac{(d'_i+2)!!\ (d''_i+2)!!\ (d_i+1)!!}
{(d'_i+1)!!\ (d''_i+1)!!\ d_i!!}\,\cdot\\
\cfrac{(k_1-2)!\,(k_2-2)!}{(k-4)!}\,\cdot\,
\begin{cases}\ 1 &\text{when both $d'_i,d''_i$}\\ &\text{are odd} \\
\cfrac{4}{\pi^2}&\text{otherwise}
\end{cases}
\end{multline}
\end{theorem}
\begin{proof}[Proof of~\eqref{eq:ratio:Siegel:Veech:constant:zero:to:itself}]
Let
$$
\alpha'_a:=\{d_i',d_{i_1}, \dots d_{i_{k_1}}\}
\qquad
\alpha'_b:=\{d_i'',d_{j_1}, \dots d_{j_{k_2}}\}
\qquad
\alpha'_\cC:=\alpha'_a\sqcup\alpha'_b\,.
$$
Contracting  a  saddle  connection  of  type  II  and  detaching  the
resulting  singular  flat  surface  into  two  components  we  get  a
disconnected     flat    surface    $S'=S'_a\sqcup    S'_b$,    where
$S'\in\cQ(\alpha'_\cC)$.   The   stratum   of  disconnected  surfaces
$\cQ(\alpha'_\cC)$  is  the  principal boundary for configuration II.
By~\eqref{eq:volume:of:a:cusp:1}
$$
\Vol\left(\cQ_1^\varepsilon(\cC)\right)=
(\text{explicit factor})\cdot\pi\varepsilon^2\cdot
\Vol\cQ_1(\alpha'_\cC)\ +\ o(\varepsilon^2)\, .
$$
By~\eqref{eq:total:volume:of:nonprimitive:stratum} we have
$$
\Vol\cQ_1(\alpha'_\cC)=
\frac{1}{2}\cdot
\frac{(\dim_\cx\cQ(\alpha'_a)-1)! (\dim_\cx\cQ(\alpha'_b)-1)!}
{(\dim_\cx\cQ(\alpha'_\cC)-1)!}\cdot
\Vol\cQ_1(\alpha'_a)\cdot\Vol\cQ_1(\alpha'_b)\,.
$$
Note that by definition
$\dim_\cx\cQ(\alpha'_\cC)=
\dim_\cx\cQ(\alpha'_a)+\dim_\cx\cQ(\alpha'_a)$. Hence
$$
\dim_\cx\cQ(\alpha'_\cC)=
\big((k_1+1)-2\big)+
\big((k_2+1)-2\big)=(k_1+k_2)-2=k-3=
\dim_\cx\cQ(\alpha)-2\,.
$$
The     ``$(\text{explicit     factor})\cdot\pi\varepsilon^2$''    in
formula~\eqref{eq:volume:of:a:cusp:1}  stands  for the measure of the
space  of parameters of deformation corresponding to holonomy vectors
in  $D^2_\epsilon/\pm$.  For  configuration  $\cC$  of  type  II this
measure                          was                         computed
in~\eqref{eq:disk:measure:in:parallelogram:construction}.   Thus,  we
can rewrite~\eqref{eq:volume:of:a:cusp:1} as
\begin{multline*}
\Vol\left(\cQ_1^\varepsilon(\cC)\right)=
\frac{(d'+2)(d''+2)}{2}\cdot\pi\varepsilon^2\cdot\\
\frac{1}{2}\cdot\cfrac{
  (\dim_{\C{}}\cQ(d'_i,d_{i_1},  \dots, d_{i_{k_1}})-1)!\
  (\dim_{\C{}}\cQ(d''_i,d_{j_1}, \dots, d_{j_{k_2}})-1)!}
{(\dim_{\C{}}\cQ(d_1,d_2,\dots,d_k)-2)!}\ \cdot\\
\Vol \cQ_1(d'_i,d_{i_1}, \dots, d_{i_{k_1}})\,\cdot\,
       \Vol \cQ_1(d''_i,d_{j_1}, \dots, d_{j_{k_2}})\,.
\end{multline*}
Applying~\eqref{eq:bC:equals:Vol:over:Vol} and~\eqref{eq:b:equals:2c}
to the above expression we
obtain~\eqref{eq:ratio:Siegel:Veech:constant:zero:to:itself}.
\end{proof}

%--------------------------------------------------------------------
\subsection{A ``pocket'', i.e. a cylinder bounded by a pair of poles}
\label{sec:subsec:sv:pocket}

Consider  a  configuration  $\cC$  of type III where we have a single
cylinder filled with closed regular geodesics, such that the cylinder
is  bounded  by  a  saddle  connection joining a fixed pair of simple
poles $P_{j_1}, P_{j_2}$ on one side and by a separatrix loop emitted
from  a  fixed zero $P_i$ of order $d_i\ge 1$ on the other side. This
information is considered to be part of the configuration.
By  convention,  the affine holonomy associated to this configuration
corresponds  to  the  closed  geodesic and \textit{not} to the saddle
connection joining the two simple poles. (Such a saddle connection is
twice as short as the closed geodesic.)

\begin{theorem}
The Siegel--Veech constant $c_\cC$ for this configuration
is expressed as follows:
\begin{equation}
\label{eq:ratio:Siegel:Veech:constant:pocket}
c=\cfrac{d_i}{2(\dim_{\C{}}\cQ(d_1, \dots, d_k)-2)}\cdot
\cfrac{\Vol \cQ_1(d_1,d_2, \dots, d_i-2, \dots, d_k)}{\Vol
\cQ_1(d_1, \dots,d_i,\dots, d_k)}\,.
\end{equation}
After plugging in Theorem~\ref{theorem:volume}, we get
\begin{equation}
\label{eq:Siegel:Veech:constant:pocket}
c_\cC =\cfrac{d_i+1}{2(k-4)}\cdot\cfrac{1}{\pi^2}\,.
\end{equation}
\end{theorem}

\begin{proof}[Proof of~\eqref{eq:ratio:Siegel:Veech:constant:pocket}]
Let      $\alpha'_\cC=\{d_1,\dots,d_{i-1},d_i-2,d_{i+1},\dots,d_k\}$.
Consider  a  configuration of type III with a short saddle connection
$\gamma$  joining  a  zero  of  degree  $d_i$  to itself. Contracting
$\gamma$ we get a flat surface $S'$ in the principal boundary stratum
$\cQ(\alpha'_\cC)$.

To  go  backwards,  we  need  to  create a hole on $S'$ with geodesic
boundary  having  holonomy  $\pm\vec  v$  and  attach  a  cylindrical
``pocket''    to    this    hole;    see   the   right   picture   in
Figure~\ref{fig:III:pocket}.  The cone angle at the singularity $P_i$
of  degree  $(d_i-2)$  is  $\pi\cdot  d_i$.  Thus,  having  a surface
$S'\in\cQ_1(\alpha'_\cC)$     and     a    vector    $\pm\vec    v\in
D^2_\epsilon/\pm$ there are $d_i$ rays in line direction $\pm \vec v$
adjacent to the singularity $P_i$.

Note,  however,  that now a deformation involves not only the surface
$S'$  from  the principal boundary and a holonomy vector $\pm\vec v$,
but  also  additional  parameters  describing  the  geometry  of  the
``pocket''.  Geometrically,  a  ``pocket''  is  equivalent  to a flat
cylinder  endowed  with  a  distinguished  line  direction and with a
marked point on each of the boundary components. Thus, in addition to
the  holonomy vector $\pm \vec v$ representing the waist curve, it is
parameterized  by  the \textit{height} $h$ of the cylinder and by the
\textit{twist}  $t$  of  the cylinder, $0\le t< |\vec v|$. Parameters
$h$  and $t$ record the information about the holonomy along a saddle
connection  joining the zero $P_i$ on one side of the cylinder to one
of  the  simple  poles,  say,  $P_{j_1}$  on  the  other  side of the
cylinder.  The flat area of a ``pocket'' $T(\pm \vec v, h, t)$ equals
$|\vec v|\cdot h$.

The   measure   $d\mu$   in   $\cQ^{\epsilon,\mathit{thick}}(\alpha)$
disintegrates  into the product measure $d\mu'$ on $\cQ(\alpha')$ and
the measure $d\nu$ on the ``space of pockets'' $\cR$,
$$
d\mu(S)=d\mu(S')\cdot d\nu(T)\,.
$$
The   parameter   $\pm\vec  v$  corresponds  to  the  holonomy  along  a
\textit{closed}  saddle  connection,  while  the  parameters  $(h,t)$
correspond  to  holonomy  along  a saddle connection joining distinct
singularities.  Hence,  the  resulting measure on the space of parameters
defining a ``pocket'' is
$$
d\nu(T)=d\vec v\cdot 4dhdt\,.
$$

Following  Convention~\ref{con:area:1:2}  we  denote  by  $\cR_1$ the
hypersurface     of    pockets    of    area    $\frac{1}{2}$.    Let
$S\in\cQ_1(\alpha)$.  We  denote  by  $rS\in\cQ(\alpha)$  the surface
proportional  to  the  initial  one  with  the  coefficient  $r$;  in
particular,  $area(rS)=r^2/2$,  see Convention~\ref{con:area:1:2}. We
use  similar  notations  $r_S  S'$  and  $r_T  T$  for  surfaces from
$\cQ(\alpha')$  and  from  $\cR$  correspondingly. We recall that the
volume   elements  in  the  strata  and  the  area  elements  on  the
corresponding   ``unit   hyperboloids''   are   related   as follows,
see~\eqref{eq:disintegration}:
$$
\begin{array}{lll}
d\mu=r^{2n-1}dr\,\,d\mu_1\,,       &\text{ where }n\ \ =\dim_\C{} \cQ(\alpha)\,=2(k-2)\\
d\mu=r_S^{2n_S-1}dr_S\,\,d\mu'_1\,,&\text{ where }n_S=\dim_\C{} \cQ(\alpha')=2(k-4)\\
d\nu=r_T^{2n_T-1}dr_T\,\,d\nu_1\,, &\text{ where }n_T=\dim_\C{} \cR=2\,.
\end{array}
$$

Let   $S'\in\cQ_1(\alpha')$.  Consider  a  surface  $r_S  S'$,  where
$0<r_S\le      1$;      it     has     area     $r_S^2/2$.     Define
$\Omega(\epsilon,r_S)\subset \cR$ to be the set of pockets, such that
performing  an  appropriate  surgery  to  $r_S  S'$  and pasting in a
``pocket''   from  $\Omega(\epsilon,r_S)$  we  get  a  surface  $S\in
C(Q_1^\epsilon(\alpha))$. Ignoring a negligible change of the area of
the  surface $r_S S'$ after creating a hole, we get the following two
constraints.  The  first  constraint  imposes  the  bound  on the area
$r^2_T/2$  of  a pocket $r_T T$, where $T\in\cR_1$: the total area of
the compound surface $S$ should be at most $1/2$, so $r^2_S +r^2_T\le
1$.  The  second  constraint  imposes a bound on the length of the
waist  curve  of  the  cylinder:  after  rescaling proportionally the
compound  surface $S$ to let it have area $\frac{1}{2}$ we should get
a  waist curve of length at most $\epsilon$. Thus, the waist curve of
the original cylinder should be at most $\epsilon\sqrt{r^2_S+r^2_T}$.
Clearly,  the  set  $\Omega(\epsilon,r_S)$  does  not  depend  on the
particular   surface   $r_s   S'\in\cQ(\alpha')$,  but  only  on  the
parameters $r_S$ and $\epsilon$.

We have seen that there are $d_i$ rays in line direction $\pm \vec v$
adjacent  to  the singularity $P_i$. Using the above notations we can
represent   the   volume   of   a   cone   in   $\cQ_1(\alpha)$  over
$\cQ_1^{\epsilon}(\alpha,\cC)$     (see~\eqref{eq:cone}    for    the
definition of a \textit{cone}) as
\begin{equation}
\label{eq:initial:integral}
\mu(C(\cQ_1^{\epsilon}(\alpha,\cC)))
=d_i\cdot\Vol\cQ_1(\alpha')\cdot
\int_0^1 \nu_T(\Omega(\epsilon,r_S))
r_S^{2n_S-1} dr_S + o(\epsilon^2)\,.
\end{equation}

Denote by $\mathit{Cusp}(\epsilon)$ the volume of the $\epsilon$-thin part
of the ``unit hyperboloid'' in the space of ``pockets'':
$$
Cusp(\epsilon):=\Vol\cR_1^\epsilon\,.
$$
From the  definition  of  the  subset  $\Omega(\epsilon,r_S)$  it
immediately follows that its volume is expressed by the following
integral
\begin{equation}
\label{eq:volume:of:omega}
\nu_T(\Omega(\epsilon,r_S))=
\int_0^{\sqrt{1-r_S^2}} r_T^{2n_T-1}\cdot
Cusp\left(\frac{\epsilon\cdot\sqrt{r_S^2+r_T^2}}{r_T}\right)\,
dr_T\,.
\end{equation}
Thus, we need to evaluate the following integral
\begin{multline}
\label{eq:complete:initial:integral}
\mu(C(\cQ^{\epsilon}(\alpha,\cC)))
=d_i\cdot\Vol(\cQ_1(\alpha'))\cdot\\
\int_0^1 r_S^{2n_S-1} dr_S
\int_0^{\sqrt{1-r_S^2}} r_T^{2n_T-1}\cdot
Cusp\left(\frac{\epsilon\cdot\sqrt{r_S^2+r_T^2}}{r_T}\right)\,
dr_T\
 + o(\epsilon^2)\,.
\end{multline}

\begin{lemma}
$$
Cusp(\epsilon)=2\pi \epsilon^2.
$$
\end{lemma}
\begin{proof}
We  first  evaluate  the  volume  $\nu(C(\cR^{\epsilon}_1))$  of  the
corresponding  cone.  Pockets belonging to this cone are described by
the following conditions:
  % (see Figure~\ref{fig:domain:of:integration:2}):
   %
$$
\begin{cases}
h\cdot|\vec v|\le 1/2\\
|\vec v|\le \epsilon\cdot \sqrt{2h\cdot|\vec v|}\,.
\end{cases}
$$
Hence
$$
\nu(C(\cR^{\epsilon}_1))=
\int_{D^2_\epsilon/\pm}d\vec v
\int_{w/(2\epsilon^2)}^{1/(2w)} 2dh
\int_0^w 2dt
= 4\pi\int_0^\epsilon w \left(\frac{1}{2w}- \frac{w}{2\epsilon^2}\right)w\,dw =
\frac{\pi\epsilon^2}{2},
$$
where  $w=|\vec  v|$.  It  remains  to apply~\eqref{eq:normalization}:
$$
\nu(C(\cR^{\epsilon}_1))= \dim_\R\cR\cdot\Vol(\cR^{\epsilon}_1)
$$
and to note that $\dim_\R\cR=4$.
\end{proof}

Having found the expression
$$
%\begin{equation}
%\label{eq:cusp}
  %
Cusp\left(\frac{\epsilon\cdot\sqrt{r_S^2+r_T^2}}{r_T}\right)=
2\pi\epsilon^2\cdot\frac{r_S^2+r_T^2}{r_T^2}
  %
%\end{equation}
$$
we  can
rewrite the integral~\eqref{eq:complete:initial:integral} as
\begin{multline}
\label{eq:final:initial:integral}
\mu(C(\cQ^{\epsilon}(\alpha,\cC)))  %=\\
=d_i\cdot\Vol\cQ_1(\alpha')\cdot\\
2\pi\epsilon^2\cdot\int_0^1 r_S^{2n_S-1} dr_S
\int_0^{\sqrt{1-r_S^2}} r_T^{2n_T-1}\cdot
\frac{r_S^2+r_T^2}{r_T^2}
\,\,dr_T\
 + o(\epsilon^2)\,.
\end{multline}
Taking  into  consideration  that  $n_T=\dim_\C{}\cR=2$ we compute
the above integrals and get
\begin{equation}
\label{eq:postfinal:initial:integral}
\mu(C(\cQ_1^{\epsilon}
(\alpha,\cC)))
=d_i\cdot\Vol\cQ_1(\alpha')
\cdot\frac{2\pi\epsilon^2}{2n_S(2n_S+4)}
\
 + o(\epsilon^2)\,.
\end{equation}

It remains to note that
$$
\Vol\cQ_1^{\epsilon}(\alpha,\cC)=
\dim_\R\cQ(\alpha)\cdot
\mu(C(\cQ_1^{\epsilon}(\alpha,\cC)))
\,,
$$
see~\eqref{eq:normalization}, and that
$$
\dim_\R\cQ(\alpha)=
2\dim_\cx\cQ(\alpha)=
2(\dim_\cx\cQ(\alpha')+2)=
2n_S+4.
$$
to get
\begin{equation}
\label{eq:pocket:almost:answer}
\Vol\cQ_1^{\epsilon}(\alpha,\cC)=
\pi\epsilon^2\cdot \frac{d_i}{\dim_\cx\cQ(\alpha)-2}
\cdot\Vol\cQ_1(\alpha')
 + o(\epsilon^2)\,.
\end{equation}
Applying~\eqref{eq:bC:equals:Vol:over:Vol} and~\eqref{eq:b:equals:2c}
to the above expression we
obtain~\eqref{eq:ratio:Siegel:Veech:constant:pocket}.
\end{proof}

The  rest  of  the  discussion  in  \S\ref{sec:subsec:sv:pocket} also
depends  on  Theorem~\ref{theorem:volume}.  Consider  a slightly more
general  configuration:  as before we consider a fixed pair of simple
poles  $P_{j_1}, P_{j_2}$, but this time we do not specify which zero
do  we  have  at  the base of the cylinder. Clearly the corresponding
Siegel--Veech  constant $c^{pocket}_{j_1,j_2}$ is equal to the sum of
the Siegel--Veech constants considered above over all zeroes $P_i$ on
our surface $S$:
$$
c^{pocket}_{j_1,j_2}=
\sum_{i=1\, |\, d_i\ge 1}^k c_i\,.
$$
The following Corollary follows immediately from the
formula~\eqref{eq:Siegel:Veech:constant:pocket} above.

\begin{corollary}
For any stratum of meromorphic quadratic differentials with at
most simple poles and with no marked points on $\CP$ and for
every fixed pair $P_{j_1}, P_{j_2}$ of simple poles, the
Siegel--Veech constant $c^{pocket}_{j_1,j_2}$ is equal to
\begin{equation}
\label{c:pocket}
c^{pocket}_{j_1,j_2}=\cfrac{1}{2\pi^2}\,.
\end{equation}
\end{corollary}
\begin{proof}
By  assumption  the  stratum  $\cQ(d_1, \dots, d_k)$ does not contain
marked points. We can order $d_i$ in the reverse lexicographic order,
so  that  $d_1,  \dots,  d_m$  are  positive  (i.e. correspond to the
zeroes)  and  $d_{m+1},  \dots,  d_{m+n}$  are  equal  to  $-1$ (i.e.
correspond to the simple poles).

Since we live on $\CP$ we have
$
\sum_{i=1}^k d_i = -4
$
which is equivalent to
$
\sum_{i=1}^m d_i = n-4
$.
Hence,
\begin{multline*}
c^{pocket}_{j_1,j_2}=
\cfrac{1}{2(k-4)}\sum_{i=1}^m (d_i+1)\cfrac{1}{\pi^2}=\\
=\cfrac{1}{2(n+m-4)}\Big(\sum_{i=1}^m d_i+\sum_{i=1}^m 1\Big)\cfrac{1}{\pi^2}=
\cfrac{1}{2(n+m-4)}\big(n-4+m\big)\cfrac{1}{\pi^2}=\cfrac{1}{2\pi^2}\,.
\end{multline*}
\end{proof}

\begin{proof}[Proof of Theorem~\ref{theorem:pi:2:general}]
Note  that  Theorem~\ref{theorem:pi:2:general}  counts  the number of
generalized  diagonals  joining  two  fixed corners of a right-angled
billiard,  while  in the ``pocket'' configuration we count the number
of  closed  flat  geodesics  on the induced cylinder, which are twice
longer.  Rescaling, we get an extra factor 4 for the counting problem
in        this        alternative       normalization.       Applying
Theorem~\ref{theorem:SV:for:billiards}, and taking into consideration
the  factor  $\frac{1}{4}$  in formula~\eqref{eq:SV:for:billiards} we
get the answer
$$
N_{ij}(\Pi,L)\mysim
\frac{1}{4}\cdot 4\cdot c^{pocket}_{i,j}
\frac{\pi L^2}{\text{Area of the billiard table }\Pi}\,.
$$
Plugging in expression~\eqref{c:pocket} for $c^{pocket}_{i,j}$ we get
formula~\eqref{eq:theorem:L:shaped:almost:all}.
\end{proof}

%--------------------------------------------------------------------
\subsection{A ``dumbbell'', i.e. a simple cylinder separating the sphere and joining a pair of
distinct zeroes}.
\label{sec:subsec:sv:dumbell}

Consider  a  configuration  $\cC$  of type IV, where we have a single
cylinder filled with closed regular geodesics, such that the cylinder
is  bounded  by  a  separatrix  loop on each side. We assume that the
separatrix  loop  bounding the cylinder on one side is emitted from a
fixed  zero  $P_i$  of  order $d_i\ge 1$ and that the separatrix loop
bounding  the cylinder on the other side is emitted from a fixed zero
$P_j$ of order $d_j\ge 1$.

Such a cylinder separates the original surface $S$ in two parts;
let $P_{i_1}, \dots, P_{i_{k_1}}$ be the list of singularities
(zeroes and simple poles) which get to the first part and $P_{j_1},
\dots, P_{j_{k_2}}$ be the list of singularities (zeroes and
simple poles) which get to the second part. In particular, we have
$i\in\{i_1, \dots, i_{k_1}\}$ and $j\in\{j_1, \dots,
j_{k_2}\}$.
We assume that $S$ does not have any marked points.
Denoting as usual by $d_n$ the order of the singularity $P_n$
we can represent the sets with multiplicities $\alpha:=\{d_1, \dots, d_k\}$
as a disjoint union of the two subsets
$$
\{d_1, \dots, d_k\}=
\{d_{i_1}, \dots d_{i_{k_1}}\}\sqcup
\{d_{j_1}, \dots, d_{j_{k_2}}\}.
$$

This information is considered to be part of the
configuration.

\begin{theorem}
\label{th:SV:IV}
The Siegel--Veech constant $c_\cC$ for this configuration
is expressed as follows:
\begin{multline}
\label{eq:ratio:Siegel:Veech:constant:dumbell}
c_\cC=\cfrac{d_i\cdot d_j}{4}\ \cdot\\
\cfrac{
  (\dim_{\C{}}\cQ(d_{i_1}, \dots,d_i-2, \dots, d_{i_{k_1}})-1)!
  (\dim_{\C{}}\cQ(d_{j_1}, \dots,d_j-2, \dots, d_{j_{k_2}})-1)!}
{(\dim_{\C{}}\cQ(d_1,d_2,\dots,d_k)-2)!}\ \cdot\\
\cdot\cfrac{\Vol \cQ_1(d_{i_1}, \dots,d_i-2, \dots, d_{i_{k_1}})\,\cdot\,
       \Vol \cQ_1(d_{j_1}, \dots,d_j-2, \dots, d_{j_{k_2}})}
{\Vol \cQ_1(d_1, \dots, d_k)}\,.
\end{multline}
Plugging in Theorem~\ref{theorem:volume} we get:
\begin{equation}
\label{eq:Siegel:Veech:constant:dumbell}
c_\cC =\cfrac{(d_i+1)(d_j+1)}{2}\cdot
\cfrac{(k_1-3)!\,(k_2-3)!}{(k-4)!}\cdot\cfrac{1}{\pi^2}\,.
\end{equation}
\end{theorem}
\begin{proof}[Proof of~\eqref{eq:ratio:Siegel:Veech:constant:dumbell}]
The proof is completely analogous to computation of the Siegel--Veech
constant  for  configuration  III. Denote by $\alpha'_a$ the set with
multiplicities  obtained  from  $\{d_{i_1},  \dots  d_{i_{k_1}}\}$ by
replacing   the   entry   $d_i$  by  $d_i-2$.  Similarly,  denote  by
$\alpha'_b$  the  set  with  multiplicities obtained from $\{d_{j_1},
\dots, d_{j_{k_2}}\}$ by replacing the entry $d_j$ by $d_j-2$. Define
$\alpha':=\alpha'_a\sqcup\alpha'_b$.   Contracting   the  two  saddle
connections  we get a disconnected flat surface $S'$ in the principal
boundary stratum $\cQ(\alpha')$.

Given  a  flat  surface  $S'\in\cQ(\alpha')$  and  a  holonomy vector
$\pm\vec  v$  we  have $d_i$ separatrix rays in direction $\pm\vec v$
adjacent  to  the  point  $P_i$  of $S'$ and $d_j$ separatrix rays in
direction $\pm\vec v$ adjacent to the point $P_j$.

Following line-by-line the proof of~\eqref{eq:ratio:Siegel:Veech:constant:pocket}
in the previous section we get an expression for
$\Vol\cQ_1^{\epsilon}(\alpha,\cC)$ completely analogous to~\eqref{eq:pocket:almost:answer}:
the only adjustment consists in replacing the factor $d_i$ by the
product $d_i d_j$:
$$
\Vol\cQ_1^{\epsilon}(\alpha,\cC)=
\pi\epsilon^2\cdot \frac{d_i d_j}{\dim_\cx\cQ(\alpha)-2}
\cdot\Vol\cQ_1(\alpha')
 + o(\epsilon^2)\,.
$$
Applying expression~\eqref{eq:total:volume:of:nonprimitive:stratum}
from ~\S\ref{sec:subsec:Strata:of:Disconnected:Surfaces} for
$\Vol\cQ_1(\alpha')$ and taking into consideration that
$\dim_\cx\cQ(\alpha')=\dim_\cx\cQ(\alpha)-2$ we can rewrite
the latter expression as
\begin{multline*}
\Vol\cQ_1^{\epsilon}(\alpha,\cC)=
\pi\epsilon^2\cdot \frac{d_i d_j}{2}\cdot\\
\cfrac{
  (\dim_{\C{}}\cQ(d_{i_1}, \dots,d_i-2, \dots, d_{i_{k_1}})-1)!
  (\dim_{\C{}}\cQ(d_{j_1}, \dots,d_j-2, \dots, d_{j_{k_2}})-1)!}
{(\dim_{\C{}}\cQ(d_1,d_2,\dots,d_k)-2)!}\ \cdot\\
\Vol \cQ_1(d_{i_1}, \dots,d_i-2, \dots, d_{i_{k_1}})\,\cdot\,
\Vol \cQ_1(d_{j_1}, \dots,d_j-2, \dots, d_{j_{k_2}})
 + o(\epsilon^2)\,.
\end{multline*}
Applying~\eqref{eq:bC:equals:Vol:over:Vol} and~\eqref{eq:b:equals:2c}
to the above expression we
obtain~\eqref{eq:ratio:Siegel:Veech:constant:dumbell}.
\end{proof}

%---------------------------------------------------------------
\subsection{Siegel--Veech  constant $\mathbf{c_{area}}$}
\label{sec:subsec:carea}

Consider   an  $\SL$-invariant  manifold  in  a  stratum  of  Abelian
differentials   or  a  $\PSL$-invariant  manifold  in  a  stratum  of
quadratic   differentials.   Denote   by   $c_{cyl}$  the  associated
Siegel--Veech constant responsible for counting the maximal cylinders
of  closed  geodesics  and  denote  by  $c_{area}$  the Siegel--Veech
constant  responsible  for counting the cylinders of closed geodesics
counted with weight
$$
\frac{\text{(area of the cylinder)}}{\text{(area of the surface)}}\,.
$$
\smallskip

\noindent
In~\cite{Vorobets} Ya.~Vorobets proved the following result:

\begin{NoNumberTheorem}[Vorobets, 2003]
For  any  connected  component $\cH^{comp}(\alpha)$ of any stratum of
Abelian    differentials    and   for   almost   any   flat   surface
$S\in\cH^{comp}_1(\alpha)$   the  ratio  of  Siegel--Veech  constants
$c_{area}/c_{cyl}$ satisfies the following relation:
$$
\cfrac{c_{area}}{c_{cyl}} = \cfrac{1}{2g-2+n} =
\cfrac{1}{\dim_{\C{}} \cH(\alpha)-1}\,.
$$
\end{NoNumberTheorem}

Note that a configuration of \^homologous saddle connections of $\CP$
involves  at  most one cylinder. The following proposition states the
Vorobets  formula  for  individual configurations involving cylinders
for  strata  of meromorphic quadratic differentials with simple poles
on $\CP$.

\begin{proposition}
\label{pr:Vorobets:constant}
For  any  stratum $\cQ_1(d_1,\dots,d_n)$ of meromorphic quadratic
differentials  with simple  poles on $\CP$  and  for any admissible
configuration $\cC$ of saddle connections involving a cylinder the
following equality holds:
$$
\frac{c_{area}(\cC)}{c(\cC)}=
\frac{1}{\dim_{\C{}} \cQ(d_1,\dots,d_n)-1}=
\frac{1}{n-3}\,.
$$
\end{proposition}
\begin{proof}
The  proof  consists  in  an elementary adjustment of the computation
from  the  previous  two sections. We will present the computation of
$c_{area}(\cC)$  for  the  ``pocket configuration'' (configuration of
type     III)     following     the    analogous    computation    in
~\S\ref{sec:subsec:sv:pocket}.    The    computation   for   the
configuration of type IV is completely analogous and is omitted.

This   time   we   have   to   compute  the  integral  of  the  ratio
$\cfrac{r_T^2}{r_S^2+r_T^2}$  of  the  area $r_T^2/2$ of the cylinder
over  the  total  area  $(r_S^2+r_T^2)/2$  of  the entire surface. We
integrate  this  expression over $\cQ_1^{\epsilon}(\alpha,\cC)$. Note
that  this  ratio  is the same for proportional surfaces. Thus we can
integrate     with     respect     to    the    corresponding    cone
$C(\cQ_1^{\epsilon}(\alpha,\cC))$:
$$
\int_{\cQ_1^{\epsilon}(\alpha,\cC)} \cfrac{r_T^2}{r_S^2+r_T^2}\,\, d\mu_1 =
\dim_{\R}\cQ(\alpha)\cdot
\int_{C(\cQ_1^{\epsilon}(\alpha,\cC))} \cfrac{r_T^2}{r_S^2+r_T^2}\,\, d\mu(S)
$$
Moreover, the ratio of the corresponding Siegel--Veech constants
satisfies
\begin{equation}
\label{eq:ratio}
\cfrac{c_{area}(\alpha,\cC)}{c_{cyl}(\alpha,\cC)}=
\lim_{\epsilon\to 0}
\frac
{\int_{C(\cQ_1^{\epsilon}(\alpha,\cC))} \cfrac{r_T^2}{r_S^2+r_T^2}\,\, d\mu(S)}
{\int_{C(\cQ_1^{\epsilon}(\alpha,\cC))}  d\mu(S)}\,.
\end{equation}
The denominator
$$
\int_{C(\cQ_1^{\epsilon}(\alpha,\cC))}  d\mu(S) =
\mu(C(\cQ_1^{\epsilon}(\alpha,\cC)))
$$
of the above ratio is given by~\eqref{eq:final:initial:integral}.
To    evaluate    the    integral    in    the    numerator    we
modify~\eqref{eq:final:initial:integral}   by   multiplying   the
function   inside    the    integral    by    an   extra   factor
$r_T^2/(r_S^2+r_T^2)$ obtaining:
\begin{multline}
\label{eq:new:integral}
\int_{C(\cQ_1^{\epsilon}(\alpha,\cC))} \cfrac{r_T^2}{r_S^2+r_T^2}\,\, d\mu(S)
=d_i\cdot\Vol(\cQ_1(\alpha'))\ \cdot\\
\cdot 2\pi\epsilon^2\cdot
\int_0^1 r_S^{2n_S-1} dr_S
\int_0^{\sqrt{1-r_S^2}} r_T^{2n_T-1}
\,\,dr_T\
 + o(\epsilon^2),
\end{multline}
Taking  into  consideration  that $n_T=\dim_\C{}\cR=2$ and evaluating
the latter integral we obtain
\begin{equation}
\label{eq:postfinal:carea:integral}
\int_{C(\cQ_1^{\epsilon}(\alpha,\cC))} \cfrac{r_T^2}{r_S^2+r_T^2}\,\, d\mu(S)
=d_i\cdot\Vol\cQ_1(\alpha')\, \cdot
\cfrac{4\pi\epsilon^2}{2n_S(2n_S+2)(2n_S+4)}+o(\epsilon^2)\,.
\end{equation}
Plugging~\eqref{eq:postfinal:carea:integral}
and~\eqref{eq:postfinal:initial:integral}                        into
expression~\eqref{eq:ratio} and recalling the definition
$$
n_S=\dim_\C{}\cQ(\alpha')=\dim_\cx\cQ(\alpha)-2
$$
we obtain
$$
\cfrac{c_{area}(\alpha,\cC)}{c_{cyl}(\alpha,\cC)}=
\cfrac{2}{2n_S+2}=\cfrac{1}{\dim_\cx \cQ(\alpha)-1}\,,
$$
which completes the proof of proposition~\ref{pr:Vorobets:constant}.
\end{proof}

Proposition~\ref{pr:Vorobets:constant}  immediately implies the following
statement.

\begin{corollary}
\label{cr:Vorobets:formula}
For  any  stratum  $\cQ_1(d_1,\dots,d_n)$  of  meromorphic  quadratic
differentials  with  simple poles on $\CP$ the Siegel--Veech constant
$c_{area}$  is  expressed  in terms of the Siegel--Veech constants of
configurations as follows:
$$
c_{area}=
\cfrac{1}{n-3}\,\cdot\,
\sum_{\substack{\mathit{Configurations}\ \cC\\
\mathit{containing\ a\ cylinder}}}
c_\cC\,.
$$
\end{corollary}

%%%%%%%%%%%%%%%%%%%%%%%%%%%%%%%%%%%%%%%%%%%%%%%%%%%%%%%%%%%%%%%%%%%%
%%%%%%%%%%%%%%%%%%from report_april29_2012
%%%%%%%%%%%%%%%%%%%%%%%%%%%%%%%%%%%%%%%%%%%%%%%%%%%%%%%%%%%%%

\section{Computation of the volumes of the moduli spaces}
\label{sec:induction}

In this section, we prove Theorem~\ref{theorem:volume}. The approach
taken here is somewhat indirect.

\subsection{An identity for the Siegel--Veech constant}

The idea is to prove

formula~\eqref{eq:volume} in Theorem~\ref{theorem:volume}

for   the   volume   by  induction,  using  the  formulas  expressing
Siegel--Veech   constants   in  terms  of  the  volumes.  Namely,  by
\cite[Theorem~3]{Eskin:Kontsevich:Zorich} one has:

$$
c_{\mathit{area}}(\cQ(d_1,\dots,d_k))=
-\cfrac{1}{8\pi^2}\,\sum_{j=1}^k \cfrac{d_j(d_j+4)}{d_j+2}\,.
$$

On  the  other  hand,  by  the Vorobets formula applied to $\CP$ (see
Corollary~\ref{cr:Vorobets:formula}  in  ~\S\ref{sec:subsec:carea})  one
has
$$
c_{\mathit{area}}(\cQ(d_1,\dots,d_k))=
\cfrac{1}{\dim_\mathbb{C}\cQ(d_1,\dots,d_k)-1}\cdot
\sum_{\substack{\mathit{Configurations}\ \cC\\
\mathit{containing\ a\ cylinder}}}
c_\cC\,.
$$

In  view  of  \S\ref{sec:subsec:configurations},  for $\CP$ there are
exactly  two configurations containing a cylinder: a ``pocket'' and a
``dumbbell''. The formulas for the Siegel--Veech constants were given
in \S\ref{sec:Siegel:Veech}.

Taking              into              consideration              that
$\dim_\mathbb{C}\cQ(d_1,\dots,d_k)-1=k-3$,    for    any   collection
$d_1,\dots,d_k$ of integers in $\{-1,1,2,3\}$ satisfying the relation
$$
\sum_{i=1}^k d_i=-4
$$
we get the following identity:
\begin{equation}
\label{eq:carea:carea}
-\cfrac{1}{8\pi^2}\,\sum_{j=1}^k \cfrac{d_j(d_j+4)}{d_j+2}
=
\cfrac{1}{k-3}\cdot\left(
\sum_{\substack{``pocket''\\configurations}} c_\cC +
\sum_{\substack{``dumbbell''\\configurations}} c_\cC
\right)\,.
\end{equation}
If we plug in the expressions
(\ref{eq:ratio:Siegel:Veech:constant:pocket}) and
(\ref{eq:ratio:Siegel:Veech:constant:dumbell}) into
(\ref{eq:carea:carea}), we get a formula of the form:
\begin{equation}
\label{eq:recurrence:volumes}
\Vol \cQ_1(d_1, \dots, d_n) = \text{ Explicit polynomial in
  volumes of simpler strata.}
\end{equation}
The formulas (\ref{eq:recurrence:volumes}) clearly determine the
volumes. Thus, to prove Theorem~\ref{theorem:volume}, it is enough to
show that the expressions for the volumes given by
Theorem~\ref{theorem:volume} satisfy the recurrence relation
(\ref{eq:recurrence:volumes}), or equivalently to prove the following:
\begin{theorem}
\label{theorem:carea:identity}
The explicit expressions (\ref{eq:Siegel:Veech:constant:pocket}) and
(\ref{eq:Siegel:Veech:constant:dumbell}) for the Siegel--Veech
constants satisfy the identity
(\ref{eq:carea:carea}).
\end{theorem}
The proof of the Theorem~\ref{theorem:carea:identity} is quite involved
and is done in Appendix~\ref{sec:appendix:identity}.
This completes the proof of Theorem~\ref{theorem:volume}.

%###################################################################
%###################################################################
%###################################################################

\section{Counting trajectories and ergodic theory on moduli space}
\label{sec:ergodic}
In  this section we will prove Theorem~\ref{theorem:SV:for:billiards}.
We   modify  appropriately  the
strategy  from ~\S\ref{sec:subsec:Reduction:to:ergodic:theory}
to obtain
the  asymptotic  formula  (\ref{eq:SV:for:billiards}). The key tool
is Theorem~\ref{theorem:birkhoff:chaika} proved by Jon Chaika in
Appendix~\ref{sec:chaika}.

%--------------------------------------------------------------------

\subsection{Pointwise asymptotics}
\label{sec:point:weak}

To understand the asymptotics for any set of special trajectories for the flat metric associated to $q \in \cQ_1$, we use (\ref{eq:circ:av}) to reduce the problem to understanding
\begin{equation}
\frac{1}{2\pi}\int_{0}^{2\pi} \widehat{f}(g_t
  r_{\theta} q)d\theta,
\end{equation}
where $\widehat{f}$ is the indicator function of the trapezoid defined in \S\ref{Siegel:Veech:formulas}. We are particularly interested in the metrics $q_{\Pi}$, $\Pi \in \cB$. If $\widehat{f} \in \mathcal L_{c}$ (in the notation of \S\ref{sec:chaika}), we could directly apply Theorem~\ref{theorem:birkhoff:chaika} to conclude that
$$\lim_{t \rightarrow \infty}  \frac{1}{2\pi}\int_{0}^{2\pi} \widehat{f}(g_t
  r_{\theta} q_{\Pi})d\theta = \frac{3}{4} b_\cC(\cQ)$$ for almost every $\Pi \in \cB$. Following an argument from~\cite{EMS}, we will approximate $\widehat{f}$ by such
functions. Fix  $\epsilon  >0$,  let  $h_{\epsilon}:  \cQ_1  \rightarrow  \R$ be a
continuous function with
\begin{equation}\label{eq:h:def}
h_{\epsilon}(q)  = \begin{cases} 1 & l(q) > \epsilon \\ 0 & l(q) <
\epsilon/2\end{cases}
\end{equation}

\noindent Here, $l(q)$ denotes the length of the shortest saddle
connection on $q$. The function $h_{\epsilon}$ is a smoothed version of the
indicator function of the compact part of the stratum $\cQ_1$. Given $\phi: \cQ_1 \rightarrow \R$, define
\begin{equation}
\label{eq:At:def}
\left(A_t\phi\right)(q) = \frac{1}{2\pi}\int_0^{2\pi} \phi(g_t r_{\theta}
q) d\theta.
\end{equation}

\noindent For any $q \in \cQ_1$,
\begin{equation}
\label{eq:ahat:bound}
\left(A_t(\widehat{f}h_{\epsilon})\right)(q) \le
\left(A_t\widehat{f}\right)(q) = \left(A_t(\widehat{f}h_{\epsilon})\right)(q)+
\left(A_t(\widehat{f}(1-h_{\epsilon}))\right)(q).
\end{equation}

\noindent We follow~\cite[p.435, proof of Theorem 2.4]{EMS} . Fix $1>\eta > \delta >0$. \cite[Theorem 5.1]{Eskin:Masur} shows there is a $C(\delta)$ so that for all $q \in \cQ$
\begin{equation}\label{eq:fhat:bound}\widehat{f}(q) \le
\frac{C(\delta)}{l(q)^{1+\delta}}\,.
\end{equation}
On the other hand, $1-h_{\epsilon}(q) > 0$ implies $l(q) \le \epsilon$, so
$$
\widehat{f}(q)(1-h_{\epsilon}(q))
\le \widehat{f}(q)\le
\frac{C(\delta)}{l(q)^{1+\eta}}
\cdot l(q)^{\eta-\delta}
\le \epsilon^{\eta - \delta} \frac{C(\delta)}{l(q)^{1+\eta}}\,.
$$

Thus,
$$\left(A_t(\widehat{f}(1-h_{\epsilon}))\right)(q) \le C(\delta)\epsilon^{\eta
-\delta}\left(A_tl^{-1-\eta}\right)(q).$$

\noindent\cite[Theorem 5.2]{Eskin:Masur} states that for $\eta <1$, there is a $C_1 = C_1(\eta, \Pi)$ so that for all $t>0$,
$$\left(A_t l^{-1-\eta}\right)(q_{\Pi}) < C_1(\eta,\Pi)\,.$$

Since  $\widehat{f}h_{\epsilon}$  is  continuous and compactly
supported,
for any $\Pi$ from the set of full measure to which
Theorem~\ref{theorem:birkhoff:chaika} applies we get
$$
\lim_{t \rightarrow \infty}
 \left(A_t (\widehat{f}h_{\epsilon})\right)(q_{\Pi})\,
= \frac{1}{\mu_1(\cQ_1)} \int_{\cQ_1} \widehat{f}h_{\epsilon}(q)\,d\mu_1(q)
$$

\noindent So we have
\begin{equation}
\label{eq:atliminf}
\liminf_{t \rightarrow \infty}
 A_t \widehat{f}(q_{\Pi}) \geq
\frac{1}{\mu_1(\cQ_1)}\int_{\cQ_1} \widehat{f}h_{\epsilon}(q)\,d\mu_1(q)
\end{equation}
\noindent and
\begin{equation}\label{eq:atlimsup}
\limsup_{t \rightarrow \infty} A_t\widehat{f}(q_{\Pi})
\le \frac{1}{\mu_1(\cQ_1)}\int_{Q_1} \widehat{f} h_{\epsilon}(q)\, d\mu_1(q)
\, +\, C(\delta)C_1(\eta,\Pi)\epsilon^{\eta - \delta}\,.
\end{equation}

\noindent Combining  (\ref{eq:atliminf}) and (\ref{eq:atlimsup}) and
letting   $\epsilon   \rightarrow   0$,   we   obtain, as   desired, Theorem~\ref{theorem:SV:for:billiards}

\qed

\medskip

\noindent We     complete     this     section     with     the     proof    of
Theorem~\ref{theorem:Narea:weak:asymp}                           from
~\S\ref{sec:subsec:introduction:beginning}.
\begin{proof}[Proof of Theorem~\ref{theorem:Narea:weak:asymp}]
The    proof    is    completely    analogous   to   the   proof   of
Theorem~\ref{theorem:SV:for:billiards}   above;   we   just  have  to
carefully  follow  the  normalization  which  is  different  from the
previous case.

By~\cite[Theorem~3]{Eskin:Kontsevich:Zorich} one has:
\begin{multline}
\label{eq:carea:surface}
c_{\mathit{area}}(\cQ(k_1-2,\dots,k_n-2))=
\\
-\cfrac{1}{8\pi^2}\,\sum_{j=1}^n \cfrac{(k_j-2)(k_j+2)}{k_j}=
\cfrac{1}{8\pi^2}\,
\sum_{j=1}^n\left(\cfrac{4}{k_j}-k_j\right)
\,.
\end{multline}
The  length of a closed trajectory in $\Pi$ is the same as the length
of the associated closed geodesic on the covering flat sphere $S$. By
definition,  the  area of the band of closed trajectories on $\Pi$ is
the  same  as  the  area  of  each of the maximal cylinders on $\CP$.
However,  the  flat area of $S$ is twice the area of $\Pi$. Thus, the
ratio
$$
\frac{(\text{area of the band of periodic trajectories on }\Pi)}{(\text{area of }\Pi)}
$$
is twice larger then the corresponding ratio
$$
\frac{(\text{area of the maximal cylinder of periodic geodesics on }S)}{(\text{area of }S)}\,.
$$
Taking  into  consideration  that the bands of closed trajectories in
the  polygon  $\Pi$ are in the natural one-to-two correspondence with
the  maximal  cylinders  of  closed regular geodesics on the covering
flat sphere $S$, see Figure~\ref{fig:billiard:and:sphere}, we get
$$
\Narea(\Pi,L)=\Narea(S,L)\,.
$$
It remains to note that
$$
\Narea(S,L)\sim
c_{\mathit{area}}(\cQ)
\cdot
\cfrac{L^2}{\area(S)}=
\frac{c_{\mathit{area}}(\cQ)}{2}
\cdot
\cfrac{L^2}{\area(\Pi)}
$$
to    conclude    that   the   constant   in   the   weak   quadratic
asymptotic~\eqref{eq:billiard:carea}                               in
Theorem~\ref{theorem:Narea:weak:asymp}    is    one   half   of   the
Siegel--Veech            constant            $c_{\mathit{area}}(\cQ)$
from~\eqref{eq:carea:surface}.
\end{proof}

% %--------------------------------------------------------------------
% \subsection{$L^1$ results}
% \label{sec:l1}

% Here, we use (\ref{eq:leaf:ave}) from
% Corollary~\ref{cor:birkhoff:average}. Again letting $f$ be the
% indicator function of the trapezoid $\mathcal T$, defined above, and using the
% reduction (\ref{eq:circ:av}), we can reduce studying
% \begin{displaymath}
% \lim_{t \rightarrow \infty}\int_A e^{-2t}N_\cC(\Pi, e^t)\, d\mu_{\cB}(\Pi)
% \end{displaymath}
% to studying
% \begin{displaymath}
% \lim_{t \rightarrow \infty}
% \frac{4}{3}\pi \frac{1}{\mu_{\cB}(A)} \int_A \frac{1}{2\pi}\int_0^{2\pi} \widehat{f}(g_T
% r_{\theta} q_{\Pi})\, d\mu_{\cB}(\Pi),
% \end{displaymath}
% where $\widehat{f}$ is as in (\ref{eq:siegel:veech:def}).  As in
% \S\ref{sec:point:weak}, we approximate $\widehat{f}$ by the compactly supported
% function $\widehat{f} h_{\epsilon}$, and applying the equidistribution result
% (\ref{eq:leaf:ave}), we have
% \begin{displaymath}
% \lim_{t \rightarrow \infty} \frac{4}{3}\pi \frac{1}{\mu_{\cB}(A)} \int_A \frac{1}{2\pi}\int_0^{2\pi} \widehat{f}(g_T
% r_{\theta} q_{\Pi})\, d\mu_{\cB}(\Pi),
% =
% \frac{4}{3}\pi \frac{1}{\mu_1(\cQ_1)}  \int_{\cQ_1} f(q)\, d\mu_1(q) = \pi b_\cC(\cQ)\,.
% \end{displaymath}

% Theorem~\ref{theorem:l1poly} is proved.

%####################################################################
%####################################################################
%####################################################################

\appendixmode

\section{Proof of combinatorial identity}
\label{sec:appendix:identity}

\noindent In this appendix, we prove Theorem~\ref{theorem:carea:identity}.
Our   proof   follows   the   following   scheme.  In
section~\ref{ss:General:identity:to:prove} we rewrite the conjectural
identity~\eqref{eq:carea:carea}      in      a      more     detailed
form~\eqref{eq:end2011}   and   then  applying  elementary  algebraic
manipulations  we rewrite it again in the form~\eqref{eq:apr2012}. In
section~\ref{sec:subsec:multinomial:coefficients}  we  rearrange  the
summation  in~\eqref{eq:apr2012}  and in section~\ref{ss:Notation} we
introduce  multiindex notation. Combining this rearrangement with new
notation    we    rewrite    the    conjectural   identity   in   the
form~\eqref{eq:identity:Dec:2014}.

In section~\ref{ss:Generating:Functions} we introduce
generating  functions  $F(\sss)$  and  $G(\sss)$  as  power series in
(multi)variable  $\sss$ with coefficients involved in the conjectural
combinatorial     identity~\eqref{eq:identity:Dec:2014}.     The    desired
combinatorial  identity~\eqref{eq:identity:Dec:2014} now wraps to the
identity  $F^2(\sss)  \stackrel{?}{=} G(\sss)$. At this stage we have
just  gained  a  more  concise  and  convenient  form  for  the
conjectural identity, nothing serious has happened.

In  section~\ref{ss:Mohanty:Formula}  we introduce an
entire  collection  of  auxiliary  generating  functions  $M_a(\sss)$
indexed  by  a  positive  integer  $a$  (and  depending on an integer
multiindex parameter $\bbb$):
$$
M_a(\sss):=\sum_{\kk   \in (\Z^{\geq 0})^m} A(a; \bbb; \kk) \sss^{\kk}\,,
$$
where  the  \textit{Mohanty coefficient} $A(a; \bbb; \kk)$ is defined
in     section~\ref{ss:Mohanty:Formula}.     By    a   theorem    of
\mbox{Mohanty}~\cite{Mohanty},    \textit{all}    these    generation
functions  $M_a(\sss)$  are expressed in terms of $M_1(\sss)$ denoted
by  $z(\sss):=M_1(\sss)$.  Moreover,  \textit{all}  these  generation
functions  are  expressed in terms of $z(\sss)$ in a very simple way,
namely,
$$
M_a(\sss) = z^a(\sss)\,.
$$
By the same theorem of Mohanty, the basic generating function $z(\sss)$
satisfies the functional relation
\begin{equation}
\tag{$\ast$}
1- z + \sum_{i=1}^{m} s_i z^{b_i} = 0\,.
\end{equation}
Note that this relation is \textit{polynomial} in the basic generating
function $z$ and in formal variables $\sss=(s_1,\dots,s_m)$.

The   strategy   of  the  proof  is  to  express  our
generating  functions  $F(\sss)$  and  $G(\sss)$  as  polynomials  in
Mohanty  functions  $M_a(\sss)  =  z^a(\sss)$  and  formal  variables
$\sss$.  As  soon  as  we  get  the corresponding expressions for $F$
(Lemma~\ref{lemma:F}      in     section~\ref{sec:F})     and     $G$
(Lemma~\ref{lemma:G}   in   section~\ref{sec:G})   we   express   the
difference $G-F^2$ as a polynomial in $z$ and formal variables $\sss$
and  show (Theorem~\ref{theorem:id}) that in the resulting polynomial
one  can  factor  out  the  square  of  expression~$(\ast)$. Since by
Mohanty's  Theorem  this  expression is identically zero, this proves
that  $G-F^2$  is  identically  zero,  which  completes  the proof of
Theorem~\ref{theorem:carea:identity}.

%-----------------------------------------------------
\subsection{General identity to prove}
\label{ss:General:identity:to:prove}

Let $d_1,\dots,d_m$ be the degrees of zeroes only. Let the number $n$
of  simple  poles  is  expressed as $n=4+\sum_{i=1}^m d_i$. The total
number   $k=m+n$   of   all  singularities  is,  thus,  expressed  as
$k=4+\sum_{i=1}^m (d_i+1)$.

Recall  that all zeroes and poles are \textit{named}.
A  ``pocket''  configuration  is  uniquely  defined  by a choice of a
distinguished  zero  (at the base of the cylinder) and by a choice of
an  unordered  pair  of simple poles (corners of the ``pocket''); all
choices  of  a  zero  and of a pair of poles are admissible. When the
distinguished  zero  at  the  base  of the cylinder has degree $d_i$,
formula~\eqref{eq:Siegel:Veech:constant:pocket}  gives  the following
value  for  the  Siegel--Veech  constant  $c_\cC$  for  an individual
``pocket''  configuration  (with distinguished zero and distinguished
pair of fixed poles):
$$
c_\cC =\cfrac{d_i+1}{2(\sum_{i=1}^m (d_i+1))}\cdot\cfrac{1}{\pi^2}\,,
$$
where    we    have    replaced   $k-4$   in   the   denominator   of
formula~\eqref{eq:Siegel:Veech:constant:pocket}  by $k-4=\sum_{i=1}^m
(d_i+1)$.   For  each  choice  of  the  zero  in the ``pocket''
configuration there are
$$
\begin{pmatrix}n\\2\end{pmatrix}=
\begin{pmatrix}4+\sum_{i=1}^m d_i\\2\end{pmatrix}
$$
ways  to chose a pair of distinguished poles. Hence, the total impact
of    all    ``pocket''   configurations   to   the   right-hand-side
of~\eqref{eq:carea:carea}                  (based                  on
formula~\eqref{eq:Siegel:Veech:constant:pocket}) can be written as
\begin{multline*}
\cfrac{1}{k-3}\cdot
\sum_{\substack{``pocket''\\configurations}} c_\cC
=\\=
\cfrac{1}{\left(1+\sum_{i=1}^m (d_i+1)\right)}\cdot
\left(\begin{array}{c} 4+\sum_{i=1}^m d_i \\[-\halfbls] \\ 2 \\ \end{array}\right)
\sum_{i=1}^m
\cfrac{d_i+1}{2(\sum_{i=1}^m (d_i+1))}\cdot\cfrac{1}{\pi^2}
=\\=
\cfrac{1}{2\pi^2}\cdot
\cfrac{1}{\left(1+\sum_{i=1}^m (d_i+1)\right)}\cdot
\left(\begin{array}{c} 4+\sum_{i=1}^m d_i \\[-\halfbls] \\ 2 \\ \end{array}\right)\,.
\end{multline*}

A  ``dumbbell''  configuration  $\cC$  is  uniquely
defined  by  a choice of the following data. We need to choose zeroes
go  to  one  part of the ``dumbbell''; all the remaining zeroes go to
the  complementary  part.  In  other  words,  we have to consider all
partitions of the set $\{1,\dots,m\}$ enumerating the zeroes into two
nonempty                     complementary                    subsets
$\{i_1,\dots,i_{m_1}\}\sqcup\{j_1,\dots,j_{m_2}\}$.   For  each  such
partition we have to consider all possible choices of a distinguished
zero  (at  the base of the cylinder) in each of the two groups. After
that,   we   have  to  choose  $n_1=2+\sum_{i=1}^{m_1}  d_i$  out  of
$n=4+\sum_{i=1}^m d_i$ simple poles which go to the first part of the
``dumbbell'';  the  remaining simple poles go to the other part. When
all  these  data are chosen and when the distinguished two zeros (one
in  each  of the two groups) at the base of the cylinder have degrees
$d_i,  d_j$,  formula~\eqref{eq:Siegel:Veech:constant:dumbell}  gives
the  following  value  for the Siegel--Veech constant $c_\cC$ for the
individual ``dumbbell'' configuration:
$$
c_\cC =\cfrac{(d_i+1)(d_j+1)}{2}\cdot
\cfrac{\left(-1+\sum_{i=1}^{m_1} (d_i+1)\right)!\,
\left(-1+\sum_{j=1}^{m_2} (d_j+1)\right)!}
{\left(\sum_{i=1}^m (d_i+1)\right)!}
\cdot\cfrac{1}{\pi^2}\,.
$$
Here     we     have     replaced     $k$    in    the    denominator
of~\eqref{eq:Siegel:Veech:constant:dumbell}    by   $k=4+\sum_{i=1}^m
(d_i+1)$, and have replaced $k_1$ and $k_2$ in the numerator
of~\eqref{eq:Siegel:Veech:constant:dumbell}                        by
$k_1=m_1+n_1=2+\sum_{i=1}^{m_1}         (d_i+1)$        and        by
$k_2=m_2+n_2=2+\sum_{j=1}^{m_2} (d_j+1)$ correspondingly.

For   each   partition   of   the   set
$\{1,\dots,m\}$  enumerating the  zeroes   into   two   nonempty   subsets
$\{i_1,\dots,i_{m_1}\}\sqcup\{j_1,\dots,j_{m_2}\}$
(which   makes  part  of  the  ``dumbbell''  configuration) there are
$$\cfrac{\big(4+\sum_{i=1}^m d_i\big)!}
{\big(2+\sum_{i=1}^{m_1} d_i\big)! \big(2+\sum_{j=1}^{m_2} d_j\big)!}
$$
ways  to  partition  the  simple  poles  between  two  parts  of  the
``dumbbell''. Taking into consideration this counting and plugging in
the    explicit conjectural  expressions~\eqref{eq:Siegel:Veech:constant:pocket}
and~\eqref{eq:Siegel:Veech:constant:dumbell}  for  the  Siegel--Veech
constants $c_\cC$ into~\eqref{eq:carea:carea} we observe that
the right-hand-side of~\eqref{eq:carea:carea} can be read as
\begin{multline*}
\cfrac{1}{k-3}\cdot\left(
\sum_{\substack{``pocket''\\configurations}} c_\cC +
\sum_{\substack{``dumbbell''\\configurations}} c_\cC
\right)
\ \stackrel{?}{=}\\\stackrel{?}{=}\
\cfrac{1}{2\pi^2}\cdot
\cfrac{1}{\left(1+\sum_{i=1}^m (d_i+1)\right)}\cdot
\left(
\left(\begin{array}{c} 4+\sum_{i=1}^m d_i \\[-\halfbls] \\ 2 \\ \end{array}\right)
\right.
\ +\\ \\+\
\sum_{1\le i<j \le m}(d_i+1)(d_j+1)
\sum_{\substack{
\text{partitions of }\{1,...,m\}\text{ into}\\
\{i_1,\dots,i_{m_1}\}\sqcup
\{j_1,\dots,j_{m_2}\}\\
\text{such that }i\text{ is in the first subset}\\
\text{ and } j\text{ is in the second subset}}}
  % \\ \\
\frac{\big(4+\sum_{i=1}^m d_i\big)!}
{\big(2+\sum_{i=1}^{m_1} d_i\big)! \big(2+\sum_{j=1}^{m_2} d_j\big)!}
\cdot
   \\ \\
\left.
\frac
{\big(-1+\sum_{i=1}^{m_1} (d_i+1)\big)!
\big(-1+\sum_{j=1}^{m_2}(d_j+1)\big)!}
{\big(\sum_{i=1}^m (d_i+1)\big)!}
\right)\,.
\end{multline*}

\noindent Multiplying both parts of the
conjectural identity by the common factor
$$
4\pi^2\cdot\left(
1+\sum_{i=1}^m(d_i+1)
\right)
$$
moving the binomial coefficient
$$
\left(\begin{array}{c} 4+\sum_{i=1}^m d_i \\[-\halfbls] \\ 2 \\ \end{array}\right)
$$
coming from the ``pocket'' configuration
to the left-hand-side of the identity and simplifying the resulting
expressions we
get the following conjectural identity:

\begin{multline}
\label{eq:end2011}
\left(
6+\sum_{i=1}^m\frac{d_i(d_i+1)}{d_i+2}
\right)\cdot
\left(
1+\sum_{i=1}^m(d_i+1)
\right)
-
\left(
4+\sum_{i=1}^m d_i
\right)
\left(
3+\sum_{i=1}^m d_i
\right)
\stackrel{?}{=}\\ \\ \stackrel{?}{=}
2\cdot\frac{\big(4+\sum_{i=1}^m d_i\big)!}
{\big(\sum_{i=1}^m (d_i+1)\big)!}\ \cdot
\sum_{1\le i<j \le m}(d_i+1)(d_j+1)\ \cdot
\\ \\
\cdot \sum_{\substack{
\text{partitions of }\{1,...,m\}\text{ into}\\
\{r_1,\dots,r_{m_1}\}\sqcup
\{s_1,\dots,s_{m_2}\}\\
\text{such that }i\text{ is in the first subset}\\
\text{ and } j\text{ is in the second subset}}}
% \\ \\
% \left.
\frac
{\big(-1+\sum_{i=1}^{m_1} (d_{r_i}+1)\big)!\cdot
\big(-1+\sum_{j=1}^{m_2}(d_{s_j}+1)\big)!}
{\big(2+\sum_{i=1}^{m_1} d_{r_i}\big)!\cdot
\big(2+\sum_{j=1}^{m_2} d_{s_j}\big)!}
\,.
\end{multline}
This is the identity which we need to prove.

Changing the order of the summation we can first sum over
all possible partitions of the set of indices $\{1,\dots,m\}$
and having chosen the partition we consider all possible ways
to select a distinguished element $i$ in the first subset and
a distinguished element $j$ in the second subset.
Note, however, that we will see each of the elements
of the above sum twice. Thus,
collecting the
resulting sums we can rewrite the sum in the
right-hand-side of the above expression as follows:
\begin{multline*}
\sum_{\substack{
\text{partitions of }\{1,...,m\}\text{ into}\\
\text{two nonempty sets}\\
\{r_1,\dots,r_{m_1}\}\sqcup
\{s_1,\dots,s_{m_2}\}
}}
% \\ \\
% \left.
\frac
{\big(-1+\sum_{i=1}^{m_1} (d_{r_i}+1)\big)!\cdot
\big(-1+\sum_{j=1}^{m_2}(d_{s_j}+1)\big)!}
{\big(2+\sum_{i=1}^{m_1} d_{r_i}\big)!\cdot
\big(2+\sum_{j=1}^{m_2} d_{s_j}\big)!}
\,\cdot
\\
\cdot\left(
\sum_{1\le i\le m_1}(d_{r_i}+1)
\right)\left(
\sum_{1\le j\le m_2}(d_{s_j}+1)
\right)\,\stackrel{}{=}
\\ \\ \stackrel{}{=}
\sum_{\substack{
\text{partitions of }\{1,...,m\}\text{ into}\\
\text{two nonempty sets}\\
\{r_1,\dots,r_{m_1}\}\sqcup
\{s_1,\dots,s_{m_2}\}
}}
\frac
{\big(\sum_{i=1}^{m_1} (d_{r_i}+1)\big)!\cdot
\big(\sum_{j=1}^{m_2}(d_{s_j}+1)\big)!}
{\big(2+\sum_{i=1}^{m_1} d_{r_i}\big)!\cdot
\big(2+\sum_{j=1}^{m_2} d_{s_j}\big)!}
\,.
\end{multline*}
Hence, we can rewrite the right-hand-side in~\eqref{eq:end2011} as
a sum over the ratios of binomial coefficients:
\begin{multline*}
\frac{\big(4+\sum_{i=1}^m d_i\big)!}
{\big(\sum_{i=1}^m (d_i+1)\big)!}\ \cdot
\sum_{\substack{
\text{partitions of }\{1,...,m\}\text{ into}\\
\text{two nonempty sets}\\
\{r_1,\dots,r_{m_1}\}\sqcup
\{s_1,\dots,s_{m_2}\}
}}
\frac
{\big(\sum_{i=1}^{m_1} (d_{r_i}+1)\big)!\cdot
\big(\sum_{j=1}^{m_2}(d_{s_j}+1)\big)!}
{\big(2+\sum_{i=1}^{m_1} d_{r_i}\big)!\cdot
\big(2+\sum_{j=1}^{m_2} d_{s_j}\big)!}
\,\stackrel{}{=}
      \\ \stackrel{}{=}
\sum_{\substack{
\text{partitions of }\{1,...,m\}\text{ into}\\
\text{two nonempty sets}\\
\{r_1,\dots,r_{m_1}\}\sqcup
\{s_1,\dots,s_{m_2}\}
}}
\frac{
\left(\begin{array}{c}
4+\sum_{i=1}^m d_i \\
[-\halfbls]\\
2+\sum_{l=1}^{m_1} d_{r_l} \\ \end{array}\right)
}{
\left(\begin{array}{c}
\sum_{i=1}^m (d_i+1) \\
[-\halfbls]\\
\sum_{l=1}^{m_1} (d_{r_l}+1) \\ \end{array}\right)
}\,.
\end{multline*}

Finally, omitting the conditions that the subsets of the partition
are nonempty, we get two extra terms. It is immediate to verify
that their sum is equal to
$$
\left(
4+\sum_{i=1}^m d_i
\right)
\left(
3+\sum_{i=1}^m d_i
\right)
$$
and, thus, we can rewrite the needed conjectural
identity~\eqref{eq:end2011}
as follows:
\begin{multline}
\label{eq:apr2012}
\left(
6+\sum_{i=1}^m\frac{d_i(d_i+1)}{d_i+2}
\right)\cdot
\left(
1+\sum_{i=1}^m(d_i+1)
\right)\ \stackrel{?}{=}
\\ \stackrel{?}{=}\
\sum_{\substack{
\text{partitions of }\{1,...,m\}\text{ into}\\
\text{two complementary sets}\\
\{r_1,\dots,r_{m_1}\}\sqcup
\{s_1,\dots,s_{m_2}\}
}}
\frac{
\left(\begin{array}{c}
4+\sum_{i=1}^m d_i \\
[-\halfbls]\\
2+\sum_{l=1}^{m_1} d_{r_l} \\ \end{array}\right)
}{
\left(\begin{array}{c}
\sum_{i=1}^m (d_i+1) \\
[-\halfbls]\\
\sum_{l=1}^{m_1} (d_{r_l}+1) \\ \end{array}\right)
}\,.
\end{multline}

We will show that (\ref{eq:apr2012})  is valid for any nonempty
collection of nonnegative integers $\{d_1,\dots,d_m\}$.

%---------------------------------------------------------------------
\subsection{Identity in terms of multinomial coefficients.}
\label{sec:subsec:multinomial:coefficients}

Let $n_d$ be the total number of entries $d$ in the set
(with multiplicities) $\{d_1,\dots,d_m\}$. The left-hand-side of
conjectural identity~\eqref{eq:apr2012} can be expressed as
$$
\left(
6+\sum_{d}\frac{d(d+1)}{d+2}\,n_d
\right)\cdot
\left(
1+\sum_{d}(d+1)n_d
\right)\,.
$$

The right-hand-side can be represented in terms of $n_d$ as
\begin{multline*}
\cfrac
{\left(4+\sum_{d}d\cdot n_d\right)!}
{\left(\sum_{d}(d+1)n_d\right)!}
\cdot
\sum_{k_1=0}^{n_1}\sum_{k_2=0}^{n_2}\cdots
\left(\begin{array}{c}n_1\\k_1\end{array}\right)
\left(\begin{array}{c}n_2\\k_2\end{array}\right)
\cdots
\\
\cfrac
{\left(\sum_{d}(d+1)k_d\right)!\cdot
\left(\sum_{d}(d+1)(n_d-k_d)\right)!}
{\left(2+\sum_{d}d\cdot k_d\right)!\cdot
\left(2+\sum_{d}d\cdot (n_d-k_d)\right)!}\ =
\\ \\
=
\cfrac
{n_1!\,n_2!\,\dots\left(4+\sum_{d}d\cdot n_d\right)!}
{\left(\sum_{d}(d+1)n_d\right)!}\cdot
\sum_{k_1=0}^{n_1}\sum_{k_2=0}^{n_2}\cdots
\cdots
\\
\cfrac
{\left(\sum_{d}(d+1)k_d\right)!}
{k_1!\,k_2!\cdot\dots\cdot\left(2+\sum_{d}d\cdot k_d\right)!}
\cdot
\cfrac
{\left(\sum_{d}(d+1)(n_d-k_d)\right)!}
{(n_1-k_1)!\,(n_2-k_2)!\cdot\dots\cdot
\left(2+(\sum_{d}d\cdot (n_d-k_d)\right)!}\,.
\end{multline*}

Note now that the common factor is (up to four missing factors) is a
multinomial coefficient and that the bottom line is a product of two
``complementary'' multinomial coefficients (with two missing factors
each).

%----------------------------------------------------------------
\subsection{Notation}
\label{ss:Notation}
To simplify the otherwise complicated
factorials and terms, we introduce some notation: $\kk = (k_1, \ldots,
k_m), \sss = (s_1, \ldots, s_m)$, and $\dd = (d_1, \ldots, d_m)$ are
all $m$-tuples. We will think of $\kk, \dd \in (\Z^{\geq 0})^m$, and
$\sss$ as variables. We write $$\one = (1, 1, \ldots, 1) :=
\sum_{i=1}^m \ee_i,$$ where $\ee_i$ are the standard basis
vectors. Let $n$ denote an integer.

\begin{description}

\item[Inner Product] $$\kk \cdot \dd := \sum_{i=1}^m k_i d_i$$ is the standard inner product.
\medskip
\item[Factorials] $$\kk! := \Pi_{i=1}^m k_i !$$
\medskip
\item[Multinomial Coefficients] $$\binom{n}{\kk} : = \binom{n}{k_1, \ldots, k_m, n- \kk\cdot\one}$$
\medskip
\item[Deletion of variables] Here, we can have $i=j$: $$\kk^{i} = \kk - \ee_i, \kk^{i,j}= \kk - \ee_i - \ee_j$$
\medskip
\item[Powers] $$\sss^{\kk} = \Pi_{i=1}^m s_i ^{k_i}$$
\end{description}
\medskip

We  redefine  notations $\mathbf{d}$  and $m$ denoting from now on by
$\mathbf{d}$  the  original  set  $\{d_1,\dots,d_m\}$ with suppressed
multiplicities. In other words, we define the new
$\mathbf{d}$ as the set of distinct entries of the original set $\{d_1,\dots,d_m\}$.
We  also  redefine $m$ denoting by $m$ the cardinality of the new
set $\mathbf{d}$. Applying manipulations
performed in ~\S\ref{sec:subsec:multinomial:coefficients}
we can rewrite the identity we need to prove in the following way:
\begin{multline}
\label{eq:identity:Dec:2014}
\cfrac{
6+\sum_{i=1}^m\cfrac{d_i(d_i+1)}{d_i+2}\,n_i
}{\Big(2+(\dd+\one)\cdot \nn\Big)\cdot
\Big(3+(\dd+\one) \cdot \nn \Big)\cdot
\Big(4+(\dd+\one) \cdot \nn \Big)}\,\cdot
\left(\begin{array}{c}
4+(\dd+\one)\cdot \nn\\
[-\halfbls]\\
\nn
\end{array}\right) \stackrel{?}{=} \\
\stackrel{?}{=}
\sum_{\kk=0}^{\nn}
\cfrac{1}{\Big(1+(\dd+\one)\cdot \kk\Big)\Big(2+(\dd+\one)\cdot \kk \Big)}
\,\cdot
\left(\begin{array}{c}
2+(\dd+\one)\cdot \kk\\
[-\halfbls]\\
\kk
\end{array}\right)\cdot
\\ \cdot\,
\cfrac{1}{\Big(1+(\dd+\one)\cdot (\nn-\kk)\Big)\Big(2+(\dd+\one)(\nn-\kk)\Big)}
\,\cdot
\left(\begin{array}{c}
2+(\dd+\one)\cdot (\nn-\kk)\\
[-\halfbls]\\
\nn-\kk
\end{array}\right)\,.
\end{multline}

%---------------------------------------------------------------------
\subsection{Generating Functions}
\label{ss:Generating:Functions}

We define
$$
F(\sss) : = \sum_{\kk \in (\Z^{\geq 0})^m} \frac{ \binom{2
    + (\dd+\one)\cdot \kk}{\kk} }{(1 + (\dd+\one)\cdot\kk)(2 +
  (\dd+\one)\cdot\kk)} \sss^{\kk}\,,
$$
and
\begin{multline*}
G(\sss) := \sum_{\kk \in
  (\Z^{\geq 0})^m} \frac{ 6 + \sum_{i=1}^m \frac{d_i (d_i +1)}{d_i+2}
  k_i}{(2 + (\dd+\one)\cdot\kk)(3 + (\dd+\one)\cdot\kk)(4 +
  (\dd+\one)\cdot\kk)}\,\cdot
\\
\cdot\binom{4 + (\dd+\one)\cdot \kk}{\kk}
\sss^{\kk}\,.
\end{multline*}

In terms of these generating functions, the conjectural
identity to be proved becomes
$$F^2 \stackrel{?}{=} G\,.$$

%---------------------------------------------------------------------
\subsection{Mohanty's Formula}
\label{ss:Mohanty:Formula}
Our main tools are the combinatorial identities developed by \mbox{Mohanty}~\cite{Mohanty}.
We recall formulas (31) and (32) of~\cite{Mohanty}, in our own
notation. Given $a \in \N$, $\bbb, \kk \in (\Z^{\geq 0})^m$, define
the \textit{Mohanty coefficient} $$A(a; \bbb; \kk) := \frac{a}{a+
  \bbb\cdot\kk} \binom{a+\bbb\cdot\kk}{\kk}.$$ We have

\begin{theorem}\label{Mh}[Mohanty~\cite{Mohanty}, (31) and (32)]
With notation as above, we have $$\sum_{\kk \in (\Z^{\geq 0})^m} A(a; \bbb; \kk) \sss^{\kk} = z^a,$$ where $$1- z + \sum_{i=1}^{m} s_i z^{b_i} = 0.$$
\end{theorem}
\medskip
\noindent Since we will use only one $\bbb$, namely $\bbb = \dd + \one$, we will abbreviate the Mohanty coefficient  by defining $A(a;\kk) = A(a; \dd+\one; \kk)$.

In the rest of the appendix we prove:

\begin{theorem}
\label{theorem:id}
$$F^2 = G.$$ More precisely, \begin{equation}\label{eq:id}G(\sss) = F^2(\sss) - \frac{1}{4}z^2 \left(1-z + \sum_{i=1}^m s_i z^{d_i +1}\right)^2\,,\end{equation} where $z$ is as in Mohanty's formula Theorem~\ref{Mh} for $A(a;\kk) = A(a; \dd+\one; \kk)$, so $$1-z + \sum_{i=1}^m s_i z^{d_i +1} = 0.$$
\end{theorem}
\medskip
\noindent To prove this formula, we will derive formulas for $F$ (\S\ref{sec:F}) and $G$ (\S\ref{sec:G}), and show (\ref{eq:id}).

%---------------------------------------------------------------------
\subsection{Formula for $F$}
\label{sec:F} \noindent Our first lemma is the formula for $F$:

\begin{lemma}\label{lemma:F} $$F(\sss) = \sum_{i=1}^{m} \frac{s_i}{d_i+2} z^{d_i+2} - \frac{1}{2}z^2 + z, $$ where $$1-z + \sum_{i=1}^m s_i z^{d_i +1} = 0\,.$$

 \end{lemma}

 \medskip

 % \noindent\textbf{Remark:} If you differentiate this formula for $F$ naively with respect to $z$, you get the identity $1-z + \sum_{i=1}^m s_i z^{d_i +1} = 0.$ Not sure what (if any) meaning this has.

\begin{proof} We expand the right hand side using Mohanty's formula, and equalize the $\sss^{\kk}$ terms of the right hand and left hand sides. The right hand side expands, term-by-term, as:

$$\sum_{i=1}^{m} \frac{s_i}{d_i+2} z^{d_i+2} \longmapsto \sum_{i=1}^{m} \frac{s_i}{d_i +2} \frac{d_i+2}{d_i + 2 + (\dd+\one)\cdot \kk} \binom{d_i + 2 + (\dd+\one)\cdot \kk}{\kk}$$

$$\frac{1}{2}z^2 \longmapsto \frac{1}{2} \frac{2}{2 + (\dd+\one)\cdot \kk}\binom{2+ (\dd+\one)\cdot \kk}{\kk}$$

$$z \longmapsto  \frac{1}{1 + (\dd+\one)\cdot \kk}\binom{1+ (\dd+\one)\cdot \kk}{\kk}\,.$$
The $\sss^{\kk}$ terms of each of the second and third expressions can be read off directly. For the first, we have:

$$\sum_{i=1}^{m} \frac{s_i}{d_i+2} z^{d_i+2} \longmapsto \sum_{i=1}^{m} \frac{1}{d_i+ 2 + (\dd+\one)\cdot \kk^{i}} \binom{d_i + 2 + (\dd+\one)\cdot \kk^{i}}{\kk^{i}}\,.$$
Observing that $$a + d_i + (\dd + \one) \cdot \kk^i = a-1 + (\dd+\one) \cdot \kk,$$ we can re-write this as

$$\sum_{i=1}^{m} \frac{s_i}{d_i+2} z^{d_i+2} \longmapsto \sum_{i=1}^{m} \frac{1}{1+ (\dd+\one)\cdot \kk} \binom{1+ (\dd+\one)\cdot \kk}{\kk^{i}}\,.$$
Thus, our identity reduces to showing that :

$$\frac{ \binom{2 + (\dd+\one)\cdot \kk}{\kk} }{(1 + (\dd+\one)\cdot\kk)(2 + (\dd+\one)\cdot\kk)}$$
is the sum of
$$\sum_{i=1}^{m} \frac{1}{1+ (\dd+\one)\cdot \kk} \binom{1+ (\dd+\one)\cdot \kk}{\kk^{i}},$$
and
$$\frac{1}{1 + (\dd+\one)\cdot \kk}\binom{1+ (\dd+\one)\cdot \kk}{\kk} - \frac{1}{2 + (\dd+\one)\cdot \kk}\binom{2+ (\dd+\one)\cdot \kk}{\kk}\,.$$
Multiplying through by $$(1 + (\dd+\one)\cdot\kk)(2 + (\dd+\one)\cdot\kk),$$
our identity reduces to showing that
$$\binom{2 + (\dd+\one)\cdot \kk}{\kk}$$ equals
\begin{multline*}
(2+ (\dd+\one)\cdot\kk) \left(\binom{1+ (\dd+\one)\cdot\kk}{\kk} + \sum_{i=1}^m \binom{1+ (\dd+\one)\cdot\kk}{\kk^i}\right) -\\
 (1+ (\dd+\one)\cdot\kk)\binom{2+ (\dd+\one)\cdot\kk}{\kk}\,.
\end{multline*}
Moving the last term to the left hand side, and canceling the resulting $(2+ (\dd+\one)\cdot\kk)$, our identity reduces to:
$$\binom{2+ (\dd+\one)\cdot\kk}{\kk} = \binom{1+ (\dd+\one)\cdot\kk}{\kk} + \sum_{i=1}^m \binom{1+ (\dd+\one)\cdot\kk}{\kk^i} ,$$ which is the basic identity for multinomial coefficients $$\binom{n}{\kk} = \binom{n-1}{\kk} + \sum_{i=1}^m \binom{n-1}{\kk^i},$$ with $n = 2+ (\dd+\one)\cdot\kk$.
\end{proof}

%-------------------------------------------------------------------
\subsection{Formula for $G$}\label{sec:G} Our second main lemma is a formula for $G$:

\begin{lemma}
\label{lemma:G}
\begin{multline*}
G(\sss) = \frac{3}{4}z^2 -\frac{1}{2} z^3 + \frac{1}{2}\left( \sum_{i=1}^m\frac{d_i s_i z^{d_i+4}}{d_i+2}   -\sum_{i=1}^m \frac{(d_i-2)s_iz^{d_i+3}}{d_i+2}\  -\right.\\
\left.\sum_{i, j=1}^{m} \frac{d_i (d_i+4)s_i s_j  z^{4+d_i+d_j}}{(d_i+2)(4+d_i+d_j)}\right)\,.
\end{multline*}
\end{lemma}
\medskip
\noindent Before we prove this lemma, we prove Theorem~\ref{theorem:id} assuming it.

%-------------------------------------------------------------------
\subsection{Proof of Theorem~\ref{theorem:id}}
\label{ss:Proof:of:Theorem:A:2}
We want to show (\ref{eq:id}) assuming Lemma~\ref{lemma:G} (and Lemma~\ref{lemma:F}), that is, we want to show
\begin{equation}
\label{eq:G:Fsquared}
G(\sss) \stackrel{?}{=} F^2(\sss) -  \frac{1}{4}z^2 \left(1-z + \sum_{i=1}^m s_i z^{d_i +1}\right)^2\,.
\end{equation}
Expanding $F^2(\sss)$ using $$F(\sss) = \sum_{i=1}^{m} \frac{s_i}{d_i+2} z^{d_i+2} - \frac{1}{2}z^2 + z,$$ we obtain $$F^2(\sss) = \left( \sum_{i=1}^{m} \frac{s_i}{d_i+2} z^{d_i+2}\right)^2 + \left(z-\frac{1}{2}z^2\right)^2 + 2  \left( \sum_{i=1}^{m} \frac{s_i}{d_i+2} z^{d_i+2}\right) \left(z-\frac{1}{2}z^2\right)\,.$$
Expanding this
expression for $F^2(\sss)$, expanding the second term
in the right-hand side of~\eqref{eq:G:Fsquared} and
simplifying,
we obtain three types of terms in the resulting expression for
the right-hand side of~\eqref{eq:G:Fsquared}:
 \medskip
 \begin{description}
 \item[Simple powers] $\frac{3}{4}z^2 - \frac{1}{2} z^3$
 \medskip
 \item[Single sums] $\sum_{i=1}^m \left( \frac{2}{d_i+2} - \frac{1}{2} \right) s_i z^{d_i+3} + \sum_{i=1}^m \left( \frac{1}{2} - \frac{1}{d_i + 2} \right) s_i z^{d_i+4}$
 \medskip
 \item[Double sum] $\sum_{i,j =1}^{m} \left(\frac{1}{(d_i+2)(d_j+2)} - \frac{1}{4}\right) s_i s_j z^{4 + d_i + d_j}$
  \end{description}
 Expanding Lemma~\ref{lemma:G} in a similar fashion, we have the corresponding terms for $G(\sss)$:
  \begin{description}
 \item[Simple powers] $\frac{3}{4}z^2 - \frac{1}{2} z^3$
 \medskip
 \item[Single sums] $\sum_{i=1}^m \left( \frac{2-d_i}{2(d_i+2)}\right) s_i z^{d_i+3} + \sum_{i=1}^m \left( \frac{d_i}{2(d_i + 2)} \right) s_i z^{d_i+4}$
 \medskip
 \item[Double sum] $-\sum_{i,j =1}^{m} \left(\frac{d_i(d_i+4)}{2(d_i+2)} \frac{1}{4+d_i+d_j}\right) s_i s_j z^{4 + d_i + d_j}$
  \end{description}
A  quick  inspection shows that the simple powers and single sums are
equal.  For  the  double sum, we need to combine the $(i,j)$ and $(j,
i)$  terms  in  both  sums  (note that the terms are identical in the
$F^2$  expansion,  but  not  in  the  $G$ expansion), and check their
equality. The $F^2$ term is thus
$$
-2 \left(\frac{1}{4} - \frac{1}{(d_i+2)(d_j+2)}\right)
$$
and the $G$ term is
$$
-\frac{1}{4+d_i+d_j} \left( \frac{d_i(d_i+4)}{2(d_i+2)} +\frac{d_j(d_j+4)}{2(d_j+2)}\right)\,.
$$
To check their equality, we reorganize and obtain:
 $$\frac{-4+(d_i+2)(d_j+2)}{2(d_i+2)(d_j+2)} \stackrel{?}{=} \frac{1}{2(4+d_i+d_j)} \left( \frac{d_i(d_i+4)(d_j+2) + d_j(d_j+4)(d_i+2)}{(d_i+2)(d_j+2)}\right)$$
Cancelling and cross-multiplying, this reduces to
 $$(4+d_i+d_j)(d_i d_j   + 2d_i + 2d_j) = (d_i^2 + 4d_i)(d_j+2) + (d_j^2 + 4d_j)(d_i+2),$$
 which is easily verified. \qed

\subsection{Proof of Lemma~\ref{lemma:G}}
Recall that the $\sss^{\kk}$ term for $G$
is
\begin{equation}
\label{eq:Gterm}
\frac{ 6 + \sum_{i=1}^m \frac{d_i
      (d_i +1)}{d_i+2} k_i}{(2 + (\dd+\one)\cdot\kk)(3 +
    (\dd+\one)\cdot\kk)(4 + (\dd+\one)\cdot\kk)} \binom{4 +
    (\dd+\one)\cdot \kk}{\kk}
\end{equation} We
observe $$\frac{d_i(d_i+1)}{(d_i+2)} = d_i - 1 + \frac{2}{d_i+2},$$
and write $$6 = \frac{3}{2} \left( \left(4 + \dd\cdot\kk\right) - \dd
  \cdot \kk \right).$$ Using these, we rewrite the term
(\ref{eq:Gterm}) as the product of three terms:
\begin{description}

\item[Numerator] $\left( \frac{3}{2}(4 + \dd \cdot \kk) - \sum_{i=1}^m \left(\frac{1}{2} d_i + 1 - \frac{2}{d_i+2} \right) k_i \right)$

\medskip

\item[Partial Fractions] $\left( \frac{1}{2+ (\dd+\one)\cdot \kk} - \frac{1}{3+ (\dd+\one)\cdot \kk} \right)$

\medskip

\item[Multinomial Coefficient] $\frac{1}{4+ (\dd+\one)\cdot \kk}\binom{4+ (\dd+\one)\cdot \kk}{\kk} = \frac{ (3+(\dd+\one)\cdot\kk)!}{\kk! (4 + \dd\cdot\kk)!}$
\medskip

\end{description}
We consider terms from this triple product in turn.

\subsubsection{$k_i$-terms}\label{sec:ki} First, we consider the individual term $$\left(\frac{1}{2} d_i + 1 - \frac{2}{d_i+2} \right) k_i \left( \frac{1}{2+ (\dd+\one)\cdot \kk} - \frac{1}{3+ (\dd+\one)\cdot \kk} \right) \frac{ (3+(\dd+\one)\cdot\kk)!}{\kk! (4 + \dd\cdot\kk)!}.$$ Keeping the $\left(\frac{1}{2} d_i + 1 - \frac{2}{d_i+2} \right)$ term outside for now, and considering only the first part of the difference, we are interested in $$k_i  \frac{1}{2+ (\dd+\one)\cdot \kk}  \frac{ (3+(\dd+\one)\cdot\kk)!}{\kk! (4 + \dd\cdot\kk)!} = (3 + (\dd+\one)\cdot\kk) \frac{(1 + (\dd+ \one)\kk)!}{\kk^i! (4 + \dd\cdot\kk)!}$$
Expanding $$(3 + (\dd+\one)\cdot\kk) = (4 + \dd \cdot \kk )+ \one \cdot \kk^i,$$ we first consider
\begin{eqnarray*} (4+ \dd\cdot\kk) \frac{(1 + (\dd+ \one)\kk)!}{\kk^i! (4 + \dd\cdot\kk)!} &=& \frac{(1 + (\dd+ \one)\kk)!}{\kk^i! (3 + \dd\cdot\kk)!} \\ &=& \frac{(d_i+2+ (\dd+ \one)\kk^i)!}{\kk^i! (d_i + 3 + \dd\cdot\kk^i)!} \\ &=& \frac{1}{d_i+3}\frac{d_i+3}{d_i+3 + (\dd+\one)\cdot\kk^i} \binom{d_i+3 + (\dd+\one)\cdot\kk^i}{\kk^i} \\ &=& \frac{1}{d_i+3} A(d_i+3; \kk^i)\end{eqnarray*}
Now expanding $\one \cdot \kk^i = (k_i-1) + \sum_{j \neq i} k_j$, we have the terms \begin{eqnarray*} (k_i - 1) \frac{(1 + (\dd+ \one)\kk)!}{\kk^i! (4 + \dd\cdot\kk)!} &=& \frac{1}{4+2d_i}\frac{4+2d_i}{4 + 2d_i + (\dd+\one)\cdot \kk^{i,i}}\binom{4 + 2d_i + (\dd+\one)\cdot \kk^{i,i}}{\kk^{i,i}}  \\ &=& \frac{1}{4+2d_i} A(4+2d_i;\kk^{i,i}) \end{eqnarray*} and
\begin{eqnarray*}
(k_j) \frac{(1 + (\dd+ \one)\kk)!}{\kk^i! (4 + \dd\cdot\kk)!} &=& \frac{1}{4+d_i +d_j}\frac{4+d_i + d_j}{4 + d_i +d_j+ (\dd+\one)\cdot \kk^{i,j}}\binom{4 + d_i +d_j + (\dd+\one)\cdot \kk^{i,j}}{\kk^{i,j}}  \\ &=& \frac{1}{4+d_i+d_j} A(4+d_i+d_j;\kk^{i,j})
\end{eqnarray*}
Collecting all of these, we have
\begin{multline*}
k_i  \frac{1}{2+ (\dd+\one)\cdot \kk}  \frac{ (3+(\dd+\one)\cdot\kk)!}{\kk! (4 + \dd\cdot\kk)!}=
\\
=\frac{1}{d_i+3} A(d_i+3; \kk^i)  + \sum_{j=1}^{m} \frac{1}{4+d_i+d_j} A(4+d_i+d_j; \kk^{i,j})
\end{multline*}
Next, we work with the factor
\begin{eqnarray*} -k_i \frac{1}{3+(\dd+\one)\cdot\kk} \frac{ (3+(\dd+\one)\cdot\kk)!}{\kk! (4 + \dd\cdot\kk)!} &=& -\frac{ (2+(\dd+\one)\cdot\kk)!}{\kk^i! (4 + \dd\cdot\kk)!}\\ &=& \frac{(3+d_i+ (\dd+\one)\cdot \kk^i)!}{\kk^i! (4 +d_1+ \dd\cdot\kk^i)!}\\ &=& -\frac{1}{d_i+4} A(4+d_i; \kk^i)\\ \end{eqnarray*}
These $k_i$ terms come with the factor of $$\left(\frac{1}{2} d_i + 1 - \frac{2}{d_i+2} \right) = \frac{d_i(d_i+4)}{2(d_i+2)},$$ so we have that their total contribution is: $$ \sum_{i=1}^m \frac{d_i(d_i+4)}{2(d_i+2)} \left( \frac{A(d_i+3; \kk^i)}{d_i+3} - \frac{A(d_i+4; \kk^i)}{d_i+4} + \sum_{j=1}^m \frac{A(4+d_i+d_j; \kk^{i,j})}{4+d_i+d_j} \right).$$ Summing over $\kk \in (\Z^{\geq 0})^m$, we obtain, using Mohanty's formula, \begin{equation}\label{eq:ki} \sum_{i=1}^m \frac{d_i(d_i+4)}{2(d_i+2)} \left( \frac{s_i z^{d_i+3}}{d_i+3} - \frac{s_i z^{d_i+4}}{d_i+4} + \sum_{j=1}^m \frac{s_i s_j z^{4+d_i+d_j}}{4+d_i+d_j} \right)\end{equation}

\subsubsection{$\frac{3}{2}\left(4+ \dd \cdot \kk\right)$-terms} We now expand the $\frac{3}{2}\left(4+ \dd \cdot \kk\right)$-terms, keeping $\frac 3 2$ on the outside for now. That is, we consider $$(4 + \dd \cdot \kk)\left( \frac{1}{2+ (\dd+\one)\cdot \kk} - \frac{1}{3+ (\dd+\one)\cdot \kk} \right)\frac{ (3+(\dd+\one)\cdot\kk)!}{\kk! (4 + \dd\cdot\kk)!}.$$
As above, we first work with the term \begin{eqnarray*} (4 + \dd \cdot \kk)\frac{1}{2+ (\dd+\one)\cdot \kk}\frac{ (3+(\dd+\one)\cdot\kk)!}{\kk! (4 + \dd\cdot\kk)!} &=&(3+(\dd+\one)\cdot\kk) \frac{ (1+(\dd+\one)\cdot\kk)!}{\kk! (3 + \dd\cdot\kk)!}  \\ &=&( (3+\dd \cdot \kk) + \one \cdot \kk) \frac{ (1+(\dd+\one)\cdot\kk)!}{\kk! (3 + \dd\cdot\kk)!} \end{eqnarray*}
The $\one \cdot \kk$ term can be split up into individual terms, and as above, we have  $$k_i \frac{(1 + (\dd+ \one)\kk)!}{\kk! (3 + \dd\cdot\kk)!} = \frac{(1 + (\dd+ \one)\kk)!}{\kk^i! (3 + \dd\cdot\kk)!} = \frac{1}{d_i+3} A( d_i+3; \kk^i)$$
The $(3 + \dd \cdot \kk)$ term yields
\begin{eqnarray*}(3+\dd \cdot \kk) \frac{ (1+(\dd+\one)\cdot\kk)!}{\kk! (3 + \dd\cdot\kk)!} &=& \frac{ (1+(\dd+\one)\cdot\kk)!}{\kk! (2 + \dd\cdot\kk)!}\\ &=& \frac{1}{2}\frac{2}{2+(\dd +\one)\cdot\kk}\binom{2+(\dd+\one)\cdot\kk}{\kk}\\ &=& \frac{1}{2}A(2;\kk)\\ \end{eqnarray*}
Thus we have $$(4 + \dd \cdot \kk)\frac{1}{2+ (\dd+\one)\cdot \kk}\frac{ (3+(\dd+\one)\cdot\kk)!}{\kk! (4 + \dd\cdot\kk)!} = \frac{A(2;\kk)}{2} + \sum_{i=1}^m \frac{A(d_i+3; \kk^i)}{d_i+3}$$
We are left with the term \begin{eqnarray*} -(4 + \dd \cdot \kk)\frac{1}{3+ (\dd+\one)\cdot \kk}\frac{ (3+(\dd+\one)\cdot\kk)!}{\kk! (4 + \dd\cdot\kk)!} &=& -\frac{1}{3}\frac{3}{3+(\dd+\one)\cdot\kk} \binom{3+(\dd+\one)\cdot\kk}{\kk} \\ &=& -\frac{1}{3}A(3;\kk) \\\end{eqnarray*}
Combining the above, and recalling the coefficient of $\frac 3 2$, and summing over $\kk \in (\Z^{\geq 0})^m$ we have the total contribution of the $\frac{3}{2}\left(4+ \dd \cdot \kk\right)$-terms:
\begin{equation}\label{eq:32} \frac 3 2 \left( \frac 1 2 z^2 - \frac 1 3 z^3 + \sum_{i=1}^m \frac{s_iz^{d_i+3}}{d_i+3} \right) \end{equation}

\subsubsection{Combining  terms}  To  conclude,  we combine equations
(\ref{eq:ki}) and (\ref{eq:32}) to obtain
\begin{multline*}
G(\sss) = \frac 3 2
\left( \frac 1 2 z^2 - \frac 1 3 z^3 + \sum_{i=1}^m
\frac{s_iz^{d_i+3} }{d_i+3}\right) -\\
- \sum_{i=1}^m
\frac{d_i(d_i+4)}{2(d_i+2)} \left( \frac{s_i z^{d_i+3}}{d_i+3} -
\frac{s_i z^{d_i+4}}{d_i+4} + \sum_{j=1}^m \frac{s_i s_j
z^{4+d_i+d_j}}{4+d_i+d_j} \right).
\end{multline*}
Collecting terms, we have
\begin{multline*}
G(\sss) = \frac 3 4 z^2 - \frac 1 2 z^3 + \sum_{i=1}^m \frac{s_i
z_i^{d_i+3}}{d_i+3} \left (\frac 3 2 -
\frac{d_i(d_i+4)}{2(d_i+2)}\right) +\\
+ \frac{1}{2}\left( \sum_{i=1}^m
\frac{d_i s_i z^{d_i+4} }{d_i+2}- \sum_{i,j=1}^m
\frac{d_i(d_i+4)}{d_i+2} \frac{s_i s_j
z^{4+d_i+d_j}}{4+d_i+d_j}\right)
\end{multline*}
Finally, using
$$
\frac 3 2 - \frac{d_i(d_i+4)}{2(d_i+2)}
= -\frac{1}{2} \frac{ (d_i+3)(d_i-2)}{d_i+2}\,,
$$
we get
$$
\frac{s_i z_i^{d_i+3}}{d_i+3} \left (\frac 3 2 - \frac{d_i(d_i+4)}{2(d_i+2)}\right)
= -\frac{1}{2}\frac{d_i-2}{d_i+2} z^{d_i+3.}
$$
Substituting this into our expression for $G$, we obtain as desired
\begin{multline*}
G(\sss) =\frac{3}{4}z^2 -\frac{1}{2} z^3 + \frac{1}{2}\left( \sum_{i=1}^m\frac{d_i s_i z^{d_i+4}}{d_i+2}   -\sum_{i=1}^m \frac{(d_i-2)s_iz^{d_i+3}}{d_i+2}\right.\ -\\
\left.\sum_{i, j=1}^{m} \frac{d_i (d_i+4)s_i s_j  z^{4+d_i+d_j}}{(d_i+2)(4+d_i+d_j)}\right)
\end{multline*}
\qed

%####################################################################
%###############  from billiards_summer 2010.tex  ###################
%####################################################################

\section{Counting pillowcase covers}
%\label{sec:Jenkins}
\label{sec:pillowcase:covers}

In this Appendix we describe the original approach to calculating the
volume  of  the  moduli  space of Abelian and quadratic differentials
suggested  by  \mbox{H.~Masur},  M.~Kontsevich,  and the authors, and
developed   with   success   by   A.~Eskin   and  \mbox{A.~Okounkov},
see~\cite{Eskin:Okounkov,   EO2}.   This   approach   was  also  used
in~\cite{Zorich:volumes} and~\cite{EMS}. The key idea is to translate
the volume calculation into a counting problem for ``integer points''
,  which  geometrically  correspond to \textit{square-tiled surfaces}
for   the   moduli   spaces   of   Abelian   differentials   and   to
\textit{pillowcase   covers}  for  the  moduli  spaces  of  quadratic
differentials.

In  \S\ref{subsec:integer:points}  we  show why volume calculation is
equivalent     to     counting     the     lattice     points.     In
\S\ref{subsec:lattice:square:tiled:pillowcase}    we    recall    the
definition  of  the  \textit{pillowcase cover}, show that counting of
lattice  points  is equivalent to the counting problem for pillowcase
covers and prove Theorem~\ref{th:pillowcases:as:volume}.

%--------------------------------------------------------------------
\subsection{Reduction    of    volume    calculation    to   counting
lattice points}
\label{subsec:integer:points}

The   volume   of   a   stratum   $\cQ_1(d_1,\dots,d_k)$  is  defined
by~\eqref{eq:normalization} as
$$
\Vol\cQ_1(d_1,\dots,d_k) =
\dim_{\R{}} \cQ(d_1,\dots,d_k)\cdot\mu(C(\cQ_1(d_1,\dots,d_k))\,,
$$
where  $\mu(C(\cQ_1(d_1,\dots,d_k))$  is  the  total  volume  of  the
``cone''  $C(\cQ_1(d_1,\dots,d_k))\subset\cQ(d_1,\dots,d_k)$ measured
by means of the volume element $d\mu$ on $\cQ(d_1,\dots,d_k)$ defined
in  ~\S\ref{sec:subsec:notmalization}.  The  total volume of the
cone  $C(\cQ_1(d_1,\dots,d_k))$  is  the  limit  of the appropriately
normalized Riemann sums.

The  volume  element  $d\mu$ is defined as a linear volume element in
cohomological  coordinates,  normalized  by certain specific lattice.
Chose  a  positive  $\epsilon$ such that $1/\epsilon$ is integer, and
consider   a   sublattice   of   the   initial   lattice   of   index
$(1/\epsilon)^{\dim_{\R{}}  \cQ(d_1,\dots,d_k)}$  partitioning  every
side   of   the   initial   lattice  into  $1/\epsilon$  pieces.  The
corresponding  Riemann  sums  count  the  number  of  points  of  the
sublattices  which  get  inside  the cone. Thus, by definition of the
measure $\mu$ we get
\begin{multline*}
\mu(C(\cQ_1(d_1,\dots,d_k))=
\\
\lim_{\epsilon\to 0}
\epsilon^{\dim_{\R{}}\! \cQ(d_1,\dots,d_k)}
%\\
\big(\text{Number of points of the $\epsilon$-sublattice inside } %the cone }
C(\cQ_1(d_1,\dots,d_k))\big).
\end{multline*}

We  assume that $1/\epsilon$ is integer. Note that a flat surface $S$
represents  a  point  of  the  $\epsilon$-lattice, if and only if the
surface     $(1/\epsilon)\cdot     S$     (in     the     sense    of
definition~\eqref{eq:Cstar:action}) represents a point of the integer
lattice.  Denoting  by  $C(\cQ_\Deg(d_1,\dots,d_k))$  the set of flat
surfaces   in  the  stratum  $\cQ(d_1,\dots,d_k)$  of  area  at  most
$\Deg/2$, and taking into consideration that
$$
\area((1/\epsilon)\cdot S)=1/\epsilon^2\cdot\area(S)
$$
we can rewrite the above relation as
\begin{multline}
\label{eq:lattice:points}
\mu(C(\cQ_1(d_1,\dots,d_k))=
\lim_{\Deg\to+\infty}
\Deg^{-\dim_{\C{}} \cQ(d_1,\dots,d_k)}\cdot
\\
\big(\text{Number of lattice points inside the cone }
C(\cQ_\Deg(d_1,\dots,d_k)\big)\,.
\end{multline}

%--------------------------------------------------------------------
\subsection{Lattice  points,  square-tiled  surfaces,  and pillowcase
covers}
\label{subsec:lattice:square:tiled:pillowcase}

Let  $\Lambda \subset \cx$ be a lattice, and let $\T^2 = \cx/\Lambda$
be the associated torus. The quotient
$$
\cP : = \T^2/\pm
$$
by  the  map  $z  \rightarrow  -z$ is known as the \textit{pillowcase
orbifold}.  It  is  a  sphere with four $(\Z/2)$-orbifold points (the
corners  of  the  pillowcase). The quadratic differential $(dz)^2$ on
$\T^2$  descends  to  a  quadratic differential on $\cP$. Viewed as a
quadratic  differential  on  the  Riemann sphere, $(dz)^2$ has simple
poles  at  corner  points. When the lattice $\Lambda$ is the standard
integer  lattice $\Z\oplus i\Z$, the flat torus $\T^2$ is obtained by
isometrically  identifying  the  opposite sides of a unit square, and
the  pillowcase  $\cP$  is  obtained by isometrically identifying two
squares    with    the    side    $1/2$    by   the   boundary,   see
Figure~\ref{fig:pillowcase:cover}.

\begin{figure}[htb]
\includegraphics{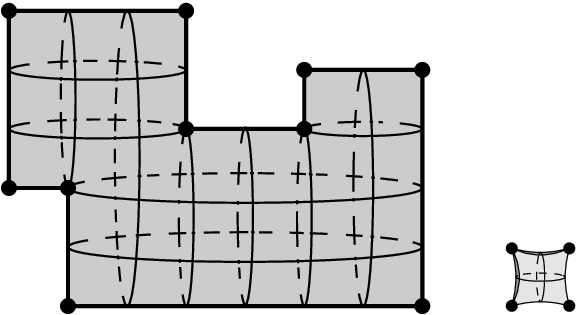}
%    %
% \special{
% psfile=pillow_simple.eps
% hscale=10
% vscale=10
% voffset=-70
% hoffset=250
% angle=180
% }
  %
\vspace{92pt}
\caption{
\label{fig:pillowcase:cover}
Pillowcase  cover  (on the left) over the pillowcase orbifold (on the
right).  A general pillowcase cover
it is not necessarily glued from two identical polygons. }
\end{figure}

Consider  a  connected  ramified cover $\hat\cP$ over $\cP$ of degree
$\Deg$  having  ramification  points  only  over  the  corners of the
pillowcase.  Clearly,  $\hat\cP$  is  tiled by $2\Deg$ squares of the
size $(1/2)\times(1/2)$ in such way that the squares do not superpose
and  the vertices are glued to the vertices. Coloring the two squares
of the pillowcase $\cP$ one in black and the other in white, we get a
chessboard  coloring of the square tiling of the cover $\hat\cP$:
the white squares are always glued to the black ones and vice versa.

\begin{lemma}
\label{lm:pillowcase}
Let $S$ be a flat surface in the stratum $\cQ(d_1,\dots,d_k)$.
The following properties are equivalent:
\begin{enumerate}
\item
The surface $S$ represents a lattice point in $\cQ(d_1,\dots,d_k)$;
\item
$S$ is a cover over $\cP$ ramified only over the corners of the pillow;
\item
$S$ is tiled by black and white $(1/2)\times(1/2)$ squares respecting
the chessboard coloring.
\end{enumerate}
\end{lemma}
\begin{proof}
We  have  just proved that (2) implies (3). To prove that (1) implies
(2)  we  define  the following map from $S$ to $\cP$. Fix a zero or a
pole  $P_0$  on  $S$.  For  any  $P\in S$ consider a path $\gamma(P)$
joining  $P_0$  to  $P$  having  no  self-intersections and having no
zeroes or poles inside. The restriction of the quadratic differential
$q$   to   such   $\gamma(P)$   admits  a  well-defined  square  root
$\omega=\pm\sqrt{q}$,  which is a holomorphic form on the interior of
$\gamma$. Define
\begin{equation}
\label{eq:definition:of:cover}
P\mapsto \left(\int_{\gamma(P)} \omega\mod \Z\oplus i\Z\right)/\pm\,.
\end{equation}
Of  course,  the  path  $\gamma(P)$ is not uniquely defined. However,
since  the  flat  surface  $S$  represents  a  lattice point (see the
definition in ~\S\ref{sec:subsec:notmalization}), the difference
of  the  integrals  of $\omega$ over any two such paths $\gamma_1(P)$
and  $\gamma_2(P)$  belongs to $\Z\oplus i\Z$, so taking the quotient
over the integer lattice and over $\pm$ we get a well-defined map. By
definition  of  the pillowcase $\cP$ we have,
$\cP=\left(\C{}/(\Z\oplus  i\Z)\right)/\pm$.
Thus, we have defined a map $S\to\cP$. It
follows  from the definition of the map, that it is a ramified cover,
and  that  all  regular  points  of  the flat surface $S$ are regular
points  of  the cover. Thus, all ramification points are located over
the corners of the pillowcase.

A similar consideration shows that (3) implies (1).
\end{proof}

Let $\operatorname{Sq}_\Deg(d_1,\dots,d_k)$ be the number of surfaces
in  the  stratum $\cQ(d_1,\dots,d_k)$ tiled with at most $\Deg$ black
and   $\Deg$   white  squares  respecting  the  chessboard  coloring.
Lemma~\ref{lm:pillowcase}       allows       us       to      rewrite
formula~\eqref{eq:lattice:points} as follows:
$$
\mu(C(\cQ_1(d_1,\dots,d_k))=
\lim_{\Deg\to+\infty}
\Deg^{-\dim_{\C{}} \cQ(d_1,\dots,d_k)}\cdot
\operatorname{Sq}_\Deg(d_1,\dots,d_k)
\,.
$$
Taking into consideration~\eqref{eq:normalization} we get
\begin{multline}
\label{eq:Vol:through:square:tiled}
\Vol\cQ_1(d_1,\dots,d_k)=2\dim_\cx\cQ(d_1,\dots,d_k)\,\cdot
\\
\lim_{\Deg\to+\infty}
\Deg^{-\dim_{\C{}} \cQ(d_1,\dots,d_k)}\cdot
\operatorname{Sq}_\Deg(d_1,\dots,d_k)
\,.
\end{multline}

We  now  state  and  prove  two  Lemmas  which we use in the proof of
Theorem~\ref{th:pillowcases:as:volume}.

\begin{lemma}
\label{lm:Eskin:Okounkov:versa:square:tiled}
For  any  $\eta,  \nu$  as above the following asymptotic relation is
valid:
\begin{equation}
\label{eq:Eskin:Okounkov:versa:square:tiled}
\lim_{N\to+\infty}
\frac
{\sum_{d=1}^N\operatorname{Cov}^0_{4d}(\eta,\nu)}
{\sum_{d=1}^N\operatorname{Cov}^{0,{\scriptscriptstyle \boxplus}}_{4d}(\eta,\nu)}
=2^{\ell(\eta)}\,,
\end{equation}
where $\ell(\eta)$ is the number of entries in $\eta$.
\end{lemma}

\begin{proof}
Let  $P$  be  a  zero of even degree of a quadratic differential, and
$U_\epsilon(P)$  a  multidisc  of  flat radius $\epsilon$ centered at
$P$.  Choosing  $\epsilon$  sufficiently  small,  we  can assume that
$U_\epsilon(P)$  is  embedded  into  the  ambient  flat surface. (For
example,  for  the  flat  surface $\hat\cP$ induced from the standard
$\frac{1}{2}\times\frac{1}{2}$  square  pillowcase  $\cP$ by means of
the cover $\pi_{\scriptscriptstyle \boxplus}:\hat\cP \to \cP$ one can
choose any $\epsilon$ satisfying $0<\epsilon<\frac{1}{2}$.) Choose an
orientation  of  the  vertical  direction in $U_\epsilon(P)$. For any
vector  $\vec  v\in\R^2$  such that $\|\vec v\|< \epsilon$ there is a
unique way to move the zero $P$ in direction $\vec v$ by the distance
$\|\vec  v\|$  via a local move inside $U(P)$ keeping the flat metric
outside  of $U(P)$ unchanged. The corresponding local surgery (called,
\begin{figure}[htb]
\includegraphics{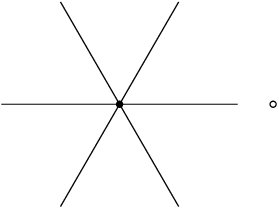}
\includegraphics{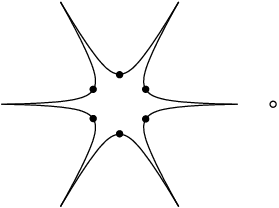}
\includegraphics{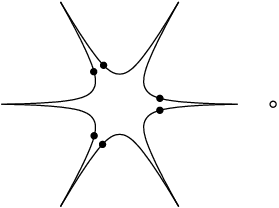}
\includegraphics{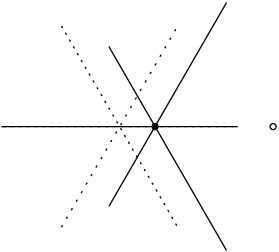}
\vspace{55pt}
\caption{
\label{fig:moving:zero}
Cartoon of a local move of a zero of even degree.
}
\end{figure}
depending  on the author and the context \textit{Schiffer variation},
or  \textit{deformation along the kernel (or Rel) foliation}, etc) is
represented  in  Figure~\ref{fig:moving:zero},  where  the separatrix
rays  (``prongs'')  adjacent  to $P$ are chosen to be parallel to the
vector  $\vec  v$.  This  local  deformation  can be  performed as
follows.  Make short slits along all prongs in direction $\vec v$ and
open  them.  The original conical point of a     cone angle $\pi\cdot
2k$  gives rise to $2k$ marked points on the sides of the slits. Move
the marked points along the sides of the slit by the distance $\|\vec
v\|$  in  direction  $\vec  v$  and  zip the slits back up to the new
marked  points.  All  new  marked points get identified into a single
conical  singularity  with  the  cone  angle $\pi\cdot 2k$. Note that
usually,  the flat surfaces obtained after deformations along vectors
$\vec   v$   and   $-\vec   v$   are   generically   non-isomorphic.

Consider           now           a          pillowcase          cover
$\pi_{\scriptscriptstyle\boxplus}:\hat\cP  \to  \cP$ as above. As the
base  sphere  choose  the  standard  pillowcase $\cP$ endowed with the
quadratic  differential $q_0=(dz)^2$. It has four simple poles at the
corners  of  the  pillow  and  no  other  singularities. Pulling back
$(dz)^2$  via  $\pi_{\scriptscriptstyle\boxplus}$  gives  a quadratic
differential  on  the  covering  $\CP$ with zeros and simple poles of
degrees   $\{\nu_i-2\}$   and   $\{2\eta_j-2\}$  and  with  no  other
singularities.   Thus,   by   construction   the   pillowcase   cover
$\hat\cP:=(\CP,\pi^\ast q_0)$ belongs to $\cQ(\eta,\nu)$.

Move  the  zero  $P_1$ of degree $2\eta_1-2$ in direction $\vec v$ as
above.  The  deformed flat surface inherits a structure of a ramified
cover over the pillowcase orbifold $\cP$. The corresponding cover can
be       defined      by      a      formal      construction      as
equation~\eqref{eq:definition:of:cover},  or  can  be  seen  in plain
terms  as  follows. We have deformed our flat structure only inside a
neighborhood  $U(P_1)$, so we let the projection of the complement of
$U(P_1)$ to the pillowcase orbifold $\cP$ unchanged. The neighborhood
$U(P_1)$  is  glued  from  even  number  $2k$  of  half-disks  as  in
Figure~\ref{fig:moving:zero};   we  define  the  projection  of  each
half-disk  to  the pillowcase orbifold $\cP$ unchanged.The definition matches on the common boundaries of the
half-discs.  By construction, the deformed cover
$\pi':\hat\cP'    \to    \cP$,    has    the   ramification   profile
$(2\eta_2,\dots,2\eta_{\ell(\eta)},\nu,2^{2d-|\eta|+\eta_1-|\nu|/2})$
over  $0\in\cP$, profile $(\eta_1,1^{4d-\eta_1})$ over the projection
of  the  deformed  zero $P_1'$, and profile $(2^{2d})$ over the other
three corners of $\cP$. The cover $\pi'$ is unramified elsewhere.

Consider  now  the  same  pillowcase  cover  $\pi_{\scriptscriptstyle
\boxplus}:\hat\cP \to \cP$ as above and move the zero $P_1$ of degree
$2\eta_1-2$  of  the  initial quadratic differential $q$ in direction
$-\vec  v$.  Clearly  we  get a pillowcase cover $\pi'':\hat\cP'' \to
\cP$  with  exactly  the  same  profile  as  $\pi':\hat\cP' \to \cP$.
Moreover, since the zero $P_1$ of $q$ on the original cover $\hat\cP$
was  projected  to  a  conical singularity of the pillowcase orbifold
$\cP$ with the cone angle $\pi$, moving from the corresponding corner
of the pillow in directions $\vec v$ and $-\vec v$ we get to the same
point  of  the  pillowcase  orbifold  $\cP$. In other words, the zero
$P_1'$  of  the  deformed  quadratic  differential $q'$ on $\hat\cP'$
obtained  by  moving  the  zero $P_1$ of $q$ in direction $\vec v$ is
projected  to  the same point of the pillowcase orbifold $\cP$ as the
zero  $P_1''$  of  $q''$  on  $\hat\cP''$ obtained by moving the zero
$P_1$  of  $q$ in direction $-\vec v$. The number of covers $\hat\cP$
for  which  the  resulting  covers  $\hat\cP'$  and  $\hat\cP''$  are
isomorphic   has   asymptotics   of   lower   order   in   $N$   than
$\sum_{d=1}^N\operatorname{Cov}^0_{4d}(\eta,\nu)$.

Moving all even-order zeroes $P_1,\dots,P_{\ell(\eta)}$ in directions
of     pairwise-distinct     vectors    $\pm\vec    v_1,\dots,\pm\vec
v_{\ell(\eta)}$ we establish a $2^{\ell(\eta)}$-to-one correspondence
(up to a term of lower order asymptotics in $N$) between covers $\pi$
and $\pi_{\scriptscriptstyle \boxplus}$.
\end{proof}

Quadratic   differentials   induced   from  $dz^2$  on  the  standard
pillowcase         orbifold         via         the        pillowcase
covers~\eqref{eq:Eskin:Okounkov:pillowcase:cover}    constructed   in
\S\ref{sec:subsec:counting:pillowcase:covers}   have   the  following
structure.  All  zeroes  of  odd  degrees  of  such differentials are
projected  to the same corner of the pillow, while all zeroes of even
degrees  are  projected  to  pairwise  distinct  non-corner  points.
Quadratic   differentials   induced   from  $dz^2$  on  the  standard
pillowcase      orbifold      via      the      pillowcase     covers
$\pi_{\scriptscriptstyle\boxplus}:\hat\cP  \to  \cP$  as  above  have
slightly  different  structure. Namely, \textit{all} their zeroes (no
matter  of  odd  or  even  degree)  project to the same corner of the
pillowcase   orbifold,  so  they  are  really  \textit{square-tiled}.
Moreover,   since   all   preimages   under  $\pi_{\scriptscriptstyle
\boxplus}$  of the three remaining corners of the pillowcase orbifold
are  regular  points  of  the flat metric, the resulting flat surface
$\hat\cP$  can  be  tiled  with  $1\times  1$  squares  (compared  to
$\frac{1}{2}\times\frac{1}{2}$  squares  of the pillowcase orbifold).
Our  next  Lemma  proves, that (in genus zero) this new square tiling
with larger squares admits a chessboard coloring.

\begin{lemma}
\label{lm:chessboard:coloring}
Any pillowcase cover $\pi_{\scriptscriptstyle\boxplus}:\hat\cP \to
\cP$ of genus zero with a ramification profile as above decomposes
into  two consecutive covers
$$
\pi: \hat\cP \rightarrow \cP_4 \rightarrow \cP\,,
$$
where $\cP_4 \rightarrow \cP$ is a cover of order $4$ of a pillowcase
orbifold  of  size  $1\times  1$ over the standard pillowcase of size
$\frac{1}{2}\times\frac{1}{2}$.
\end{lemma}

\begin{remark}
\label{rm:genus:0:only}
Without  the  condition that the genus of $\hat\cP$ as above is equal
to  zero  the  assertion  of Lemma~\ref{lm:chessboard:coloring} is no
longer true in general.
\end{remark}

\begin{proof}[Proof of Lemma~\ref{lm:chessboard:coloring}]
Consider  the  decomposition  of  $\hat\cP$ into  maximal  horizontal
cylinders.  We  associate  to  this decomposition a finite graph. The
edges  of  the  graph  are  in  one-to-one  correspondence  with  the
cylinders. The vertices of the graph are in one-to-one correspondence
with  connected  components  of singular horizontal layers. Two edges
have  common  vertex  if  the  corresponding  maximal  cylinders  are
adjacent  to  the same connected component of the critical horizontal
layer.

Note,  that  for  square-tiled  surfaces  of genus zero the resulting
graph  is, actually, a tree which we denote by $\operatorname{T}$. To
prove  the  Lemma  we  first  prove that perimeters of all horizontal
cylinders  are even integer numbers. The proof is an induction in the
number of horizontal cylinders.

The  base  of  induction corresponds to the case when $\hat\cP$ has a
single  horizontal  cylinder.  Then  $\hat\cP$ has only two singular
layers,  one  on  each  side  of  the  cylinder.  Consider one of the
horizontal   singular   layers   as  a  graph  of  horizontal  saddle
connections. Since $\hat\cP$ is a sphere, the corresponding graph is,
actually,  a tree. By construction each saddle connection has integer
length.   The  waist  curve  of  the  cylinder  follows  each  saddle
connection twice, so its perimeter is twice the sum of lengths of all
saddle  connections  in the layer, and hence the perimeter is an even
number.

When  the  number  of  horizontal  cylinders  is  greater  than  one,
analogous  consideration  shows  that  perimeters  of  all  cylinders
represented   by   the   extremity   edges   (leaves)   of  the  tree
$\operatorname{T}$   have  even  perimeters.  Chopping  one  of  these
cylinders   out   from   the   initial  flat  surface  $\hat\cP$  and
isometrically  identifying  the  parts of the boundary in the natural
way  we  get  a  new pillowcase cover $\tilde\cP$ satisfying the same
properties  as before. By induction all its horizontal cylinders have
even perimeters.

Having  proved  that  the  perimeters of all horizontal cylinders are
even  integers  we apply the induction in the number of cylinders one
more  time  proving  now  that  the tiling of $\hat\cP$ with $1\times
1$ squares admits chessboard coloring.
\end{proof}

Now         everything         is        ready        to        prove
Theorem~\ref{th:pillowcases:as:volume}.

\begin{proof}[Proof of Theorem~\ref{th:pillowcases:as:volume}]
By     Lemma~\ref{lm:chessboard:coloring}    a    pillowcase    cover
$\pi_{\scriptscriptstyle\boxplus}:\hat\cP  \to \cP$ of genus zero and
of  degree $4d$ with ramification data $(\eta,\nu)$ as above uniquely
defines   a   square-tiled   pillowcase   cover   of  degree  $d$  in
$\cQ(\eta,\nu)$.

Reciprocally,  consider  an  arbitrary square-tiled surface $S$ as in
Lemma~\ref{lm:pillowcase}  above  in the stratum $\cQ(\eta,\nu)$, and
let  $d$  be  the  degree  of  the  corresponding  cover  over $\cP$.
Subdividing   each  square  into  four;  considering  the  underlying
pillowcase  as  $\cP_4$  and  postcomposing  the  initial cover $S\to
\cP_4$  with  the  cover $\cP_4\to\cP$ we get a pillowcase cover with
singularity pattern $(\eta,\nu)$. This implies that
   %\marginpar{Is it really a one-to-one correspondence or I am overlooking something? I think you are fine here (JA)}
   %
$$
\sum_{d=1}^{\Deg}
\operatorname{Cov}^{0,{\scriptscriptstyle \boxplus}}_{4d}(\eta,\nu)
=
\operatorname{Sq}_\Deg(\nu_1-2,\nu_2-2,\dots,2\eta_1-2,2\eta_2-2,\dots)\,.
$$
Applying  equation~\eqref{eq:Eskin:Okounkov:versa:square:tiled}  from
Lemma~\ref{lm:Eskin:Okounkov:versa:square:tiled},     taking     into
consideration that
$$
\dim_{\C{}}\cQ(\nu_1-2,\nu_2-2,\dots,2\eta_1-2,2\eta_2-2,\dots)=
\ell(\nu)+\ell(\eta)-2
$$
and applying equation~\eqref{eq:Vol:through:square:tiled} we complete
the proof of Theorem~\ref{th:pillowcases:as:volume}.
\end{proof}

%####################################################################
%####################################################################
%####################################################################
\section{Equidistribution of Circle Translates\\   \smallskip by Jon Chaika}
\label{sec:chaika}
We use a variation of an argument of G.~A.~Margulis to obtain equidistribution of circles from exponential mixing of the Teichm\"uller geodesic flow on $\cQ_1$. The strategy is similar in spirit to \cite[Section 3.6]{Eskin:Margulis:Mozes:31}.

\subsection{Notation} As in \S\ref{sec:subsec:introduction:beginning},
let $\cB = \cB(k_1,
\dots, k_n)$ be the space of directional billiard tables, that is,
billiard tables with
interior angles $(\tfrac{\pi}{2} k_1, \dots, \tfrac{\pi}{2}
k_n)$ and a distinguished direction. $\cB$ has the natural
the measure $\mu_\cB$ which is the product measure of the Lebesgue
measure arising from the side lengths and the angular measure
$d\phi$.

Given $\Pi \in \cB$, let $q_{\Pi}$ be the meromorphic quadratic
differential given by gluing together two copies of $\Pi$. Recall that $\Pi \mapsto q_{\Pi}$ from $\cB$ to $\cQ =
\cQ(k_1-2, \dots, k_n-2)$ is a local embedding where $\cQ(d_1, \dots, d_n)$ is the stratum
of quadratic differentials with zeros of order $d_1, \dots,
d_n$. Using this map, we may view $\cB$ as a subset of $\cQ$. Let
$\cQ_1 \subset \cQ$ denote the subset of surfaces of flat area $1/2$ (see Convention~\ref{con:area:1:2}),
and let $\mu_1= \mu_{\cQ_1}$ denote the Masur-Veech measure on $\cQ_1$. Let
$\cB_1$ denote the intersection of $\cB \subset \cQ$ with $\cQ_1$.
We abuse notation by denoting the restriction of the measure $\mu_\cB$
to $\cB_1$ again by $\mu_{\cB}$. Let $\tilde{\cB}_1$ denote the subset of $\cB_1$ where the direction of the flow is parallel to one of the sides, and let $\mu_{\tilde{\cB}}$ be the restriction of the measure to $\tilde{\cB}_1$. We recall notation for important one-parameter subgroups of $SL(2, \R)$:
\begin{displaymath}
g_t=\left(\begin{array}{cc}e^t&0\\0&e^{-t}\end{array}\right) \mbox{ and } r_{\theta}=
  \left(\begin{array}{cc}\cos \theta &\sin \theta \\
                        -\sin\theta & \cos \theta \end{array}\right)
\end{displaymath}

\subsection{Equidistribution} We define the function space $$\mathcal L_{c}^{0} : = \left\{ f\in C_c(\cQ_1):  f \mbox{ is } 1\mbox{-Lipschitz }, \int_{\cQ_1} f d\mu_1 = 0\right\}$$ with respect to the Euclidean metric induced by local coordinates on $\cQ_1$ (see, for example~\cite[\S5]{AG} for a formal definition of this distance). We denote $\mathcal L_{c}$ the space without the mean $0$ condition.

\begin{theorem}
\label{theorem:birkhoff:chaika} Let $f \in \mathcal L_c^0$. Then for $\mu_{\mathcal{B}}$-almost every right angled billiard $q_{\Pi}$ we have
$$ \underset{T \to \infty}{\lim} \frac 1 {2\pi}
\int_0^{2\pi}f(g_Tr_{\theta}q_{\Pi})\, d\theta=\frac{1}{\mu_1(\cQ_1)}
\int_{\cQ_1}  f \, d\mu_1=0\,.$$
\end{theorem}
\medskip
\noindent

\subsection{Small arcs and strategy}\label{sec:arcs} The strategy to prove Theorem~\ref{theorem:birkhoff:chaika} is to break the integral over the circle into (exponentially) small arcs so that the limit converges as desired. Let $\mathcal{M}_{\epsilon} \subset \cQ_1$ be the $\epsilon$-thick part of the stratum, that is, the set of $q \in \cQ_1$ so that all saddle connections on $q$ have length at  least $\epsilon$.

\begin{prop}
\label{small chunks}
Let $f \in \mathcal L_c^0$ and $\delta>0$. Define $$S_{\delta} :=\left\{\theta \in [0,2\pi)\setminus\left(B\left(0,\delta\right)\cup B\left(\frac \pi 2, \delta\right) \cup B\left(\pi,\delta\right) \cup B\left(\frac {3\pi}2,\delta\right)\right)\right\}.$$ $S_{\delta}$ avoids neighborhoods of directions parallel to the sides. There exists an exponentially decaying function $v:\mathbb{R}^+\to [0,\pi)$ such that for any $\epsilon>0$
$$\lim_{N \rightarrow \infty} \frac 1 {2 v(\epsilon
  N)}\int_{-v(\epsilon N)}^{v(\epsilon
  N)}f(g_{\epsilon N}r_{\theta+\phi}q_{\Pi})d\phi=\frac{1}{\mu_1(\cQ_1)}\int_{\cQ_1}fd\mu_1  $$
for $\mu_{\mathcal{B}}$-almost every $q_{\Pi} \in \tilde{\cB_1} \cap \mathcal{M}_{\delta}$, Lebesgue almost every $\theta \in S_{\delta}$.
\end{prop}

\medskip
\noindent We prove Theorem~\ref{theorem:birkhoff:chaika} assuming Proposition~\ref{small chunks} in \S\ref{sec:proof:main}. To prove Proposition~\ref{small chunks}, we  estimate the $L^2$-norms of the functions $$F_N(q)=\frac 1 {2v(N)}\int_{-v(N)}^{v(N)}f(g_Nr_{\theta}q)d\theta,$$ on $\mathcal M_{\delta} \cap r_{\theta} \tilde{\mathcal B_1}$, the $\delta$-thick part of the rotated billiard subvariety. To estimate
$$\int_{r_\theta\tilde{\cB_1}\cap \mathcal{M}_\delta}(F_N(q))^2\,d\mu_{r_\theta\mathcal{B}} =
\int_{r_\theta\tilde{\cB_1}\cap \mathcal{M}_\delta}\frac 1 {4v(N)^2}\int_{-v(N)}^{v(N)}\int_{-v(N)}^{v(N)}(g_Nr_x q)f(g_Nr_y q)\,dx\,dy\,d\mu_{r_\theta\mathcal{B}},$$
we separate the domain of integration into two pieces: one where $x$ and $y$ are very close, and another when they are sufficiently separated. Heuristically, for $x$ and $y$ sufficiently separated, the translates $g_N r_xq$ and $g_N r_y q$ move away from each other exponentially in $N$, and thus become uncorrelated due to exponential mixing of the Teichm\"uller flow. This is made precise in Proposition~\ref{distributes}. For $x$ and $y$ sufficiently close, we estimate trivially by the measure of the set.%the integrand can be estimated using Proposition~\ref{distributes} which requires, in addition to exponential mixing, recurrence and contraction results, which we describe below.

\subsection{Exponential recurrence,  mixing, and contraction}\label{sec:recmix} To implement the above strategy, we need three crucial technical results on Teichm\"uller geodesic flow.
\subsubsection{Exponential Recurrence} Athreya~\cite{Ath thesis} showed that most (in an exponential sense) trajectories spend at least half their life in the thick part of a stratum. Precisely, let
$$G_L(\epsilon)=\left\{q \in \cQ_1:\left |\left\{
0\le
t<T:g_tq\in\mathcal{M}_{\epsilon}\right \}\right|>\frac T 2 \text{ for all }T>L\right\}$$ denote the set of $q \in \cQ_1$ so that the $g_t$-trajectory of $q$ eventually (after time $L$) spends at least half its life in $\mathcal M_{\epsilon}$.

\begin{theorem}[\protect{\cite[Theorem 2.3]{Ath thesis}}]
\label{theorem:often:compact}
For all small enough $\epsilon>0$, there exists $C,\xi>0$ such that for all $L>0$
$$ \mu_1(G_L(\epsilon))>(1-Ce^{-\xi L}) \mu_1(\cQ_1).$$
\end{theorem}

\subsubsection{Exponential Mixing} Avila-Resende~\cite{AR}, building on work of Avila-Gouezel-Yoccoz~\cite{AGY}, showed that the Teichm\"uller geodesic flow is \emph{exponentially mixing}. Let $d$ be the Riemannian distance on $\SL$ induced by the Killing form. For functions $f, g$ on $\cQ_1$, and $M \in SL(2, \R)$, define the $M$-correlation $$\mathcal C (f, g, M) = \left|\int_{\mathcal{Q}_1}f(Mq)g(q)d\mu_1-\int _{\mathcal{Q}_1}fd\mu_1\int_{\mathcal{Q}_1}gd\mu_1\right|$$

\begin{theorem}(\cite{AR}, \cite[Theorem 2.14]{AGY}) \label{thm:uniform bound} There exists constants $C,\lambda$ so that if
 $h_1,h_2$ are Lipshitz and compactly supported then there exists $C_K$ that depend only on the smallest systole of a surface in the compact support such that for any $M \in SL(2, \R)$,
$$\mathcal C (h_1, h_2, M)  \leq \\ C \left(C_K+\|h_1\|_{\infty}+\|h_1\|_{Lip}\right)\left(C_K+\|h_2\|_{\infty}+\|h_2\|_{Lip}\right)e^{-\lambda d(M,id)}.
$$
\end{theorem}

\subsubsection{Exponential Contraction} Eskin-Mirzakhani-Rafi~\cite{EMR}, following Forni~\cite{Forni}, proved an important result on the hyperbolicity of the Teichm\"uller flow. For a subset $A \subset \R^n$ we use $|A|$ to denote its Lebesgue measure.

\begin{theorem}[\protect{\cite[Lemma 8.3]{EMR}}]
\label{theorem:stable}
Given a fixed compact part $\mathcal{M}_{\frac{\epsilon}2}$ there exists $c, \tilde{C}>0$ such that if $q$ and $q'$ differ only along a stable manifold for $g_t$ (that is, if they share the same horizontal foliation) then
$$d_S(g_t q,g_t q')<\tilde{C}d_S(q,q')e^{-c\left|\left\{t>0:g_t q\in \mathcal{M}_{\frac{\epsilon}2}\right\}\right|}.$$ $d_S$ denotes Hodge distance along the stable manifold. See \cite[Section 8.2]{EMR}.
\end{theorem}

\subsection{Mixing on open sets} In this section we state and prove our first key lemma Proposition~\ref{prop:local}, which will play a crucial role in the proof of Proposition~\ref{small chunks}. Proposition~\ref{prop:local} uses exponential mixing (Theorem~\ref{thm:uniform bound}) in a crucial fashion.

 %A more direct consequence of exponential mixing is following effective mixing-type result for nice functions on translates of open sets with sufficiently regular boundary behavior.

Given an open set $U\subset \mathcal{M}_{\delta}$, let $\partial_{\epsilon} U$ denote the $\epsilon$-neighborhood of the boundary $\partial U$. We say $U$ is \emph{polynomially regular} with regularity polynomial $P$ if there is a polynomial $P$ so that $$\mu_1(\partial_{\epsilon} U) \le P(\epsilon).$$

\begin{prop}\label{prop:local}Let  $\delta>0$ and let $U \subset \mathcal{M}_{\delta}$ be polynomially regular. Let $f, h \in \mathcal L_c^0$. Then there exist $\hat{\lambda},D>0$ and $\ell_0<1$ so that for all $0<\ell<\ell_0$
$$\left|\int_Uf(g_{t}q)h(g_{(1+ \ell)t}q)d\mu_1(q)\right|\leq De^{-\hat{\lambda} \ell}.$$  Moreover, the constants only depend on $\delta$, $\|f\|_{\infty}$, $\|h\|_{\infty}$ and the regularity polynomial of $U$.
\end{prop}

The first step in the proof is the following effective equidistribution lemma for translates of polynomially regular sets.
\begin{lemma}\label{lem:uniform bound} Let $\delta>0$, $U \subset \mathcal{M}_{\delta}$ be polynomially regular and $f \in \mathcal L^c_0$. There exist $E, \lambda_f>0$ so that
$$\left|\frac 1 {\mu_1(U)}\int_{U}f(g_tq)d\mu_1(q)\right|
\leq Ee^{-\lambda_ft}.$$
The constants depend only on $\delta,f$ and the regularity polynomial of $U$.
\end{lemma}
\begin{proof} Let $U_{r}=\{q\in U: B(q,r)\subset U\}$ and $$h_{\epsilon}(q)=\chi_{U}(q)\left(1-\frac 1 {\epsilon}d(q,U_{\epsilon})\right).$$ Notice that $h_\epsilon$ is $\epsilon^{-1}$-Lipshitz. We will obtain the lemma by applying exponential mixing (Theorem \ref{thm:uniform bound}) to the functions $f$ and $h_{\epsilon}$.
$$ \left|\int_Uf(g_tq)d\mu_1\right|
\leq  \left|\int_{\mathcal{Q}_1}h_{\epsilon}(q)f(g_tq)d\mu_1\right|
+\|\chi_U-h_{\epsilon}\|_{\infty}\mu_1\left\{q:\chi_U(q) \neq h_{\epsilon}(q)\right\}\|f\|_{\infty}.$$
 By Theorem \ref{thm:uniform bound} we have that there exists $C_3$ (subsuming the various constants)
$$  \left|\int_{\mathcal{Q}_1}h_{\epsilon}(q)f(g_tq)d\mu_1\right|
\leq C_3\frac 1 {\epsilon}e^{-\lambda t}.$$
By our assumption on $U$ there exists $C,d$ (essentially the leading coefficient and degree of the regularity polynomial) so that $$\|\chi_U-h_{\epsilon}\|_{\infty}\mu_1\left\{q:\chi_U(q) \neq h_{\epsilon}(q)\right\}\|f\|_{\infty}\leq C\epsilon^d.$$
 Letting $\epsilon=e^{-\frac{\lambda}{4}t}$ we obtain $$\int_Uf(g_tq)d\mu_1\leq C_{\epsilon,f}e^{\frac{\lambda}{4}t}e^{-\lambda t}+Ce^{-\frac{\lambda d}{4}t}.$$
\noindent To complete the lemma, let $\lambda_f=\min\left\{\frac{\lambda}{2}, \frac {d\lambda}4\right\}$.
\end{proof}
\noindent Applying this to small balls, we obtain
\begin{cor}\label{lem:uniform bound2} Let $\epsilon>0$, $q_0\in \mathcal{M}_{\epsilon+e^{-k}}$ and $f \in \mathcal L_c^0$. Then there exists $C_{k,\epsilon,f}>0$ so that for all $k>0$,
$$\left|\frac 1 {\mu_1(B(q_0,e^{-k})))}\int_{B(q_0,e^{-k})}f(g_tq)d\mu_1(q)\right|
\leq C_{k,\epsilon,f}e^{-\lambda_ft}.$$
Moreover, $C_{k,\epsilon,f}$ can be chosen to be $C_{\epsilon,f}e^{kL_{\epsilon,f}}$ where $L_{\epsilon,f}$ depends only on $\epsilon,f$.
\end{cor}
\medskip\noindent  Applying this corollary to $f = \chi_U - \mu_1(U)$ for a  polynomially regular set, we obtain%\begin{cor}\label{cor:helpful}Let $\epsilon>0$. There exists $k,D'$ so that for any nice $f$ and $q\in \mathcal{M}_{\epsilon+e^{-4k}}$ we have
%$$\left|\frac1{\mu_1(B(q,e^{-4t}))}\int_{B(q,e^{-4t})}f(g_{kt}\omega)d\mu_1(\omega)\right|
%\leq D'e^{-\lambda t}.$$
%\end{cor}

%\marginpar{\tiny I don't understand Corollary~\ref{cor:helpful}. Is it simply putting $k=4t$ in Lemma~\ref{lem:uniform bound}? The notation is confusing. Also, where do we use it?}

\begin{cor}\label{cor:helpful open} Let $\delta>0$. Let $U \subset \mathcal{M}_{\delta}$ be polynomially regular. There exist $k_2,D_2,\lambda_2$ so that for any $q_0 \in \mathcal{M}_\delta$ we have, for all $r>0$
$$\left|\frac 1 {\mu_1(B(q_0,e^{-r}))} \int _{B(q_0,e^{-r})}\chi_U(g_{k_2r}q)d\mu_1(q)-\mu_1(U)\right|<D_2e^{-\lambda_2 r}.$$ $k_2$ can be chosen to be either positive or negative and the corollary holds for all large enough (in absolute value) $k_2$. The constants can be chosen to only depend on $\delta$ and the regularity polynomial of $U$.
\end{cor}
\begin{proof}Let $$H_{r,q_0}(q)=\left(e^{2r}-1\right)d\left(q,B(q_0,e^{-r})\right)\chi_{B\left(q_0,e^{-r}+e^{-2r}\right)}\left(q\right).$$

%\marginpar{\tiny From the definition of $H$ it seems like it is negative and in absolute value bounded by $e^{-2r}-1$. Do you want to make it positive by taking the constant out front to be $e^{2r}-1$ instead?}

\noindent By the regularity of $U$ and using period coordinates on the stratum we have that there exist $C'_1,d_1$ and $C_2',d_2$ so that
$$\left|\int_{B(q_0,e^{-r})}\chi_U(g_{kr}q) d\mu_1-\int_{\cQ_1}h_{\epsilon}(q)H_{r,q_0}(q)d\mu_1\right|<
C'_1\epsilon^{d_1}+C'_2e^{-2rd_2}.$$
 By Theorem \ref{thm:uniform bound}, since $$\|H_{r,q_0}\|_{Lip}\leq e^{2r}\mbox{ and }\|h_\epsilon\|_{Lip}\leq \frac 1 \epsilon$$ we have that there exists $C_3$ (subsuming the various constants) so that
  $$\left|\int_{\cQ_1}h_{\epsilon}(g_{kr}q)H_{q_0,r}(q)d\mu_1-\int H_{q,r}d\mu_1 \int h_{\epsilon}d\mu_1\right|<C_3\frac 1 {\epsilon}e^{2r}e^{-\lambda k }.$$
Combining these two estimates the corollary follows.
\end{proof}

\noindent Note that if $f \in \mathcal L_c^0$ then $f\circ g_t$ is $Ce^{3t}$-Lipshitz. We use this observation to prove Proposition \ref{prop:local} by splitting $U$ into balls of size $e^{-4t}$ where $f\circ g_t$ is basically constant. Then we apply Lemma \ref{lem:uniform bound2} to these balls. This gives us the required independence.
%\marginpar{\tiny Is it Lemma \ref{lem:uniform bound} or Lemma \ref{lem:uniform bound2}?}
\begin{proof}[Proof of Proposition \ref{prop:local}]

We want to estimate $$\left|\int_U f(g_tq)h(g_{t+s}q)d\mu_1(q)\right|=\left|\int_{\cQ_1} f(q)h(g_sq)\chi_U(g_{-t}q)d\mu_1(q)\right|$$

By doing an extra integration over small balls, we rewrite this as
$$\left|\int_{\cQ_1}\frac 1 {\mu_1(B(q_0,e^{- 4s}))}\int_{B(q_0,e^{-4 s})}f(q)h(g_sq)\chi_U(g_{-t}q)d\mu_1(q) d\mu_1(q_0)\right|$$
\noindent Since $f, h \in \mathcal L_{c}^0$ we can estimate their values in small balls by values at the center points, allowing us to bound from above the previous integral by
$$\left|\int_{\cQ_1}\frac 1 {\mu_1(B(q_0,e^{- 4s}))}\int_{B(q_0,e^{-4 s})} \left(f(q_0)h(g_sq_0)+O(e^{-4s})\|f\|_{Lip}\|h \circ g_s\|_{Lip}\right)\chi_U(g_{-t}q)d\mu_1(q) d\mu_1(q_0)\right|$$
Integrating the error term out, we derive the further estimate $$\left|\int_{\cQ_1}\frac 1 {\mu_1(B(q_0,e^{- 4s}))}f(q_0)h(g_sq_0)\int_{B(q_0,e^{-4 s})} \chi_{U}(g_{-t}q)d\mu_1(q) d\mu_1(q_0)\right|+O\left(e^{-\frac{3s}4}\right).$$

\noindent By Corollary \ref{cor:helpful open} if $s<\frac t {k_2}$ then for any $q_0$ in the support of $f$ (which is assumed to be compact)
$$\left|\frac 1 {\mu_1(B(q_0,e^{-4s}))}\int_{B(q_0,e^{-4 s})} \chi_{U}(g_{-t}q)d\mu_1(q)-\mu_1(U)\right|<D_2e^{-\lambda_2 s}.$$ So we obtain
\begin{multline}\left|\int_{\cQ_1}\frac 1 {\mu_1(B(q_0,e^{- 4s}))}f(q_0)h(g_sq_0)\int_{B(q_0,e^{-4 s})} \chi_{U}(g_{-t}q)d\mu_1(q) d\mu_1(q_0)\right|=
\\ \left|\mu(U)\int_{\cQ_1}f(q_0)h(g_sq_0)d\mu_1(q_0)\right|+O\left(D_2e^{-\lambda_2 s}\right)\|f\|_{\infty}\|h\|_{\infty}.
\end{multline}
Applying Theorem~\ref{thm:uniform bound} Proposition~\ref{prop:local} follows.
 \end{proof}

\medskip

\subsection{Correlation of translates} In this subsection, we state our other key lemma Proposition~\ref{distributes}, which estimates the correlation of $f \in \mathcal L_c^0$ with a translate of $f$ along a thickened $g_t$-translate of the billiard manifold $\tilde{\mathcal B_1}$. Fix $\epsilon>0$ so that Theorem~\ref{theorem:often:compact} holds.
\begin{proposition}\label{distributes}
If $f \in \mathcal L_c^0$, $\delta,a>0$ and $\theta \notin \{0,\frac \pi 2,\pi, 3 \frac {\pi}2\}$ then there exist constants
 $C_1',\lambda'$ and  $C_{2,\theta}<1$ such that for any $M \in \SL$ with $d(M,Id)\leq e^{tC_{2,\theta}}$ we have
 $$\left|\int_{-a}^a\int_{r_\theta\tilde{\mathcal{B}_1}\cap \mathcal{M}_\delta} f(Mg_{t+\ell} q)f(g_{t+\ell} q)\,
 d\mu_{r_{\theta}\tilde{\mathcal{B}_1}}(q)d\ell\right| < C_\theta e^{-\lambda' _\theta d(M,Id)}.$$  Here, and below, $\mu_{r_\theta \cB_1}$ denotes $(r_\theta) _*\mu_{\tilde{\cB_1}}$. $C_\theta, C_{2,\theta}$ and $\lambda'_\theta$ depend on $f$, $\delta$, $a$ and $\theta$. Moreover fixing $f,a,\delta$ the dependence on $\theta$ is continuous.
\end{proposition}

\medskip
\noindent The proof of Proposition~\ref{distributes} relies on all of the technical ingredients from the previous sections: exponential recurrence (Theorem~\ref{theorem:often:compact}) exponential mixing (Theorem~\ref{thm:uniform bound}) and exponential contraction (Theorem~\ref{theorem:stable}) as well as Proposition~\ref{prop:local}. Our key lemma is:
\begin{lemma}
\label{hm}
Let $\delta>0$ and $U\subset \mathcal{M}_{\delta}$ be polynomially regular. Let $f,h \in \mathcal L_c^0$.  Then there exist constants
 $\hat{C},C_2,\lambda$ such that for any $M \in \SL$ with $\|M\|<C_2t$ we have
 $$\left|\int_{U} f(g_tq)h(Mg_tq)\, d\mu_1(q)\right|<\hat{C}e^{-\lambda d(M,Id)}.$$ As in Proposition \ref{prop:local},
 $\hat{C}$ and $\lambda$ depend on $f,h,\delta$ and the regularity polynomial of $U$.
\end{lemma}
\begin{proof}
There exist $\theta, \phi \in S^1$, $s \in \mathbb{R}$ with $e^s=\|M\|$ so that $M=r_{\theta}g_s r_{\phi}$.
 \begin{multline*}\int_{U} f(g_tq)h(Mg_tq)\, d\mu_1(q)=\int_{Q}f(q)h(r_\theta g_s r_\phi q)\chi_U(g_{-t}q)d\mu_1(q)=\\
 \int_{\cQ_1}(f \circ r_{-\phi}(q))(h\circ r_{\theta}(g_sq)) \chi_U(g_{-t}r_{-\phi}q)d\mu_1(q).
 \end{multline*}
 We now follow the approach of the proof of Proposition \ref{prop:local} and for convenience introduce $h_{\theta}=h\circ r_\theta, f_{\phi}=f\circ r_{-\phi} \in \mathcal L_c^0$.
 \begin{multline*}
\left|\int_{\cQ_1}\frac 1 {\mu_1(B(q,e^{- 4s}))}\int_{B(q,e^{-4 s})}f_{\phi}(\omega)h_{\theta}(g_s \omega)\chi_U(g_{-t}r_{-\phi}\omega)d\mu_1(\omega) d\mu_1(q)\right|=\\
\left|\int_{\cQ_1}\frac 1 {\mu_1(B(q,e^{- 4s}))}f_1(q)h_1(g_sq)\int_{B(q,e^{-4 s})} \chi_{U}(g_{-t}r_{-\phi}\omega)d\mu_1(\omega) d\mu_1(q)\right|+O\left(e^{-\frac{3s}4}\right).
 %\int_{r_{-\phi}\cQ_1}f_2(q)h_2(g_s q)d\mu_1(q)
 \end{multline*}
Now
\begin{multline*}\left|\frac 1 {\mu_1(B(q,e^{-4s}))}\int_{B(q,e^{-4 s})} \chi_{U}(g_{-t}r_{-\phi}\omega)d\mu_1(\omega)-\mu_1(U)\right|=\\
\left|\frac 1 {\mu_1(B(q,e^{-4s}))}\int_{r_{-\phi}B(q,e^{-4 s})} \chi_{U}(g_{-t}\omega)d\mu_1(\omega)-\mu_1(U)\right|<D_2e^{-\lambda_2 s}
\end{multline*}
since $r_{-\phi}B(q,e^{-4s})$ has the same regularity polynomial as $B(q,e^{-4s})$. This implies the lemma.
 \end{proof}

\medskip

\begin{proof}[Proof of Proposition \ref{distributes}]
%\marginpar{Introduced the integral over time to get the central direction. Numerous additions of this integral, the constant $a$...}

  Let $\theta \notin \{0,\frac \pi 2 , \pi, \frac {3 \pi}2\}$. Intersect $r_\theta\tilde{\mathcal B}_1$ with $\mathcal{M}_{\epsilon }$ and flow it by a small interval $\{g_{\ell}, |\ell| <a \}$, and consider $$\bigcup_{|\ell|<a}g_\ell r_\theta \tilde{\mathcal{B}}_1\cap \mathcal{M}_\epsilon.$$ Thicken it by  $c'>0$
  along the stable manifold for $g_t$, and call the resulting set $V$. We pick $a$ and $c'$ small enough so that the
 intersection of $V$
with the support of $f$ has a local product structure (as stable $\times$ unstable $\times$ flow) with respect to the Teichm\"uller flow $g_t$. By the continuity of the trigonometric functions in the proof of Proposition \ref{prop:map:onto} $a$ and $c'$ can be chosen to depend continuously on $\theta$. Let $\Phi$ denote the local projection from $V$ to the $a$-thickened and $\theta$-rotated billiard subvariety$$\bigcup_{|\ell|<a}g_\ell r_\theta \tilde{\mathcal{B}_1}.$$ By Proposition~\ref{prop:map:onto},
$$\mu_1(V)>0.$$ By Corollary \ref{hm}, if $\|M\|_{op}<\hat{C}_{2,\theta}$ we have that
$$\left|\int_{V\cap G_t}f(Mg_t q)f(g_t q)\, d\mu_1\right|\leq \hat{C}_\theta e^{-\lambda_\theta d(M,Id)}+\mu_1(G_t^c).$$ By the continuity of $r_{\theta}$ and the construction of $V$ the polynomial that bounds the decay of an $\epsilon$ neighborhood of the boundary of $V$ can be chosen to depend continuously on $\theta$. So  $\hat{C}_\theta$ and $\lambda_\theta$ depend continuously on $\theta$.
 By exponential recurrence (Theorem~\ref{theorem:often:compact}) $\mu_1(G_t^c)$ decays exponentially in $t$ and so there exists $C''_\theta,\lambda''_\theta$ such that
$$\left|\int_{V\cap G_t}f(Mg_t q)f(g_t q)\,d\mu_1\right|\leq C''_\theta e^{-\lambda''_\theta d(M,Id)}.$$
Now for each $q \in V$ there exists $q' \in r_\theta \tilde{\mathcal{B}_1}$ on
the same stable manifold which is distance at most $c'$ away. It
follows from the exponential contraction of $g_t$ (Theorem~\ref{theorem:stable}) that for $t$ large enough and $q' \in G_t$ then
$$d(g_t q,g_t q')<\tilde{C}c'e^{-\frac c 2t}.$$ By our assumption that $f \in \mathcal L_c^0$ it follows that

$$\left|f(g_t q)-f(g_t q'\right)|<\tilde{C}c'e^{-\frac c2t}$$ and $$|f(Mg_t q)-f(Mg_t q')|\leq \|M\|_{op}\tilde{C}c'e^{-\frac c 2t}, $$ where $\|\cdot\|_{op}$ denotes the operator norm of $\SL$ acting linearly on $\R^2$.
 By our assumption on $V$ and the fact that $f$ is 1-Lipschitz it follows that
\begin{multline*}
\left|\int_{V\cap  G_t}f(Mg_t q)f(g_t q) d\mu_1(q)\right|\geq\\ \left|\int_{V\cap G_t}f\left(Mg_t\Phi(q)+e_q\right)\left(f(g_t\Phi(q)+e'_q\right)d\mu_1(q)\right|
\geq\\
\left| \int_{[0,c']^k}\int_{-a}^a\int_{r_\theta\tilde{\mathcal{B}_1}\cap \mathcal{M}_{\frac \epsilon 2} }\left(f(Mg_{\ell+t} q')+e_{q',\ell,s}\right)\left(f(g_{t+\ell} q')+e'_{q',\ell,s}\right) d\mu_{r_\theta\tilde{\mathcal{B}}}(q')d\ell d\lambda^k(s)\right| +\zeta e^{-\xi t}\,,
\end{multline*}
  where $$|e'_q|,|e'_{q',s}|< \tilde{C}c'e^{-\frac c 2t}\mbox{ and }|e_q|,|e_{q',s}|< \|M\|_{op}\tilde{C}c'e^{-\frac c 2t}.$$ In the last inequality of the integral estimate, $\zeta e^{-\xi t}$ comes from the exponential recurrence result Theorem \ref{theorem:often:compact}. This establishes the Proposition.
\end{proof}

\subsection{Proof of Theorem~\ref{theorem:birkhoff:chaika}}\label{sec:proof:main} We prove our main Theorem~\ref{theorem:birkhoff:chaika} assuming our key tool Proposition~\ref{small chunks}. By Proposition~\ref{small chunks}, for every $\delta>0$ there exists a $T_0$
 such that for any
 $T>T_0$, $T\in \epsilon \mathbb{N}$ and set $S$ with $$\mu_{\mathcal{B}} (S) \geq \mu_{\mathcal{B}}(\mathcal{B}_1 \cap \mathcal{M}_\delta)-\delta$$ we have that for each $q \in S$ a subset $G_{q}$ of $S_{\delta}$ with
 $$\lambda(G_q)\geq 2\pi-8\delta-\delta$$ such that for each $\theta \in G_{q}$ we have
 $$\left| \frac 1 {2v(T)}\int_{-v(T)}^{v(T)}f(g_Tr_{\theta+\phi}q)d\phi \right |<\delta.$$
It follows that for $q \in S$ and $T>T_0$ we have
$$\left|\frac 1 {2\pi}\int_{0}^{2\pi}f(g_Tr_{\phi}q)\, d\phi \right |=
\left|\frac 1 {2\pi}\int_0^{2\pi}\frac 1 {2v(T)}\int_{-v(T)}^{v(T)} f(g_Tr_{\theta+\phi}q)d\phi d\theta\right|$$
\noindent We break this integral into two pieces, over $G_q$ and its complement. We have
$$\left|\frac 1 {2\pi}\int_{G_q}\frac 1 {2v(T)}\int_{-v(T)}^{v(T)} f(g_Tr_{\theta+\phi}q)d\phi d\theta\right| \le \delta $$ and $$\left|\int_{G_q^c}\frac 1 {2v(T)}\int_{-v(T)}^{v(T)} f(g_Tr_{\theta+\phi}q)d\phi d\theta\right| \le 9 \delta\|f\|_{\infty}$$
 So we have $$\left|\frac 1 {2\pi}\int_{0}^{2\pi}f(g_Tr_{\phi}q)\, d\phi \right | \le \delta+9\delta \|f\|_{\infty}
\leq 10\delta\max\left(1, \|f\|_{\infty}\right).$$
Because $f$ is $1$-Lipschitz we have that if  $$\underset{T \to \infty}{\limsup}\, \left|\frac 1 {2\pi}\int_0^{2\pi}f(g_Tr_\phi q)d\phi\right|\leq
\epsilon+ \underset{N  \in \mathbb{N}\epsilon}{\limsup}\, \left|\frac 1 {2\pi}\int_0^{2\pi}f(g_Nr_\phi q)d\phi\right|.$$ Since $\delta$ and $\epsilon$ are arbitrary Theorem~\ref{theorem:birkhoff:chaika} follows.

 \qed\medskip
\subsubsection{Proof of Proposition~\ref{small chunks}}
% Proposition \ref{small chunks} relies on estimating the $L^2$ norm of these integrals and uses Chebyshev's inequality:
%\begin{lemma}[Chebyshev] Let $(\Omega, \nu)$ be a probability space, and let $f:\Omega\to \mathbb{R}$ be a random variable (measurable function) with second moment bounded by $C$, that is
%$$\int_{\Omega}f(\omega)^2d\nu\leq C.$$ Then
%$$\nu(\{\omega:|f(\omega)|>sC\})\leq \frac {1}{s^2C}.$$
%\end{lemma}
%\medskip
%\noindent

To prove Proposition \ref{small chunks}, we estimate the $L^2$-norms of the integrals $$F_N(q)=\frac 1 {2v(N)}\int_{-v(N)}^{v(N)}f(g_Nr_{\theta}q)d\theta,$$ where we define $v(N)$ below. We have
$$\int_{r_\theta\tilde{\cB_1}\cap \mathcal{M}_\delta}(F_N(q))^2\,d\mu_{r_\theta\tilde{\mathcal{B}}} =
\int_{r_\theta\tilde{\cB_1}\cap \mathcal{M}_\delta}\frac 1 {4v(N)^2}\int_{-v(N)}^{v(N)}\int_{-v(N)}^{v(N)}(g_Nr_x q)f(g_Nr_y q)\,dx\,dy\,d\mu_{r_\theta\mathcal{B}}.$$ Changing the order of integration, we obtain
$$\int_{-v(N)}^{v(N)}\int_{g_Nr_x(r_\theta\tilde{\cB_1}\cap \mathcal{M}_\delta)}\frac 1 {4v(N)^2} \int_{-v(N)}^{v(N)}f(g_Nr_{x-y}g_{-N}q)f(q)\,
dy\, d((g_N)_*\mu_{r_{x+\theta}\tilde{\mathcal{B}}})\,dx$$
To ease notation, $\mu_{\mathcal{B}}$ will denote $(g_N)_*\mu_{r_{x+\theta}\tilde{\cB}}$ for the remainder of this proof. We note that
\begin{displaymath}
r_{\theta}
  =(h_{-\tan \theta})^{\tau}
  g_{\log\cos \theta}
  h_{\tan \theta}\,.
\end{displaymath}
where $$h_s=\left(\begin{array}{cc}1&s\\0&1\end{array}\right),
h_s^{\tau}=\left(\begin{array}{cc}1&0\\s&1\end{array}\right),$$
\noindent so the previous expression is
$$\int_{-v(N)}^{v(N)}\int_{g_Nr_{x+\theta}(\tilde{\cB_1}\cap \mathcal{M}_\delta)}\frac 1 {4v(N)^2}\int_{-v(N)}^{v(N)}f(h^{\tau}_{e^{-2N}\tan(x-y)}g_{\log(\cos(x-y))}h_{e^{2N}\tan(x-y)}q)f(q)\, dy\, d\mu_{\mathcal{B}}\,dx.$$

\noindent Now consider $C_{2,\theta}$ from Proposition \ref{distributes}. Recall that we defined $$S_{\delta} = [0,2\pi)\setminus\left(B\left(0,\frac{\delta}2\right)\cup B\left(\frac \pi 2, \frac{\delta}2\right) \cup B\left(\pi,\frac{\delta} 2\right) \cup B\left(\frac {3\pi}2,\frac \delta 2\right)\right),$$ and
$$C_2=\min_{\theta \in S_{\delta}}C_{2,\theta}.$$ Observe that this is defined and greater than zero because $C_{2,\theta}$ is continuous in $\theta$ and $S_{\delta}$ is compact. Let $$v(N)=\min \left\{e^{\left(-2+\frac{C_{2}}3\right)N},\frac \delta2\right\}.$$ Then $$\left|h^{\tau}_{e^{-2N}\tan(x-y)}g_{\log(\cos(x-y))}\right|<e^{-N} \mbox{ for } x, y \in [-v(N), v(N)].$$ Since $f$ is $1$-Lipschitz we can dominate the integral by

$$ e^{-N} + \int_{-v(N)}^{v(N)}\int_{g_Nr_{x+\theta}(\tilde{\cB_1}\cap \mathcal{M}_\delta)}\frac 1 {4v(N)^2}\int_{-v(N)}^{v(N)}f(h_{e^{2N}\tan(x-y)}q)f(q)\, dy\, d\mu_{\mathcal{B}}\,dx$$

\noindent We break the domain of integration into pieces, and we estimate the integral on each separately. Let $$\Delta_N : = \left\{x, y \in \left[-v(N), v(N)\right], |x-y| \le e^{\left(-2+\frac{C_2}6\right)N}\right\}$$ be a small neighborhood of the diagonal in $[-v(N), v(N)]^2$. The first piece $P_1$ is when $(x,y) \notin \Delta_N$:

   $$P_1 : = \left\{(x,y, q): (x,y) \notin \Delta_N, q \in g_N\left(r_{x+\theta}\cB_1\cap \mathcal{M}_\delta\right)\right\} $$
   %\marginpar{Is this a (spiritually) correct invocation of Prop~\ref{prop:local}?}

 \noindent We add an integral over time so we can estimate the integral over $P_1$ using Proposition~\ref{distributes}, yielding %\marginpar{Changes to introduce $a$ start here and continue to the end}
   \begin{multline*}\frac 1 {4v(N)^2}
   \left(
 \int_{-a}^a \int_{ P_1} f(h_{e^{2N}\tan(x-y)}g_\ell q)f(q)\, dy\, d\mu_{\mathcal{B}}\,dxd\ell \right)=\\ \frac 1 {4v(N)^2}
 \int_{(x,y) \notin \Delta_N}\int_{-a}^a \int_{r_\theta\tilde{\mathcal{B}_1}\cap \mathcal{M}_\delta}f(h_{e^{2N}\tan(x-y)}g_{N+\ell}q)f(g_{N+\ell}q) d\mu_{\mathcal{B}}\,d\ell\, dx\, dy \\
  \le  \frac 1 {4v(N)^2}
 \int_{(x,y) \notin \Delta_N}C_1e^{-\lambda' \frac {C_2}{6}N}dx\, dy\le C_1e^{-\lambda' \frac {C_2}{6}N}.
 \end{multline*}
 To see this is justified, first observe that the domain of integration is appropriate because for all large enough $N$, $\theta+x \in S_{\delta}$. Second, the size the matrices is appropriate  because by our choice of $v(N)$ we have $$e^{2N}2v(N)+1<e^{C_2N}$$ for all $N$ sufficiently large and
and so $$\|h_{e^{2N}\tan(x-y)}\|<2e^{C_2N}\mbox{ for all }x,y \in [-v(N),v(N)],$$ since $\tan$ is $2$-Lipschitz on $[-\frac\pi 4,\frac \pi 4]$. Moreover, since $$|\tan(x-y)|\geq \frac 1 2|x-y| \mbox{ and }x,y \notin \Delta_N$$ we have $$\|h_{e^{2N}\tan(x-y)}\|>\frac 1 2 e^{\frac {C_2} 6N}.$$

 \noindent
 Our second piece, $P_2$, is when $x$ and $y$ are close:
 $$P_2: = \left\{(x,y, q): (x,y) \in \Delta_N, q \in g_N\left(r_{x+\theta}\tilde{\cB_1}\cap \mathcal{M}_\delta\right)\right\} $$
 We have, via naive (measure and $\|\cdot\|_\infty$) estimates %\marginpar{\tiny Changed justification to be naive estimates}
 \begin{eqnarray*} \frac 1 {4v(N)^2}
   \left(\int_{P_2}f(h_{e^{2N}\tan(x-y)}q)f(q)\, d\mu_{\mathcal{B}}\, dy\,dx\right) &\leq&   \frac{\|f\|_{\infty}^2 |\Delta_N|}{4v(N)^2}\\ &\leq& e^{-\frac{C_2}{6}N}\|f\|_{\infty}^2.\end{eqnarray*}
%   To apply Proposition~\ref{distributes} to estimate the integral over $P_2$, notice the domain of integration is appropriate because for all large enough $N$  $$\theta+x \in S_{\delta}$$ for all $x\in [-v(N),v(N)]$.  The size the the matrices is appropriate  because by our choice of $v(N)$ we have $$e^{2N}2v(N)+1<e^{C_2N}$$ for all $N$ sufficiently large and
%and so $$\|h_{e^{2N}\tan(x-y)}\|<2e^{C_2N}\mbox{ for all }x,y \in [-v(N),v(N)],$$ since $\tan$ is $2$-Lipschitz on $[-\frac\pi 4,\frac \pi 4]$.
 Combining the estimates on the integrals over $P_1$ and $P_2$ we obtain
  $$ \int_{r_\theta\tilde{\cB_1}\cap \mathcal{M}_\delta}(F_N(q))^2\,d\mu_{r_\theta\mathcal{B}} \le
 C_1'e^{-\lambda' \frac {C_2}{6}N}+e^{-\frac{C_2}{6}N}\|f\|_{\infty}^2+e^{-N}.$$

\noindent
So there exists $\delta>0$ such that for all large enough $N$.
%\begin{eqnarray*}
$$\int_{-a}^a\int_{r_\theta\tilde{\cB_1}\cap \mathcal{M}_\delta} \left(F_N(g_\ell q)\right)^2d\mu_{r_\theta\tilde{\mathcal B}}(q)d\ell \leq %&=& \int_{r_\theta\tilde{\mathcal{B}\cap \mathcal{M}_\delta}_1}\left(\frac 1 {2v(N)} \int_{-v(N)}^{v(N)}f(g_Nr_{\theta}q)\,d\theta\right)^2d\mu_{r_\theta\mathcal B}(q)\\ &\leq&
 e^{-\delta N}.$$
%\end{eqnarray*}
Let $$m_N(q, a) = \min_{\ell \in [-a,a]}  \left|F_N(g_\ell q)\right|.$$ We have, for any $\eta >0$, that %\marginpar{Changes from here to the end. \textcolor{blue}{ Introduced function $m_N$}}
$$\mu_{\mathcal B} \left\{ q\in r_\theta\tilde{\cB_1}\cap \mathcal{M}_\delta: m_N(q, a) > \eta \right\} \le\frac{1}{2a\eta^2} e^{-\delta N}.$$ Indeed,
 \begin{multline*}
 2a\eta^2\mu_{\mathcal B} \left\{ q\in r_\theta\tilde{\cB_1}\cap \mathcal{M}_\delta:
 m_N(q, a)> \eta \right\} \leq \\
  \int_{-a}^a\int_{r_\theta\tilde{\cB_1}\cap \mathcal{M}_\delta} \left(F_N(g_\ell q)\right)^2d\mu_{r_\theta\mathcal B}(q)d\ell.
  \end{multline*}
 Since for any $\eta >0$, we have that  $$\sum_{N=1}^{\infty} \frac{1}{\eta^2} e^{-\delta N} < \infty,$$ the easy half of the Borel-Cantelli lemma implies that for $\mu_{\mathcal B}$-almost every $q \in \tilde{\mathcal{B}_1} \cap \mathcal{M}_{\delta}$, the set $$\{N \geq 1: m_N(q, a) > \eta\}$$ is finite. Because $F$ is 1-Lipshitz, for any such $q$ we have that there exists $N_0$ so that $$|F_N(q)|<2a+\eta \mbox{ for all }N>N_0.$$ Since $\eta$ and $a$ are arbitrary and our estimates hold for all $\theta \in S_{\delta}$ the proposition follows.

 \qed

\end{document}